\documentclass[12pt]{article}

\RequirePackage[OT1]{fontenc}
\RequirePackage{amsthm,amsmath}
\usepackage{amssymb,amstext,amsbsy,amscd,mathrsfs}
\RequirePackage[numbers]{natbib}
\RequirePackage{hyperref}
\usepackage{natbib}
\RequirePackage[french,english]{babel}
\RequirePackage{url}
\RequirePackage{tikz}
\usepackage{graphicx}
\usepackage{appendix}
\usepackage{enumerate}

\usepackage{enumitem}
\usepackage{mathrsfs}
\usepackage[np,autolanguage]{numprint} 
\newcommand{\field}[1]{\mathbb{#1}}
\newcommand{\R}{\field{R}}
\newcommand{\N}{\field{N}}

\newcommand{\E}{\field{E}}

\newcommand{\pl}[1][]{\mathcal{P}_{#1}}

\newcommand{\rl}[1][]{\mathcal{R}_{#1}}
\newcommand{\rtl}[1][]{\tilde{\mathcal{R}}_{#1}}
\newcommand{\po}{\mathcal{P}_{0}}

\newcommand{\pj}{\mathcal{P}_{j}}
\newcommand{\nul}{\nu} 
\newcommand{\numes}{\boldsymbol{\nu}} 
\newcommand{\hnuml}[1][m]{\hat{\nu}_{#1}} 
\newcommand{\numl}[1][m]{\nu_{#1}}
\newcommand{\mul}{\mu}
\newcommand{\rhol}{\rho}
\newcommand{\xil}{\xi}

\newcommand{\Au}{\tilde{A}_{x ,y}}
\newcommand{\A}{A_{x ,y}}
\newcommand{\Bu}{\tilde{B}_{x ,y}}
\newcommand{\B}{B_{x ,y}}

\newcommand\nbOne{{\mathchoice {\rm 1\mskip-4mu l} {\rm 1\mskip-4mu l}
		{\rm 1\mskip-4.5mu l} {\rm 1\mskip-5mu l}}}
\newcommand\units[1]{\nbOne_{\{#1\}}}

\newcommand{\norm}[1]{\left\Vert #1 \right\Vert}
\newcommand{\Vl}{\mathbf{V}_{\lambda}}
\newcommand{\Vb}{\mathbb{V}_{b}}
\newcommand{\Cl}{C_{\lambda}}
\newcommand{\cl}{c_{\lambda}}
\newcommand{\Gl}{G_{\lambda}}
\newcommand{\rond}[1]{\mathscr{#1}}
\newcommand{\EE}[2][z_0]{\mathbb{E}_{#1}\left(#2\right)}
\newcommand{\Ec}[2]{\mathbb{E}\left(\left. #1\right\vert #2\right)}
\newcommand{\El}[2]{\mathbb{E}^{#1}\left(#2\right)}
\newcommand{\PP}[2][]{\mathbb{P}_{#1}\left(#2\right)}
\newcommand{\Pl}[2]{\mathbb{P}^{#1}\left(#2\right)}
\newcommand{\Pc}[2]{\mathbb{P}\left(\left. #1 \right\vert #2 \right)} 
\newcommand{\ep}{\varepsilon}
\newcommand{\Var}[2][z_0]{\operatorname{Var}_{#1}\left(#2\right)}
\newcommand{\Cov}[2][]{\operatorname{Cov}_{#1}\left(#2\right)}
\newcommand{\OO}[1]{\mathrm{O}\left(#1\right)}

\newcommand{\ft}{s}
\newcommand{\Dg}{\mathbf{D}}
\newcommand{\Dge}{\hat{\mathbf{D}}_n}
\newcommand{\mg}{\mathbf{m}}
\newcommand{\Mg}{\mathbf{M}}
\newcommand{\Lg}{\mathbf{L}}
\newcommand{\llg}{\mathbf{l}} 
\newcommand{\Besov}{\ensuremath{\mathbf{\mathit{B}}}}
\newcommand{\holder}[1][\alpha]{\ensuremath{H^{#1}}}
\newcommand{\cons}{a}
\newcommand{\nb}{p_n}

\newcommand{\tq}[1][g]{\mathcal{Q}_{#1}}
\newcommand{\Fcb}{\mathcal{E}(\cf,b,\alpha)}

\newcommand{\cf}{\mathfrak{s}} 

\newcommand{\n}{\mathbf{n}}
\newcommand{\rg}{\mathbf{r}}

\theoremstyle{plain}
\newtheorem{theorem}{Theorem} 
\newtheorem{definition}[theorem]{Definition}
\newtheorem{corollary}[theorem]{Corollary}
\newtheorem{lemma}[theorem]{Lemma}
\newtheorem{prop}[theorem]{Proposition}
\newtheorem{assumption}{Assumption A\!\!}{\bf}{\rm}

\newtheorem*{assumptiond}{Assumption}{\bf}{\rm}
\newtheorem*{result}{Result}
\theoremstyle{remark}
\newtheorem*{remark}{Remark}

\newcounter{numero} 

\begin{document}

\title{Nonparametric estimation of jump rates for a specific class of piecewise deterministic Markov processes}

\author{ N. Krell\thanks{Universit\'e de Rennes 1,
 Institut de Recherche math\'ematique de Rennes,
CNRS-UMR 6625, Campus de Beaulieu.
 B\^atiment 22, 35042 Rennes Cedex, France. {\tt email}: nathalie.krell@univ-rennes1.fr}
\\ E. Schmisser\thanks{Laboratoire Paul Painlevé
Université des Sciences et Technologies de Lille,
Bureau 314, Bâtiment M3, Cité Scientifique,
59 655 Villeneuve d'Ascq Cedex {\tt email}: emeline.schmisser@math.univ-lille1.fr}}

\maketitle

\begin{abstract}
	
	In this paper, we consider a  unidimensional piecewise deterministic Markov process (PDMP), with homogeneous jump rate $\lambda(x)$ . This process is observed continuously, so the flow $\phi$ is known.  To estimate nonparametrically the jump rate, we first construct an adaptive estimator of the stationary density, then we derive a quotient estimator $\hat{\lambda}_n$ of $\lambda$. Under some ergodicity conditions, we  bound the risk of these estimators (and give a uniform bound on a small class of functions), and prove that the estimator of the jump rate is nearly minimax (up to a $\ln^2(n)$ factor). The simulations illustrate our theoretical results.
	
\end{abstract}

\noindent {\it Keywords:} Piecewise deterministic Markov processes, model selection, nonparametric estimation

\noindent {\it Mathematical Subject Classification:} 62G05, 62G07, 62M05, 60J25

\section{Introduction}

Piecewise deterministic Markov processes are a large class of continuous-time stochastic models first introduced by \citet{Dav93}. They are used to model deterministic phenomenons in which randomness appears as point events. They are not diffusions,  which adds complexity to their study.  This family of stochastic processes is well adapted to model various problems in biology (see for instance \citet{cloez_genadot_2017}, \citet{MR3295786}), neuroscience (\citet{hopfnerbrodda2006}, \citet{renault_thieullen_trelat_2017}), physics (\citet{blanchard_95}), reliability (\citet{saporta_dufour_zhang}),  optimal consumption and exploration (\citet{farid_davis_99}), risk insurance, seismology,\ldots. See also the references in the survey \citet{azais_bardet_2014}.

In this article, we consider 
a filtered  piecewise deterministic Markov process (PDMP) $(X_t)_{t\geq 0}$ taking values in $\R^+$, with flow $\phi$, transition measure $Q(x,dy)$ and homogeneous jump rate $\lambda(x)$. Starting from initial value $x_0$, the process follows the flow $\phi$ until the first jump time $T_1$ which occurs  spontaneously in a Poisson-like fashion with rate $\lambda(\phi(x,t))$. The post-jump location of the process at time $T_1$ is governed by the transition distribution $Q(\phi(x_0,T_1),dy)$ and the motion restarts from this new point as before.

To fix the ideas, let us consider two major examples of unidimensional PDMP.

The TCP  (transmission control protocol) (see \citet{DGR02},  \citet{GRZ04} for instance) is one of the main data transmission protocol in Internet. The maximum number of packets that can be sent at time $t_k$ in a round is a random variable $X_{t_k}$. If the transmission is successful, then the maximum number of packets is increased by one:  $X_{t_{k+1}}=X_{t_k}+1$. If the transmission fails, then we set $X_{t_{k+1}}=\kappa X_{t_k}$ with  $\kappa\in(0,1)$. A correct scaling of this process leads to a piecewise deterministic Markov process  $(X_t)$ with flow $\phi(x,t)=x+ct$ and deterministic transition measure $Q(x,{y})=\nbOne_{\{y=\kappa x\}}$. 
 This process grows linearly (by construction) and the constant $\kappa$ can be configured in the server implementation (so it is also known), but the moment when the transmission fails is of course unknown. In the literature it is usually supposed that the jump rate satisfies $\lambda (x)=x$, but with this work we can check whether it is a realistic assumption or not.

Another example of PDMP is the size of a marked bacteria (see \citet{doumic_hoffmann_2015, hof}, \citet{laurencot_perthame_2009}). We randomly choose a bacteria, and follow its growth, until it divides in two. Then we randomly choose one of its daughters, and so on.  Between the jumps, the bacteria grows exponentially: $\phi(x,t)=xe^{ct}$. The size of the bacteria after the division is random, as the bacteria does not divide itself in two equal parts.

The process $(X_t)$ is observed continuously without errors\:   (so the flow $\phi$ is known); it is assumed to be ergodic, with fast convergence toward the stationary measure, and exponentially $\beta$-mixing. We denote by $(T_1,\ldots,T_n)$ the jump times   and consider the Markov chain  $(Z_0=x_0,(Y_k=X_{T_k^-},Z_k=X_{T_k})_{k\in\N})$. Our aim is to construct a non-parametric adaptive estimator of the jump rate $\lambda$ on a compact interval.

There exist few results concerning PDMP's estimation. \citet{azais_dufour_gegout_2014} and  \citet{azais2016} consider a more general model, for a multidimensional PDMP. They construct a quotient of kernel estimators, which estimate the compound function $\lambda(\phi(x,t))$.  Their estimator is consistent (\cite{azais_dufour_gegout_2014}), asymptotically normal, and its pointwise rate of convergence depends on the bandwidth of the kernel (see \cite{azais2016}). They explain how to construct an adaptive estimator, but do not bound its risk. 

\citet{doumic_hoffmann_2015} and  \citet{hokrlo} also consider multi-dimensional PDMPs but for very specific biological models. 
 
\citet{takayuki2013} and \citet{kre2} both consider unidimensional PDMP, and provide estimators of $\lambda(x)$. 
\cite{takayuki2013}  constructs an estimator of $\lambda(x)$ thanks to a Rice formula, by estimating local times. He proves the consistency of his estimators.  \cite{kre2} considers a deterministic transition measure (so $Y_k$ is a function of $Z_k$). Her estimator of $\lambda$ is  a quotient of a kernel estimator of the stationary density of  $Z_k$ and an empirical estimator $\hat{\mathbf{D}}_n$ of another function $\Dg$ with the parametric rate of convergence $n^{1/2}$. This nonparametric estimator is asymptotically normal, and bounds for the pointwise risk are provided. 
In a very recent article, \citet{azais_genadot_2018} construct a nonparametric estimator of $\lambda(x)$ for a multidimensional PDMP and prove its consistency.

This article is an extension of the work of \cite{kre2}. We consider a wider class of models (in particular, the transition measure $Q$ does not need to be deterministic any more). We bound the $L^2$ risk of the adaptive estimator, whereas \cite{kre2} only considers the pointwise risk of the nonparametric estimator with fixed bandwidth $h$. We also prove that our estimator is minimax (up to a $\ln^2(n)$ factor). 

 For this purpose,  in analogy with \cite{kre2}, we use the equality
\[\lambda(x)=\frac{\nul(x)}{\Dg(x)}\]
where $\nul$ is the stationary density of pre-jump locations $Y_k$ (see Assumption A2 for the existence of this stationary density) and  $\Dg$ a  function  defined in equation (\ref{def_D}). We get an estimator   $\Dge(x)$, which converges with rate $n^{1/2}$. To estimate the density function $\nul$, we use a projection method. We obtain a series of estimators $(\hnuml[0],\hnuml[1],\ldots,\hnuml,\ldots)$ of $\nul$.  
Then we choose the "best" estimator by a penalization method, in the same way as \citet{barronbirge1999}, and give an oracle inequality for the adaptive estimator $\hnuml[\hat{m}]$. The constant in the penalty term is intractable, but  can be estimated thanks to a slope heuristic. Finally, we construct a quotient estimator  of $\lambda$, $\hat{\lambda}=\hnuml[\hat{m}]/\Dge $, and  bound its $L^2$-risk. 
In Section \ref{section_PDMP}, we specify the model and its assumptions. The main results are stated in Section \ref{section_estimation_lambda}.  Proofs are gathered in Section \ref{section_proofs} and in Appendix A  for the technical results.  In Appendix B, some   simulations for the TCP protocol and the bacterial growth are provided, with various functions $\lambda$. The outcomes are consistent with the theoretical results.

\section{PDMP}\label{section_PDMP}

A piecewise deterministic Markov process (PDMP) is defined by its local characteristics, namely, the jump rate $\lambda$, the flow $\phi$ 
 and the transition measure $Q$ according to which the location of the process is chosen after the jump. In this article, we consider  a unidimensional PDMP $\{X(t)\}_{t\geq 0}$. More precisely, 
 
 \begin{assumption}\label{hypopdmp}$\;$
\begin{enumerate} 

\item The flow $\phi:\R^+ \times\R^+ \mapsto \R^+$ is a one-parameter group of homeomorphisms:
$\phi$ is $\mathcal{C}^1$, for each $t\in\R^+$, $\phi (., t)$ is an homeomorphism  satisfying the semigroup property: $\phi(., t+
s) = \phi (\phi(., s), t) $ and for each $x\in\R^+$, $\phi_x (.):=\phi (x, .)$ is an increasing $\mathcal{C}^1$-diffeormorphism.
This implies that $\phi(x,0)=x$. 

\item The jump rate $\lambda :\R^+ \rightarrow \R^+$ is  a measurable function
satisfying
\[\forall x\in \R^+ ,\;\; \exists\:\varepsilon'>0\:\; \text{such that}\:\; \int_{0}^{\varepsilon'}\lambda(\phi (x,s))ds<\infty \]
that is, the jump rate does not explode. 
\item $\forall x\in\R^+$, $Q(x,\R^+\setminus \{x\})=1$. 
\end{enumerate}
\end{assumption}
For instance, we can take $\phi(x,t)=x+ct$ (linear flow) or $\phi(x,t)=xe^{ct}$ (exponential flow). The transition measure may be continuous with respect to the Lebesgue measure or deterministic ($Q(x,\{y\})=\nbOne_{\{y=f(x)\}}$).
 
Given these three characteristics, it can be shown (\citet[p62-66]{Dav93}), that there exists
a filtered probability space  $(\Omega , \mathcal{F},\{\mathcal{F}_{t}\}, \{\mathbb{P}_{x}\})$ such that the motion of the
process $\{X(t)\}_{t\geq 0}$ starting from a point $x_0 \in \R^+$ may be constructed as follows.
Consider a random variable $T_{1}$ with survival function
\begin{equation}\Pc{T_{1}>t}{X_0=x_0}=
   e^{-\Lambda (x_0,t)}  , \text{ where }  \Lambda (x,t)=\int_{0}^{t}\lambda (\phi (x,s))ds.
\label{t1}
\end{equation}
 If $T_{1}$ is equal to infinity, then the process $\{X(t)\}_{t\geq 0}$ follows the flow, i.e. for
$t\in \R^{+}$, $X(t) = \phi (x_0, t)$. Otherwise let $Y_{1}=\phi (x_0,T_{1}^-)$ the pre-jump location and $Z_1$ the post-jump location. $Z_1$ is defined through the transition kernel $Q$: 
$\Pc{Z_1\in A}{Y_1=y}=\int_A Q(y,dz)$. 
 The trajectory of $\{X(t)\}$ starting at
$x_0$, for $t \in [0,T_{1}]$, is given by
\[
X(t)=
\begin{cases}
\phi (x_0,t) &\text{for } t<T_{1},\\
Z_{1}      &\text{for } t= T_{1}.
\end{cases}
\]
Inductively starting from $X(T_{n}) = Z_{n}$, we now select the next inter-jump time $T_{n+1}- T_{n}$
and post-jump location $X(T_{n+1}) = Z_{n+1}$ in a similar way.
This construction properly defines a strong Markov process $\{X(t)\}_{t\geq 0}$   with jump times $\{T_{k}\}_{k\in\N}$
(where $T_{0} = 0$).  A very natural Markov chain is linked to $\{X(t)\}_{t\geq 0}$, namely the jump chain
$\{Y_n,Z_{n}\}_{n\in\N}$ (or, equivalently, $\{T_n,Z_n\}_{n\in\mathbb{N}}$).

\begin{figure}[h]
\caption{Examples of simulations of processes $\{X_t\}_{t\geq 0}$ and $\{Z_k\}_{k\in\mathbb{N}}$}

\begin{tabular}{cc}
TCP protocol& Bacterial growth\\
\includegraphics[width=0.45\linewidth, height=4cm]{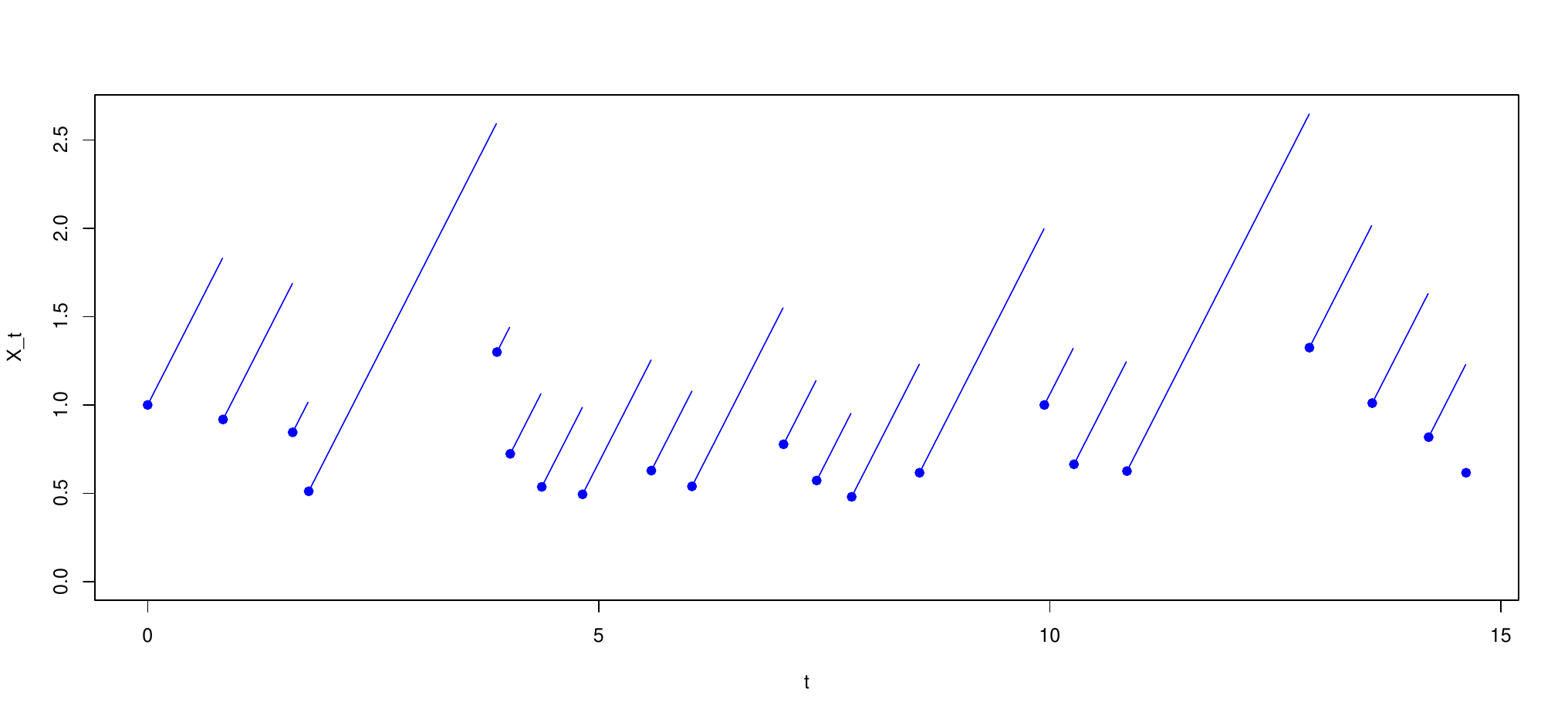}&
\includegraphics[width=0.45\linewidth, height=4cm]{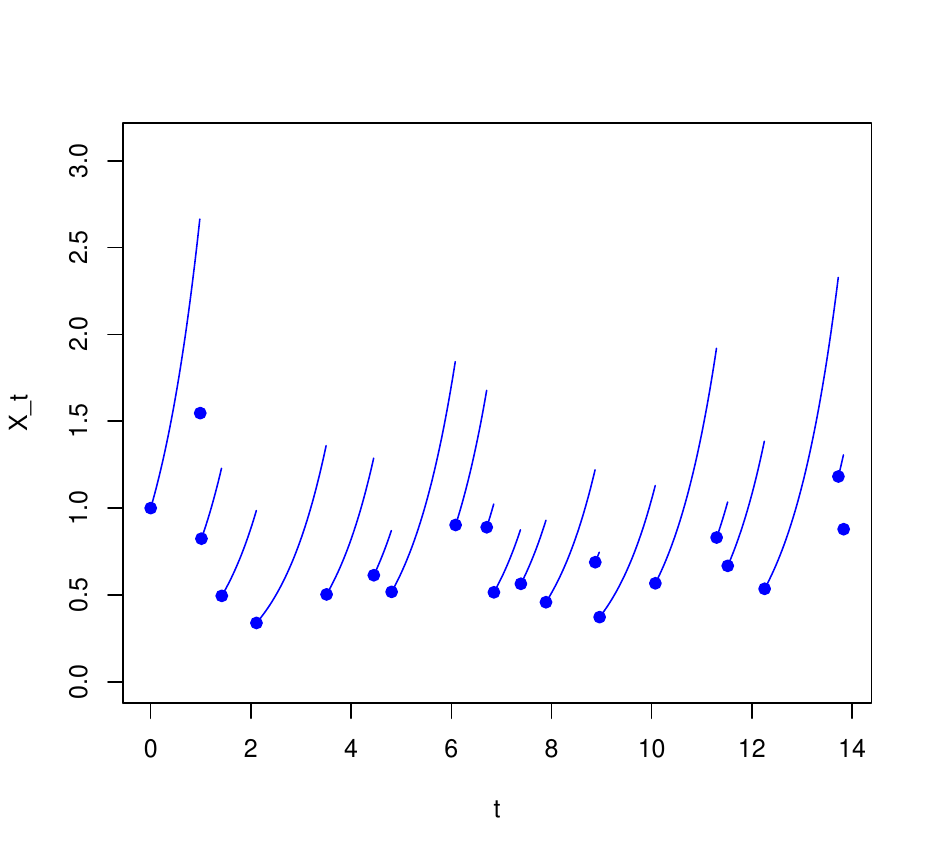}\\
$\phi(x,t)=x+t$, $Z_{k}=Y_{k}/2$, & $\phi(x,t)=xe^t$, $Z_{k}=Y_{k}U$, $U\sim\beta(20,20)$,\\ 
$\lambda(x)=\sqrt{x}$&$\lambda(x)=x^2$
\end{tabular}
\begin{center} 
\textcolor{blue}{$\bullet$ : process $\{Z_k\}_{k\in\mathbb{N}}$} $\quad $ \textcolor{blue}{ $-$ : process $\{X(t)\}_{t\geq 0}$}
\end{center} 
\end{figure}

To simplify the notations, let us set $\phi_x(t)=\phi(x,t)$ and $z_0=x_0$. 
By \eqref{t1}, 
\begin{align*}
\Pc{Y_1> y}{Z_0=z_0}&=
\Pc{T_1> (\phi_{z_0})^{-1}(y)}{Z_0=z_0}\\
&= \exp\left(-\int_0^{(\phi_{z_0})^{-1}(y)} \lambda(\phi_{z_0}(s))ds\right)\nbOne_{\{y\geq z_{0}\}}
\end{align*}
and by the change of variable $u=\phi_{z_0}(s)$ (we recall that for any $z\in\R^+$, $\phi_{z}$ is a monotonic function), we get 
\begin{equation}
\Pc{Y_1> y}{ Z_0=z_0}=\exp\left(-\int_{z_0}^{y}\lambda(u)\left(\phi^{-1}_{z_0}\right)'(u)du\right)\nbOne_{\{y\geq z_{0}\}}\label{fonction_repartition_t1}.
\end{equation} 
If the function $\lambda(y)(\phi_{z_0}^{-1})'(y)$ is finite, we obtain the conditional density:
\begin{equation}\pl(z_0,y):=\lambda (y)(\phi^{-1}_{z_{0}})'(y)e^{-\int_{z_{0}}^{y}\lambda (u)(\phi^{-1}_{z_{0}})'(u)du}\nbOne_{\{y\geq z_{0}\}}.\label{densitet1}\end{equation} 
By analogy, we set $\pl(z_0 ,dy)=\Pc{Y_1\in dy}{Z_0=z_0}$.

Our aim is to estimate the jump rate $\lambda$ on the compact interval $\mathcal{I}:=[i_1,i_2]\subset (0,\infty)$.

The ergodicity is often a keystone in statistical inference for Markov processes. We also assume fast convergence toward the stationary density.  
\begin{assumption}\label{hypo_contraction}$\;$
\begin{enumerate}
\item\label{hypo_contraction_pl_densite} The jump rate does not explode before $i_2$: for all $x\leq i_1$, $\int_0^{i_1}\lambda(y)(\phi_x^{-1})'(y)dy<\infty$ and $\sup_{y\in[i_1,i_2]} \lambda(y)<\infty$. 
\item\label{hypo_contraction_ergodicite} The process $(Y_k,Z_k)$ is recurrent positive and strongly
ergodic.  We denote by $\numes$  the stationary measure of
$Y_k$,  by
$\mul$ that of $Z_k$, by $\rhol$
the stationary measure of the couple $(Y_k,Z_k)$ and by $\xil$ that of 
$(Z_{k},Y_{k+1})$. We have that:
\begin{gather}\mul(dz)=\int_{\R^+} \numes(dy)Q(y,dz)=\int_{\R^+}
\rhol(dy,dz),\quad \rhol(dy,dz)=\numes(dy)Q(y,dz),\nonumber\\
\numes(dy)=\int_{\R^+}
\xil(dx,dy)=\int_{\R^+}\pl(z,dy)\mu(dz),\quad
\xil(dz,dy)=\mul(dz)\pl(z,dy). \label{eq_defi_mes_stat}
 \end{gather}
\item\label{hypo_contraction_V} There exist a function $\Vl$ greater than 1, two constants 
$\gamma\in]0,1[$,  $R\in\R^{+*}$ such that, for any function 
$\psi:(\R^+)^2\mapsto\R^+$, $\vert \psi\vert \leq \Vl$, for any 
integer $k$:  
\[\left\vert \Ec{\psi(Y_k,Z_k)}{Z_0=z_{0}}-\EE[\rhol] {\psi(Y_1,Z_1)}\right\vert \leq R\Vl(z_{0})\gamma^k.\]
The inequality $\vert \psi\vert \leq \Vl$ means that, for any $(y,z)\in(\R^+)^2$, $\vert \psi(y,z)\vert \leq \Vl(z)$. This inequality is true in particular for any function $\psi$ bounded by 1 and for $\psi(y,z)=\Vl(z)$. 
\end{enumerate}
\end{assumption}
Under Assumption A\ref{hypo_contraction}\ref{hypo_contraction_pl_densite}, the conditional measure $\pl$ is continuous with respect to the Lebesgue measure on $[0,i_2]\times [i_1,i_2]$ and 
$\sup_{x,y\in [0,i_2]\times [i_1,i_2]} \pl(x,y)<\infty$. 
So is  $y\to \nul(y)$: $\numes(dy)=\nul(y)dy$.
 Moreover, 
\[ \sup_{y\in[i_1,i_2]} \nul(y)=\sup_{y\in[i_1,i_2]} \int_0^{i_2} \pl(z,y)\mu(dz)<\infty.\]
We can also remark that, for any $x>0$, 
$\vert \EE[\mul]{\Vl(Z_1)}\vert \leq \Vl(x)+R\Vl(x)<\infty$. 

Let us set $\EE[z_0]{U}=\Ec{U}{Z_0=z_0}$. Under Assumption A\ref{hypo_contraction}, the empirical mean is close to its expectation under the stationary density, as shown by the following lemma (proved in the Appendix). 

\begin{lemma}\label{cor_majoration_variance}
	Under  Assumptions A\ref{hypopdmp}-A\ref{hypo_contraction}, for any bounded function $\ft$: 
	\[\left\vert \EE[z_0]{\frac{1}{n}\sum_{k=1}^n \ft(Y_k,Z_k)}-\int s(y,z)\rhol(dy,dz)\right\vert \leq \norm{s}_{\infty}\frac{R\Vl(z_0)}{n(1-\gamma)}\]
	and
	\begin{multline*}
	\Var[z_0]{\frac{1}{n}\sum_{k=1}^n\ft(Y_k,Z_k)}  \leq \frac{1}{n}\int \ft^2(y,z)\rhol(dy,dz)\\
	+\frac{\norm{s}_{\infty}}{n}
	\int \vert \ft(y,z)\vert  \Gl(z)\rhol(dy,dz)+\frac{\cl\norm{\ft}_{\infty}^2}{n^2}
	\end{multline*}
	where 
	$\Gl(z)=\frac{R}{1-\gamma}\left(\Vl(z)+\int \Vl(u)\mul(du)  \right)
	$ and $\cl$ depends explicitly on $(\gamma,R,\Vl)$. We can remark that $\Cl:=\int \Gl(z)\mul(dz)=\frac{2R}{1-\gamma}\int \Vl(z)\mul(dz)$. 
\end{lemma} 

In the bound of the variance, the first term 
is the same as for i.i.d variables. The second term is due to covariance terms (we found a similar term for stationary $\beta$-mixing processes), the third  comes from the non-stationarity of the  random vectors $(Y_k,Z_k)$. 

To study an adaptive estimator of $\nul$, we need to prove that the Markov chain $(Y_k,Z_k)$ is weakly dependent. It is the case if the process is $\beta$-mixing. 
\begin{definition}\label{def_mixing}
	Let $(X_k)_{k\geq 0}$ be a Markov process. Let us define the $\sigma$-algebra \[\rond{O}_a^b=\sigma(\{ X_{j_1}\in I_1,\ldots,X_{j_\mathbf{n}}\in I_{\mathbf{n}}\},a\leq j_1\leq \ldots\leq j_{\mathbf{n}}\leq b, \mathbf{n}\in \N, I_k\in \mathcal{B}(\R^+ )).\]
	The $\beta$-mixing coefficient of the Markov chain $(X_k)$ is 
	\[\beta_X(t)=\sup_k\sup_{E\in \rond{O}_{0}^k\times\rond{O}_{t+k}^{\infty}} \vert P_{ \rond{O}_{0}^k,\rond{O}_{t+k}^{\infty}}(E)-P_{ \rond{O}_{0}^k}\otimes P_{\rond{O}_{t+k}^{\infty}}(E)\vert \]
	where $P_{\rond{O},\rond{S}}$ is the joint law of an event on $\rond{O}\times\rond{S}$. 
	The $\beta$-mixing coefficient characterizes the dependence between what happens before $T_k$ and what happens after $T_{t+k}$.
	The process $(X_k)_{k\geq 0}$ is $\beta$-mixing if  $\lim_{k\rightarrow \infty}\beta_X(k)=0$. It is exponentially (or geometrically) $\beta$-mixing if there exists two positive constants $c$, $\beta$ such that $\beta_X(k)\leq ce^{-\beta k}$. 
\end{definition}

The following lemma is a consequence of Assumption A\ref{hypo_contraction}.  It   is proved in the Appendix. 
\begin{lemma}\label{cor_betageometricmixing}
	Under Assumptions A\ref{hypopdmp}-A\ref{hypo_contraction}, the Markov chain $(Y_k,Z_k)$ is geometrically $\beta$-mixing. Moreover, its $\beta$-mixing coefficient satisfies: $\forall k\in\N$:  
	\[ \beta_{Y,Z}(k)\leq c\gamma^k\quad \text{where}\quad c=R\int \Vl(z)\mul(dz)+R(1+R)\Vl(x_0).\]
\end{lemma}

 Estimating directly $\lambda$ is difficult, but we can construct a quotient estimator. 
By \eqref{fonction_repartition_t1} and \eqref{densitet1}, we get that, for any $y\in\mathcal{I}$, 
\begin{align*}\lambda(y)(\phi_{z_0}^{-1})'(y)\units{z_0\leq y}\Pc{Y_1> y}{Z_0=z_{0}}&=\pl(z_0,y)\\\
\lambda(y)\Ec{\units{Z_0\leq y< Y_1}(\phi_{Z_0}^{-1})'(y)}{Z_0=z_{0}}&=\pl(z_0,y)
\end{align*}
and we integrate with respect to the stationary distribution $\mul$ of $Z_0$
\[ \lambda(y) \EE[\xil]{\left(\phi^{-1}_{Z_0}\right)'(y)\units{Z_0\leq y< Y_1}} =\int \pl(z,y)\mu(dz) = \nul(y)\]
recalling that 
 $\xil$ is the stationary measure of the couple
$(Z_{0},Y_{1})$. 
Let us set 
\begin{equation} \Dg(y):=\EE[ \xil]{(\phi^{-1}_{Z_0})'(y)\nbOne_{\{Z_0\leq y< Y_1\}}}.\label{def_D}
\end{equation}
Then, if $\Dg(y)>0$,  we get: 
\begin{equation}\lambda(y)= \frac{\nul(y)}{\Dg(y)}.\label{formulel}\end{equation}  
It remains to ensure that  $\Dg(y)>0$ on  $\mathcal{I}=[i_1,i_2]$. 

\begin{assumption}\label{hypo_D_positif}
	There exists $D_0>0$ such that 
\[\inf_{y\in\mathcal{I}}\Dg(y)\geq D_0 >0.\] 
\end{assumption}

\begin{remark}
Assumption A\ref{hypo_D_positif} is very natural; indeed, let us set 
$\Phi_0:=\inf_{x\leq i_2,y\in\mathcal{I}} (\phi_x^{-1})'(y)$.
As $\phi_x$ is invertible, and $\phi_{\cdot}'(\cdot)$ is
continuous, $\Phi_0>0$. Then 
\[\Dg(y)\geq \Phi_0 \PP[\xi]{Z_0\leq y< Y_1}.\]
If the probability $\PP[\xi]{Z_0\leq y< Y_1}$ is null, then under the stationary distribution, the probability that $(X_t)$ passes through $y$ is null and  the jump rate at that point can not be measured. 

We can remark that if  $\mathbf{D}>0$ for some point $y$, then so is $\mathbf{P}(1_{Z_0\leq y\leq Y_1})>0$ and its estimator
 \[\hat{D}_n(y)=\frac{1}{n}\sum_{k=1}^n (\phi_{Z_k^{-1}})'(y)1_{Z_{k-1}\leq y\leq Y_k}>0.\] Then if we take an interval $[\hat{i}_1,\hat{i}_2]$ such that for some $n$, and some observation $(X_t)_{t\geq 0}$, $\hat{D}_n$ is positive on this interval, then Assumption A3 is satisfied on $[\hat{i}_1,\hat{i}_2]$. However, the true value of $D_0$ is unknown in that case. 
It should be noted that the interval $[\hat{i}_1,\hat{i}_2]$ should not be changed for each simulation, otherwise the convergence of the estimator on the whole interval can not be guaranteed (the interval of estimation would become larger and larger, and as $D$ is smaller on the edges on the new interval, and the convergence of the estimator is therefore slower). 
\end{remark}

  Assumptions A\ref{hypo_contraction} and A\ref{hypo_D_positif} are not explicit in $(\lambda,Q,\phi)$, so it is not easy to check that a particular model satisfies those assumptions. We give some explicit sufficient conditions on the coefficients $(\lambda,Q,\phi)$. For the next assumption, we use  the Hölder  spaces $H^{\alpha}$, as  defined in Appendix \ref{beshol}.

  \begin{assumptiond}[S] $\;$
  \begin{enumerate}
  \item\label{hypo_Q} The transition kernel is a contraction mapping:  there exists $\kappa<1$, 
 such that $\PP{Z_1\leq \kappa Y_1}=1$. 
\item\label{hypo_phi_borne} The flow is bounded: there exist two functions $\mg$ and $\Mg$ such that, $\forall x,y\in(\R^+)^2$: 
\[0<\mg(y)\leq (\phi_x^{-1})'(y)\leq \Mg(y).\]
\item\label{hypo_lambda_minore} The jump rate is positive on $[i_1,\infty[$ and there exists $\mathbf{a}>0$, $b>-1$ such that 
\[ \forall y\geq i_1,\quad \lambda(y)\mg(y)\geq \mathbf{a}\frac{y^b}{b+1}.\]
Then $\forall y\geq z$, $\PP[z]{Y_1\geq y}\leq \exp(-\mathbf{a}(y^{b+1}-z^{b+1}))$ and $\lim_{y\rightarrow \infty}\PP[z]{Y_1\geq y}=0$. 
  \item\label{hypo_lambda_borne} The jump rate does not explode too soon: there exist two positive constants  $ \Lg, \llg$, such that  $\norm{\lambda}_{L^{\infty}([i_1,i'_2])}\leq \Lg$ and  $\int_0^{i_1}\lambda(u)\Mg(u)du\leq \llg$ where   
 \[
 i'_2=\max
 \left(i_2,(i_2-i_1)+
 \left(\frac{1}{\mathbf{a}(1-\kappa^{b+1})} \ln\left(\frac{2\kappa^{b+1}}{1-\kappa^{b+1}}\right)
 \right)^{1/(b+1)}\units{\kappa^{b+1}\geq 1/3}
 \right).
 \] 
\end{enumerate}
These conditions ensure that Assumptions A\ref{hypo_contraction} and A\ref{hypo_D_positif} are satisfied. The following two assumptions allow us to control the regularity of $\nul$ (the rate of convergence of the estimator $\hat{\lambda}_{n}$ depends on the regularity of $\nul$, not on the regularity of $\lambda$). 
\begin{enumerate}[resume]   
\item For any $y\in\R^+$, $\lambda(y)<\infty$. This ensures that $\nul$ and $\pl$ are continuous with respect to the Lebesgue measure on $\R^+$. 
\item \label{hypo_regularite_nu}There exists $\alpha>0$ such that: 
\begin{itemize}
        \item $\forall K\subset\R^{+*}$ compact,    $\forall z\in\R^{+*}$, the function $(\phi_.^{-1})'(.)$ belongs to $H^{\alpha}([0,z]\times K)$. 
        \item $\forall K\subset\R^{+*}$ compact, $\lambda\in \holder(K)$. \item The transition measure $Q$ can be written \[Q(x,dy)=Q_1   (x,y)dy+p_0(x)\delta_0(dy)+\sum_{i=1}^{j_Q} p_i(x)\delta_{f_i(x)}(dy)\]
        with, for any compact $K$, $Q_1$ and $(p_i)_{0\leq i\leq j_Q}$ in $\holder[\alpha-1](K)$, and  $(f_i)_{1\leq i\leq j_Q}$ invertible functions such that $(f_i^{-1})_{1\leq i\leq j_Q}\in \holder(K)$. 
    \end{itemize}
\end{enumerate}

\end{assumptiond}

 If Assumption (S) is satisfied, for fixed flow $\phi$ and transition measure $Q$, we can introduce the  class of functions
\[\Fcb=\left\{
 \lambda\in \holder(\mathcal{J}), \forall y\geq i_1, \lambda(y)\mg(y)\geq \frac{\mathbf{a}y^b}{b+1}, \;
 \int_0^{i_1}\lambda(u)\Mg(u)\leq \llg,\;
 \norm{\lambda}_{\holder(\mathcal{J})}\leq \Lg
 \right\} \]
 with $\cf=(\mathbf{a},\llg,\Lg)\in (\R^{+})^3$ and the convex set 
 \begin{equation}
\mathcal{J}=\mathcal{J}_{\lfloor \alpha\rfloor}\cup[i_1,i'_2]\label{matcalJ}:=[j_1, j_2]
\end{equation} 
 is defined by the recurrence:
 \[ \mathcal{J}_0=\mathcal{I} \quad \text{and}\quad  \mathcal{J}_{k+1}=\operatorname{Conv}\left(\mathcal{I}\cup \bigcup_{i=1}^{j_Q}   f_i^{-1}(\mathcal{J}_{k})\right).\]

The following lemmas are proved in the Appendix. 
\begin{lemma}\label{hypo} 
	Under Assumptions A\ref{hypopdmp} and (S)
	  \begin{enumerate}
  \item  \label{Vc} Assumption A\ref{hypo_contraction} is  satisfied for  $\Vb(x):=\exp\left(\mathbf{a} x^{b+1}\right)$: there exists $R$, $\gamma$,  for any function $\vert \psi\vert \leq\Vb$,   \[
 \sup_{\lambda\in\Fcb} \left\vert \EE[z_0]{\psi(Y_k,Z_k)}-\EE[\rhol]{\psi(Y_0,Z_0)}\right\vert \leq R\Vb(z_0)\gamma^k 
\] recalling that the inequality $\vert \psi\vert \leq \Vb$ means that, for any $(y,z)\in(\R^+)^2$, $\vert \psi(y,z)\vert \leq \Vb(z)$. 
  \item \label{Dc} Assumption A\ref{hypo_D_positif} is satisfied. Moreover, there exists $\eta>0$, $D_0>0$ such that
  \[\inf_{\lambda\in\Fcb}\mul([0,i_1])\geq\eta\quad \textrm{and}\quad \inf_{\lambda\in\Fcb}\inf_{y\in\mathcal{I}}\Dg(y)\geq D_0.\]
  \end{enumerate}
\end{lemma}
\begin{lemma}\label{lem_regularite_nu}
If Assumptions A\ref{hypopdmp} and (S) are satisfied,  we can control the regularity of $\nul$: 
        \[ \norm{\nul}_{\holder(\mathcal{I})}\leq \psi_Q \left( \norm{\lambda}_{\holder(\mathcal{J})}, \norm{(\phi^{-1})'}_{\holder([0,j_2]\times \mathcal{\mathcal{J}})}\right)
        \text{ with }  \mathcal{J}= [j_1,j_2] \text{ defined in (\ref{matcalJ})}.\]  
\end{lemma}

\begin{remark}
In \cite{kre2}, the author introduces the set of functions $\mathcal{F}(\mathfrak{c},b)$ with very similar conditions. As she considers a transition measure $Q$ deterministic, the sets $\mathcal{F}(\mathfrak{c},b)$ and $\Fcb\cap \holder[\alpha]$ may not be equal. In particular, if $\lambda\in\Fcb$, then there exists $\mathfrak{c}$ such that $\lambda\in\mathcal{F}(\mathfrak{c},b)\cap\holder$. On the contrary, if  $\lambda$ belongs to $\mathcal{F}(\mathfrak{c},b)\cap \holder$ and the deterministic transition $f$ is $f(x)=\kappa x$, then for $i_1$ large enough, there exists $\cf$ such that $\lambda\in\Fcb$. This is no longer the case if, for instance, $f(x)\propto x^{\beta}$. As the transition measure $Q$ is unknown, it is not possible to exploit its characteristics. 

Another difference between the two sets is that $\lambda$ is estimated on the fixed interval $[i_1,i_2]$ and the assumptions depends on $(i_1,i_2)$,  whereas in \cite{kre2}, the interval of estimation depends on the set $\mathcal{F}(\mathfrak{c},b)$. 
\end{remark}

\section{Estimation of the jump rate}\label{section_estimation_lambda}

\subsection{The observation scheme}\label{3.1}
As in \cite{azais2016} and \cite{kre2}, the statistical inference is based on the observation scheme 
$(X(t), t\leq T_n)$
and asymptotics are considered when the number of jumps  of the process, $n$, goes to infinity.
Actually the  simpler observation scheme:
$(X(0),(X(T_{i^-}), X(T_i)), 1\leq i\leq n)=(Z_0,(Y_i,Z_i), 1\leq i\leq n)$ is sufficient, as $\phi$ is known and  one can remark that for all $n\geq 1$, $T_n=\phi_{Z_{n-1}}^{-1} (Y_n )$.

\subsection{Methodology}

\cite{kre2} and \cite{azais2016} construct a pointwise kernel estimator of $\nul$ before deriving an estimator of $\lambda$. Indeed, densities are often approximated  by kernels methods (see \citet{tsybakov} for instance). If the kernel is positive, the estimator is also a density. However, we want to control the $L^2$ risk of our estimator (not the pointwise risk), and also to construct an adaptive estimator. Estimators by projection are well adapted for $L^2$ estimation: if they are longer to compute at a single point than pointwise estimators, it is sufficient to know the estimated coefficients to construct the whole function. Furthermore, to find an adaptive estimator, we minimize a function of the norm of our estimator, that is the sum of the square of the coefficients, and the dimension. That is the reason why we choose an estimation by projection.

We first aim at estimating $\nul$ on the compact set  $\mathcal{I}$. 
We construct  a sequence of $L^2$ estimators by projection on an orthonormal basis.
As usual in nonparametric estimation, their risks can be decomposed in a variance term and a bias term which depends of the regularity of the density function $\nul$. We choose to use the Besov spaces (see Section \ref{beshol}) to characterize the regularity, which are well adapted to $L^2$ estimation (particularly for the wavelet decomposition). The "best" estimator is then selected by penalization. 
To construct the sequence of estimators,  we introduce a sequence of vectorial subspaces $S_m$. We construct an estimator $\hnuml$ of $\nul$ on each subspace  and then select the best estimator $\hnuml[\hat{m}]$.

\begin{assumption}\label{hypo_sev}$\;$
\begin{enumerate}
\item \label{hypo_sev_Sm_croissants}The subspaces $S_m$ are increasing and have finite dimension $D_m$. 
\item \label{hypo_sev_connexion_normes}The $L^ 2$-norm and the $L^{\infty}$-norm are connected: 
\[\exists \psi_1>0,\forall m\in\N, \forall s\in S_m, \quad \norm{s}_{\infty}^2\leq \psi_1 D_m\norm{s}_{L^2}^ 2.\]
This implies that, for any orthonormal basis $(\varphi_l)$ of $S_m$, \[\norm{\sum\limits_{l=1}^{D_m} \varphi_l^2}_{\infty}\leq \psi_1 D_m.\] 
\item \label{hypo_sev_non_stationnaire} There exists a constant $\psi_2>0$ such that, for any $m\in\N$, there exists an orthonormal basis $\varphi_{l}$ such that: 
\[ \norm{\sum_{l=1}^{D_m} \norm{\varphi_{l}}_{\infty}\vert \varphi_{l}(x)\vert }_{\infty}\leq \psi_2 D_m.\]
\item \label{hypo_sev_regularite} There exists $\mathbf{r}\in\N$, called the regularity of the decomposition, such that: 
\[\exists C>0,\forall \alpha\leq \mathbf{r},\forall \ft\in\Besov_{2,\infty}^{\alpha},\quad  \norm{\ft-\ft_m}_{L^2}\leq C D_m^{-\alpha}\norm{\ft}_{\Besov_{2,\infty}^{\alpha}}\]
where $\ft_m$ is the orthogonal projection of $\ft$ on $S_m$ and $\Besov_{2,\infty}^{\alpha}$ is a Besov space (see Appendix \ref{beshol}).  
\end{enumerate}

\end{assumption}
Conditions \ref{hypo_sev_Sm_croissants}, \ref{hypo_sev_connexion_normes} and \ref{hypo_sev_regularite} are usual (see \citet[section 2.3]{comtegenon2007} for instance). They are satisfied for subspaces generated by wavelets, piecewise polynomials or trigonometric polynomials (see \citet{devorelorentz} for trigonometric polynomials and piecewise polynomials and  \citet{meyer} for wavelets). 
 Condition \ref{hypo_sev_non_stationnaire}  is necessary because we are not in the stationary case: it helps us to control some covariance terms. It is obviously satisfied for bounded bases (trigonometric polynomials), and localized bases (piecewise polynomials). Let us prove it for a wavelet basis. Let   $\varphi$ be a father wavelet function, then $D_m=2^m$ and $\varphi_l(x)=2^{m/2}\varphi(2^m x-l)$. We get that 
$\norm{\sum_{l=1}^{D_m}\norm{\varphi_l}_{\infty}\vert \varphi_l(x)\vert}_{\infty}\leq 
 2^m \norm{\varphi}_{\infty}\norm{\sum_{l\in\mathbb{Z}} \vert \varphi(x-l)\vert }_{\infty}$.
As $\varphi$ is at least 0-regular, for $m=2$, there exists a constant $C$ such that 
$\vert \varphi(x)\vert \leq C(1+\vert x\vert^{-2})$. Then 
$\sup_x\sum_{l\in\mathbb{Z}}\vert \varphi(x-l)\vert\leq C\sup_x\sum_{l\in\mathbb{Z}} (1+\vert x-l\vert^{-2})<\infty$ 
and condition \ref{hypo_sev_non_stationnaire} is satisfied.

\subsection{Estimation of the stationary density}

Let us now construct an estimator $\hnuml$ of $\nul$ on the vectorial subspace $S_m$.  We consider an orthonormal basis $(\varphi_l)$ of $S_m$ satisfying Assumption A\ref{hypo_sev}.
Let us set 
\[a_{l}=<\varphi_{l},\nul>=\int_{\mathcal{I}} \varphi_{l}(x)\nul(x)dx \quad  
\textrm{and} \quad 
\numl(x)=\sum_{l=1}^{D_m} a_l\varphi_l(x).\]
The function $\numl$ is the orthogonal projection of $\nul$ on $L^2(\mathcal{I})$. 
We consider the estimator
\[\hnuml(x)=\sum_{l=
1}^{D_m} \hat{a}_{l}\varphi_{l}(x)\quad \textrm{with} \quad \hat{a}_{l}=\frac{1}{n}\sum_{k=1}^n \varphi_{l}(Y
_k).\]
\begin{prop}\label{theorem_risk_L2_m_fixed}
If $D_m^2\leq n$, under  Assumptions A\ref{hypopdmp}-A\ref{hypo_contraction} and A\ref{hypo_sev}, 
\[
\EE{\norm{\hnuml-\nul}_{L^2(\mathcal{I})}^2}\leq \norm{\numl-\nul}_{L^2(\mathcal{I})}^2+(\psi_1+C_{\lambda}\psi_2)\frac{D_m}{n}+\frac{c}{n}\]where 
$\Cl=\frac{2R}{1-\gamma} \int  \Vl(z)\mul(dz)$ and $c$ depends explicitly on $\Vl$, $\gamma$, $R$.  
\end{prop}
When $m$ increases, the bias term decreases whereas the variance term increases. It is important to find a good bias-variance compromise. 
If $\nul$ belongs to the Besov space $\Besov_{2,\infty}^{\alpha}(\mathcal{I})$, then $\norm{\numl-\nul}_{L^2(\mathcal{I})}^2\leq C\norm{\nul}_{B_{2,\infty}^{\alpha}(\mathcal{I})} D_m^{-2\alpha}$ (see Assumption A\ref{hypo_sev}\ref{hypo_sev_regularite}). If $\alpha\geq 1/2$, the risk is then  minimum for $D_{m_{opt}}\propto n^ {1/(2\alpha+1)}$ and we have, for some continuous function $\psi$: 
\[\EE{\norm{\hnuml[m_{opt}]-\nul}_{L^2(\mathcal{I})}^2}\leq
\psi\left(\norm{\nul}_{B_{2,\infty}^{\alpha}(\mathcal{I})},\Vl,R,\gamma \right)n^{-2\alpha/(2\alpha+1)}.
\]
This is the usual nonparametric convergence rate (see \citet{tsybakov}). If $\alpha<1/2$, then the risk is minimum for $D_m=n^{1/2}$ and the bias term is greater than the variance term. We can remark that a piecewise continuous function belongs to $B_{2,\infty}^{1/2}$. 

Let us now construct the adaptive estimator.  We compute $(\hnuml[0],\ldots,\hnuml,\ldots)$ for $m\in\rond{M}_n=\{m,D_m^2\leq n\}$. Our aim is to select automatically $m$, without knowing the regularity of the stationary density $\nul$.
 Let us introduce the contrast function
 $\gamma_n(\ft)=\norm{\ft}_{L^2}^2-\frac{2}{n}\sum_{k=1}^n \ft(Y_k)$. If $s\in S_m$, then we can write $s=\sum_l b_l\varphi_l$ and 
 \[\gamma_n(s)=\sum_{l=1}^{D_m} b_l^2-\sum_{l=1}^{D_m} b_l\,\frac{2}{n}\sum_{k=1}^n \varphi_l(Y_k).\]
The minimum is obtained for 
$b_l=\hat{a}_{l}=\frac{1}{n}\sum_{k=1}^n\varphi_l(Y_k)$.
 Therefore
\begin{equation}\hnuml=\arg\min_{\ft\in S_m}\gamma_n(\ft). \label{eq_num_min}\end{equation}
As the subspaces $S_m$ are increasing, the function $\gamma_n(\hnuml)$ decreases when  $m$ increases. To find an adaptive estimator, we need to add a penalty term $pen(m)$. 
Let us set $pen(m)= \frac{48(\psi_1+C_{\lambda}\psi_2)D_m}{n}+\frac{48c_{\lambda}\psi_1}{n}$ (or more generally 
$pen(m)=\frac{\sigma D_m}{n}+\frac{\sigma'}{n}$, with $\sigma\geq 48(\psi_1+C_{\lambda}\psi_2)$, $\sigma'\geq 48c_{\lambda}\psi_1$) and choose 
\begin{equation}\hat{m}=\arg\min_{m\in \rond{M}_n} \gamma_n(\hnuml)+pen(m).\label{eq_def_hm} \end{equation}
We obtain an adaptive estimator $\hnuml[\hat{m}]$. 

\begin{theorem}[Risk of the adaptive estimator]\label{theorem_risk_L2_adaptive}
Under   Assumptions A\ref{hypopdmp}-A\ref{hypo_contraction} and A\ref{hypo_sev},   
$\forall \sigma\geq 48(\psi_1+C_{\lambda}\psi_2)$, $\sigma'\geq 48c_{\lambda}\psi_1$, $pen(m)=\frac{\sigma D_m}{n}+\frac{\sigma'}{n}$,
\[
\EE{\norm{\nul-\hnuml[\hat{m}]}_{L^2(\mathcal{I})}
^2} 
\leq \min_{m\in\rond{M}_n} \left(3\norm{\numl-\nul}_{L^2(\mathcal{I})}^2+4pen(m)\right) + \frac{c'}{n}.\]
where $c'$ is a function of $(\Vl,R,\gamma,\norm{\nul}_{L^2(\mathcal{I})})$. We recall that $\rond{M}_n=\{m,D_m^2\leq n\}$.  
\end{theorem}
 The estimator is adaptive: it realizes the best bias-variance compromise, up to a multiplicative constant. We have an explicit rate of convergence if $\nul$ belongs to some (unknown) Besov space $\Besov_{2,\infty}^{\alpha}$: in that case, \[\norm{\nul-\numl}_{L^2(\mathcal{I})}^2\leq 3\norm{\nu_{m_{\text{opt}}}-\nul}_{L^2(\mathcal{I})}^2+4pen(m_{\text{opt}})+\frac{c}{n}\leq  C\norm{\nul}_{\Besov_{2,\infty}^{\alpha}}D_m^{-2\alpha}\] and if $\alpha\geq 1/2$, 
\begin{equation}\EE{\norm{\nul-\hnuml[\hat{m}]}_{L^2(\mathcal{I})}
^2}\leq\psi\left(\norm{\nul}_{\Besov_{2,\infty}^{\alpha}(\mathcal{I})},\Vl,R,\gamma\right)  n^{-2\alpha/(2\alpha+1)}\label{eq:risque_est_adaptatif_fcb}
\end{equation}
for some continuous function $\psi$.

\subsection{Estimation of the jump rate}\label{subsection_estimation_lambda}
By \eqref{formulel}, we have
\[\lambda(y)=\frac{\nul(y)}{\Dg(y)} \quad \text{recalling that}\quad 
\Dg(y)=\E_{\xil}\left((\phi^{-1}_{Z_0})'(y)\nbOne_{\{Z_0\leq y\leq Y_1\}}\right)\]
where $\xil$ is the stationary measure of $(Z_k,Y_{k+1})$. 

\begin{remark}
We notice that this formula is different as the one used in \cite{kre2} \[\lambda(y)= \frac{f(\nul(y))}{\tilde{\Dg}(y)}\] where 
\[\tilde{\Dg}(y):=\EE[ \nul]{((f\circ \: \phi_{Z_0})^{-1})'(f(y))\units{f(Z_0)\leq f(y)}\units{Z_1\geq f(y)}}.\]
As in \cite{kre2}, the author works under the assumption that $Q(x,\{y\})=\nbOne_{\{y=f(x)\}}$, the study was easier, here we need to consider the Markov chain $(Y_k ,Z_k)_{k\in\mathbb{N}}$.
\end{remark}
To estimate the jump rate, we construct a quotient estimator.
Let us consider the estimator 
\begin{equation}\hat{\lambda}_n(y)=\frac{\hnuml[\hat{m}](y)}
{\Dge(y)}\nbOne_{\{\hnuml[\hat{m}](y)\geq 0\}}\nbOne_{\{\Dge(y)\geq \ln(n)^{-1}\}}
\label{def_hat_lambda}
\end{equation}
where 
\[\Dge(y):=\frac{1}{n}\sum_{k=1}^n(\phi^{-1}_{Z_{k-1}})'(y)\nbOne_{\{Z_{k-1}\leq y\leq Y_k\}}.\]

\begin{remark}
As the process $\{X(t)\}$ is observed continuously without errors, $\phi^{-1}$ (and therefore $(\phi^{-1})'$) is known on $\cup_k [Z_{k-1},Y_k]$ so $\hat{\Dg}_n(y)$ is computable. 
 \end{remark}
The estimator  $\hat{\lambda}_n$  converges with nearly the same rate of convergence as $\hat{\nul}$:

\begin{theorem}\label{lem_estimation_lambda}
Under A\ref{hypopdmp}-A\ref{hypo_sev},
as soon as $\ln(n)^{-1}\leq D_0/2$, 
\begin{align*}\EE{\norm{\hat{\lambda}_n-\lambda}_{L^2(\mathcal{I})}^2}
&\leq 3\ln^2(n)\EE{\norm{\hnuml[\hat{m}]-\nul}_{L^2(\mathcal{I})}^2}+\cl'\frac{\ln^2(n)}{n}\\
&\leq 3\ln^2(n)\min_{m\in\rond{M}_n}\left\{ 3\norm{\hnuml-\nul}_{L^2(\mathcal{I})}^2+4pen(m)\right\}+\cl'\frac{\ln^2(n)}{n}
\end{align*}
where \[\cl'=\Phi_1^2\frac{2+\Cl}{D_0^2}\left(3\norm{\lambda}_{L^2(\mathcal{I})}^2+12\norm{\nul}_{L^2(\mathcal{I})}^2\right),\quad \Phi_1=\sup_{x\in[0,i_2],y\in\mathcal{I}} (\phi_x^{-1})'(y).\] 
\end{theorem}

The bias term depends of the regularity of the stationary density $\nul$, not of the regularity of $\lambda$. If we consider $\lambda$ and $\nul$ as functions of a Besov space, their regularities are not related: the Besov spaces are not stable by product (as they are subspaces of $L^2(\mathcal{I})$). We would like to link the rate of convergence of $\hat{\lambda}_n$ to the regularity of $\lambda$ rather than $\nul$, at least when $\lambda\in\Fcb$. In that case, $\lambda$ belong to some H\"older space, which is stable by product, composition and integration. See Appendix \ref{beshol} for the definition and properties of Besov and Hölder spaces. We obtain the following corollary: 
 \begin{corollary}\label{cor_bound_risk_lambda_fcb} Under A\ref{hypopdmp}, (S) and A\ref{hypo_sev},
as soon as $\ln(n)^{-1}\leq D_0/2$, for any $\alpha\geq 1/2$,
\[\sup_{\lambda\in\Fcb} \EE{\norm{\hat{\lambda}_n-\lambda}_{L^2(\mathcal{I})}^2}\lesssim  \ln^2(n)n^{-2\alpha/(2\alpha+1)}.\]
\end{corollary}

\begin{remark}
 \cite{kre2} obtain the same rate of convergence for a kernel estimator (with the regularity of $\lambda$ known). 
\end{remark}

\subsection{Minimax bound for the estimator of the jump rate}

We have proved that, under assumptions A\ref{hypopdmp}, (S) and A\ref{hypo_sev}, 
\[\sup_{\lambda\in\Fcb}\EE{\norm{\hat{\lambda}_n-\lambda}
_{L^2(\mathcal{I})}^2}\lesssim \ln^2(n)n^{-2\alpha/(2\alpha+1)}.\]
We would like to verify that our estimator converges with the minimax rate of convergence, i.e: 
\[ \inf_{\hat{\lambda}_n}\sup_{\lambda\in  \Fcb} \EE{\norm{\hat{\lambda}_n-\lambda}_{L^2(\mathcal{I})}^2}\geq C \ln^2(n) n^{-2\alpha/(2\alpha+1)}.\]
The $\ln^2(n)$ factor comes from the quotient estimator, we can not expect it will stay in the minimax bound. Indeed, it is clear that one could replace $\ln^{-1}(n)$ in \eqref{def_hat_lambda} by any function $w(n)$ greater than $D_0/2$. The best estimator will be obtained of course by taking $w(n)=D_0/2$ and the risk of this estimator (unreachable as $D_0$ is unknown) will be proportional to $n^{-2\alpha/(2\alpha+1)}$.

\begin{theorem}[Minimax bound] \label{lowerparametric}  If  A\ref{hypopdmp}, (S) and A\ref{hypo_sev} are satisfied,  then 
\[ \inf_{\hat{\lambda}_n}\sup_{\lambda\in  \Fcb} \EE{\norm{\hat{\lambda}_n-\lambda}_{L^2(\mathcal{I})}^2}\geq C  n^{-2\alpha/(2\alpha+1)}\]
where  the  infimum  is  taken  among  all
estimators.
\end{theorem}

\section{Proofs} \label{section_proofs}

Lemmas \ref{cor_majoration_variance}, \ref{cor_betageometricmixing}, \ref{hypo} and \ref{lem_regularite_nu} are proved in the Appendix. 
\sloppy

\subsection{Proof of Proposition \ref{theorem_risk_L2_m_fixed}}

We have the following bias-variance decomposition:
\begin{align*}
\EE{\norm{\nul-\hnuml}_{L^2(\mathcal{I})}^ 2}&=\int_{\mathcal{I}} \EE{\left(\nul(x)-\hnuml(x) \right)^2}dx\\
 &= \int_{\mathcal{I}} \left(\nul(x)-\EE{\hnuml(x)}\right)^2dx +\int_{\mathcal{I}}\Var{\hnuml(x)}dx\\
&=\norm{\EE{\hnuml}-\nul}_{L^2(\mathcal{I})}^2+
\int_{\mathcal{I}}\Var{\hnuml(x)}dx.
\end{align*}
The estimator $\hnuml$ (and therefore its expectation $\EE{\hnuml}$) belongs to the subspace $S_m$. Then, by orthogonality
\[\EE{\norm{\nul-\hnuml}_{L^2(\mathcal{I})}^ 2}=\norm{\nul-\numl}_{L^2(\mathcal{I})}^ 2+ \norm{\EE{\hnuml}
-\numl}_{L^2(\mathcal{I})}^2+\int_{\mathcal{I}} \Var{\hnuml(x)}dx.\]
The first terms are two terms of bias, the third  is a variance term. Let us first bound the second term of bias. As the functions $(\varphi_l)_{1\leq l\leq D_m}$ form an orthonormal basis of $S_m$, we have
\begin{align*}
\norm{\EE{\hnuml}-\numl}_{L^2(\mathcal{I})}^2&=\sum_{l=1}^{D_m} \left(\EE{\hat{a}_{l}}-a_{l}\right)^2\\
&=\sum_{l=1}^{D_m} \left( \frac{1}{n}\sum_{k=1}^n\EE{\varphi_{l}(Y_{k})}-\int_{\mathcal{I}} \varphi_{l}(x)\nul(x)dx\right)^2.
\end{align*}
 By Lemma \ref{cor_majoration_variance}, 
 \[\left\vert \frac{1}{n}\sum_{k=1}^n\EE{\varphi_{l}(Y_{k})}-\int_{\mathcal{I}} \varphi_{l}(x)\nul(x)dx\right\vert \leq \norm{\varphi_l}_{\infty} \frac{R\Vl(z_0)}{n(1-\gamma)}.\]
 As the $L^2$ and the $L^{\infty}$-norms are connected (see Assumption
 A\ref{hypo_sev}\ref{hypo_sev_connexion_normes}), $\norm{\varphi_l}_{\infty}^2\leq \psi_1D_m$ and, since $D_m^2\leq n$, we get:
\[
\norm{\EE{\hnuml-\numl}}_{L^2(\mathcal{I})}^2\leq \sum_{l=1}^{D_m}\frac{\norm{\varphi_{l}}_{\infty}^2}{n^2}\frac{R^2\Vl^2(z_0)}{(1-\gamma)^2}\leq \psi_1\frac{D_m^2}{n^2}\frac{R^2\Vl^2(z_0)}{(1-\gamma)^2}
\leq \psi_1\frac{1}{n}\frac{R^2\Vl^2(z_0)}{(1-\gamma)^2}.
\]
Let us now consider the variance term. As the functions $(\varphi_l)$ form an orthonormal basis of $S_m$, the integrated variance of $\hnuml$ is the sum of the variances of the coefficients $\hat{a}_{\lambda}$:
\begin{align*}
\int_{\mathcal{I}} \Var{\hnuml(x)}dx&=\int_{\mathcal{I}} \Var{\sum_{l=1}^{D_m}\hat{a}_l\varphi_l(x)}dx=\sum_{k,l} \Cov[z_0]{\hat{a}_k,\hat{a}_l}<\varphi_l,\varphi_k>_{L^2(\mathcal{I})}
\\
&=\sum_{l=1}^{D_m} \Var{\hat{a}_l}.
\end{align*}
By Lemma \ref{cor_majoration_variance}, as $\int_{\R^+} \rho(x,dz)=\nul(x)$, we get: 
\begin{align*}
\Var{\hat{a}_l}&=\Var{\frac{1}{n}\sum_{k=1}^n\varphi_l(Y_k)} \nonumber \\
&\leq \frac{1}{n} 
\int_{\mathcal{I}} \varphi_l^2(x)\nul(x)dx 
+\frac{\norm{\varphi_l}_{\infty}}{n}\int_{\mathcal{I}\times \R^+}  \vert \varphi_l(x)\vert \Gl(z)\rhol(dx,dz) +\frac{\cl\norm{\varphi_l}_{\infty}^2}{n^2}.
\end{align*}
By Assumptions A\ref{hypo_sev}\ref{hypo_sev_connexion_normes} and \ref{hypo_sev_non_stationnaire}, $\forall x$, $\sum_{l=1}^{D_m}\varphi_{l}^2(x)\leq \psi_1 D_m$, $\sum_{l=1}^{D_m}\norm{\varphi_l}_{\infty}\vert \varphi_l(x)\vert \leq \psi_2D_m$ and $\sum_{l=1}^{D_m}\norm{\varphi_l}^2_{\infty} \leq \psi_1 D_m^2\leq \psi_1n$. Therefore: 
\begin{equation} \int_{\mathcal{I}} \Var{\hnuml(x)}dx= \sum_{l=1}^{D_m} \Var{\hat{a}_l}
\leq (\psi_1+\Cl\psi_2)\frac{D_m}{n} +\frac{\cl}{n}\psi_1\label{majoration_variance_somme_al}
\end{equation}
where 
$\Cl=\int \Gl(z)\mul(dz)=\frac{2R}{1-\gamma}\int_{\mathcal{I}}  \Vl(z)\mul(dz)$ and $\cl$ depends only on $\Vl$, $R$ and  $\gamma$.

\subsection{Proof of Theorem \ref{theorem_risk_L2_adaptive}}

The number of coefficients in the adaptive estimator is random. If we are still able to control easily the bias term, we can not simply control the variance of our estimator by adding the variances of its coefficients. 
For any $m\in\rond{M}_n$, by definition of $\hat{m}$ (see \eqref{eq_num_min} and \eqref{eq_def_hm}), we have the following inequality:
\[
\gamma_n(\hnuml[\hat{m}])\leq \gamma_n(\hnuml)+pen(m)-pen(\hat{m})\leq \gamma_n(\numl)+pen(m)-pen(\hat{m})\\
,\]
with $\gamma_n(\ft)=\norm{\ft}_{L^2(\mathcal{I})}^2-2n^{-1}\sum_{k=1}^n \ft(Y_k)$. Then
\begin{equation}
\norm{\hnuml[\hat{m}]}_{L^2(\mathcal{I})}^ 2\leq\norm{\numl}_{L^2(\mathcal{I})}^ 2+pen(m)-pen(\hat{m})+\frac{2}{n}\sum_{k=1}^ n \left(\hnuml[\hat{m}](Y_k)-\numl(Y_k)\right).\label{eq_majoration_risque_ad_1}
\end{equation}
We have that, for any function $\ft\in L^2(\mathcal{I})$,  
$
\norm{\ft}_{L^2(\mathcal{I})}^2=\norm{\ft-\nul}_{L^2(\mathcal{I})}^2-\norm{\nul}_{L^2(\mathcal{I})}^2+2\int_{\mathcal{I}}\ft(x)\nul(x)dx$.
We apply this equality to $\hnuml[\hat{m}]$ and $\numl$. Equation \eqref{eq_majoration_risque_ad_1} becomes:
\begin{align*}
\norm{\hnuml[\hat{m}]-\nul}_{L^2(\mathcal{I})}^ 2 &\leq\norm{\numl-\nul}_{L^2(\mathcal{I})}^ 2+pen(m)-pen(\hat{m})\\ 
&+\frac{2}{n}\sum_{k=1}^n \hnuml[\hat{m}](Y_k)-\numl(Y_k)
- 2\int_{\mathcal{I}} (\hnuml[\hat{m}](x)-\numl(x))\nul(x)dx.
\end{align*}

The function $\hnuml[\hat{m}]-\numl$ belongs to the vectorial subspace $S_{\hat{m}}+S_m$. Therefore: 
\begin{align*}
\norm{\hnuml[\hat{m}]-\nul}_{L^2(\mathcal{I})}^ 2 &\leq \norm{\numl-\nul}_{L^2(\mathcal{I})}^ 2+pen(m)-pen(\hat{m})\\ 
&+ 2\norm{\hnuml[\hat{m}]-\numl}_{L^2(\mathcal{I})}\sup_{\ft\in\rond{B}_{m,\hat{m}}} \left\vert \sum_{k=1}^n \frac{1}{n} \ft(Y_k)-\int_{\mathcal{I}} \ft(x)\nul(x)dx\right\vert 
\end{align*}
where $\rond{B}_{m,m'}=\{\ft\in S_m+S_{m'},\norm{\ft}_{L^2(\mathcal{I})}=1\}$. As the sequence $(S_m)$ is increasing, $S_m+S_{m'}$ is simply the largest of the two subspaces. 
By the inequality of arithmetic and geometric means, 
\begin{align*}
\norm{\hnuml[\hat{m}]-\nul}_{L^2(\mathcal{I})}^ 2 &\leq \norm{\numl-\nul}_{L^2(\mathcal{I})}^ 2+pen(m)-pen(\hat{m})+ \frac{1}{4}\norm{\hnuml[\hat{m}]-\numl}_{L^2(\mathcal{I})}^2\\
&+\sup_{\ft\in\rond{B}_{m,\hat{m}}} 4 \left(\frac{1}{n}\sum_{k=1}^n \ft(Y_k)-\int_{\mathcal{I}} \ft(x)\nul(x)dx\right)^2.
\end{align*}
By the triangular inequality, $\norm{\hnuml[\hat{m}]-\numl}_{L^2(\mathcal{I})}^2\leq 2\norm{\hnuml[\hat{m}]-\nul}_{L^2(\mathcal{I})}^2+
2\norm{\numl-\nul}_{L^2(\mathcal{I})}^2$, and: 
\begin{align*}
\norm{\hnuml[\hat{m}]-\nul}_{L^2(\mathcal{I})}^ 2 &\leq 3\norm{\numl-\nul}_{L^2(\mathcal{I})}^ 2+2pen(m)-2pen(\hat{m})\\
&+ 8\sup_{\ft\in\rond{B}_{m,\hat{m}}} \left(\frac{1}{n}\sum_{k=1}^n \ft(Y_k)-\int_{\mathcal{I}}\ft(x)\nul(x)dx\right)^2.
\end{align*}
We can decompose the last term in a bias term and a variance term. Let us set: 
\begin{equation} I_n(\ft):=\frac{1}{n}\sum_{k=1}^n \ft(Y_k)-\EE{\ft(Y_k)} ,\; J_n(\ft):=\frac{1}{n}\sum_{k=1}^n\left(\EE{\ft(Y_k)}-\int_{\mathcal{I}} \ft(x)\nul(x)dx\right)
\label{def_I_J}
\end{equation}
and $p(m,m'):=(pen(m)+pen(m'))/8$.  Then: 
\begin{align}
\EE{\norm{\hnuml[\hat{m}]-\nul}_{L^2(\mathcal{I})}^ 2}&\leq 3 \norm{\numl-\nul}_{L^2(\mathcal{I})}^ 2+4pen(m)\nonumber\\&+16 \EE{\sup_{\ft\in\rond{B}_{m,\hat{m}}} I_n^2(\ft) +J_n^2(\ft)}
-16p(m,\hat{m}).
\label{risk_adaptif_decompo}
\end{align}
By Assumption A\ref{hypo_sev}\ref{hypo_sev_connexion_normes}, $\ft\in\rond{B}_{m,\hat{m}}$ implies that $\norm{\ft}_{\infty}^2\leq \psi_1(D_m+D_{\hat{m}})\leq 2\psi_1 n^{1/2}$ (we recall that $D_m$ and $D_{\hat{m}}$ are smaller than $n^{1/2}$). Then by   Lemma  \ref{cor_majoration_variance}, 
\begin{equation}
\sup_{\ft\in\rond{B}_{m,\hat{m}}}J_n^2(\ft) \leq  \sup_{\ft\in\rond{B}_{m,\hat{m}}}\frac{R^2\Vl^2(z_0)\norm{s}_{\infty}^2}{n^2(1-\gamma)^2} \leq \frac{4\psi_1^2 R^2\Vl^2(z_0)}{n(1-\gamma)^2}.\label{majoration_J}
\end{equation}
It remains to bound $\EE{\sup_{\ft\in\rond{B}_{m,\hat{m}}}I_n^2(\ft)-p(m,\hat{m})}_+$.
The unit ball $\rond{B}_{m,\hat{m}}$ is random. We can not bound $I_n^2(\ft)$ on it, we have to control the risk on the fixed balls $\rond{B}_{m,m'}$. We can write:
\begin{equation}
\EE{\sup_{\ft\in\rond{B}_{m,\hat{m}}}I_n^2(\ft)-p(m,\hat{m})}_+\leq \sum_{m,m'\in\rond{M}_n}\EE{\sup_{\ft\in\rond{B}_{m,m'}}I_n^2(\ft)-p(m,m')}_+.\label{majoration_I_somme}
\end{equation}
The Markov chain $(Y_1,\ldots,Y_n)$ is exponentially $\beta$-mixing with $\beta$-mixing coefficient $\beta_Y(k)\leq c\gamma^k=ce^{-\ln(1/\gamma)k}$.
The following lemma is deduced from the Berbee's coupling lemma and a Talagrand inequality. It is proved in the appendix. 

\begin{lemma}[Talagrand's inequality for $\beta$-mixing variables]\label{lem_talagrand_mixing}
Let  $Y_1,\dots,Y_n$ be a Markov chain exponentially $\beta$-mixing, with $\beta$-mixing coefficient $\beta_Y(k)\leq ce^{-b_0k}$. 
We choose $q_n:=c_q\ln(n)$ with $c_q\geq 2/b_0$, $p_n=n/(2q_n)$. We have that $\beta_Y(q_n)\leq c\gamma^{2 \ln(n)}\lesssim n^{-2}$. 
Let us consider
\[
I_n(\ft)=\frac{1}{n}\sum_{k=1}^ n \ft(Y_k)-\EE{\ft(Y_k)}.\]
If  we can find a triplet ($M_2$, $V$ and $H$) such that:
\begin{gather*}
\forall i, \sup_{\ft\in\rond{B}_{m,m'}}\Var{\frac{1}{q_n}\sum_{k=i}^{q_n+i} \ft(Y_k)}\leq \frac{V}{q_n},\\
 \sup_{\ft\in\rond{B}_{m,m'}}\norm{\ft}_{\infty}\leq M_2\;\textrm{ and }\;  \EE{\sup_{\ft\in\rond{B}_{m,m'}}\vert I_n(\ft)\vert }\leq \frac{H}{\sqrt{n}},
\end{gather*}
then we have: 
\begin{align*}
\EE{\sup_{\ft\in\rond{B}_{m,m'}}\vert I_n^2(\ft)-6H^2\vert }_+&
\leq K_1\frac{V}{n} \exp\left(-k_1\frac{H^2}{V}\right)+K_2\frac{M_2^2}{p_n^2}\exp\left(-k_2 \frac{\sqrt{p_n}H}{\sqrt{q_n}M_2}\right)\\
&+2\frac{M_2^2}{n^2}
\end{align*}
where $K_1$, $K_2$, $k_1$ and $k_2$ are universal constants. 
\end{lemma}
For the sake of simplicity, let us set  $D=D_m+D_{m'}$ and $\rond{B}=\rond{B}_{m,m'}$. By Assumption A\ref{hypo_sev}\ref{hypo_sev_connexion_normes},
\[\sup_{s\in\rond{B}} \norm{\ft}_{\infty} \leq \sup_{s\in\rond{B}} \psi_1^{1/2}D^{1/2}\norm{\ft}_{L^2(\mathcal{I})}= \psi_1^{1/2}D^{1/2} :=M_2.\]  By Lemma \ref{cor_majoration_variance}, 
\[\Var{\frac{1}{q_n}\sum_{k=1}^{q_n} s(Y_k)}\leq \frac{1}{q_n}\int_{\mathcal{I}} s^2(z)\nul(z)dz+\frac{\norm{s}_{\infty}}{q_n}\int_{\mathcal{I}} \vert s(z)\vert \nul(z)\Gl(z)dz+\frac{\cl\norm{s}_{\infty}^2}{q_n^2}.\]
By Cauchy-Schwarz, 
\[\norm{s^2\nul}_{L^1(\mathcal{I})}\leq \norm{s}_{L^2(\mathcal{I})}\norm{s\nul}_{L^2(\mathcal{I})}\leq \norm{s}_{L^2(\mathcal{I})}\norm{s}_{\infty}\norm{\nul}_{L^2(\mathcal{I})}\]
and 
\[\norm{s\nul \Gl}_{L^1(\mathcal{I})}\leq \norm{\Gl}_{L^{\infty}(\mathcal{I})} \norm{s}_{L^2(\mathcal{I})}\norm{\nul}_{L^2(\mathcal{I})}.\]
Then 
\[ \Var[z_0]{\frac{1}{q_n}\sum_{k=1}^{q_n} s(Z_k)}\leq \frac{\norm{s}_{L^2(\mathcal{I})}\norm{\nul}_{L^2(\mathcal{I})}}{q_n}
\left(\norm{s}_{\infty}+\norm{G_{\lambda}}_{L^{\infty}(\mathcal{I})}\right)
+c_{\lambda} \frac{\norm{s}_{\infty}^2}{q_n^2}.\]By Assumption A\ref{hypo_sev}\ref{hypo_sev_connexion_normes}, $\norm{s}_{\infty}\leq \psi_1^{1/2}D^{1/2}$, moreover, $\sup_{s\in\mathcal{B}} \norm{s}_{L^2(\mathcal{I})}=1$ and then
\begin{align*}
 \sup_{\ft\in\rond{B}} \Var{\frac{1}{q_n}\sum_{k=1}^{q_n}\ft(Z_k) }
&\leq \frac{\psi_1^{1/2}D^{1/2}\norm{\nul}_{L^2(\mathcal{I})}\left(1+\norm{\Gl}_{L^{\infty}(\mathcal{I})}\right)}{q_n}+\frac{\cl\psi_1 D}{q_n^2}\\
&\leq \frac{c_1D^{1/2}}{q_n}+c_2\frac{D}{q_n^2}:=\frac{V}{q_n}.
\end{align*}

It remains to find $H$ such that $\EE{\sup_{\ft\in\rond{B}}\vert I_n(\ft)\vert }\leq H/\sqrt{n}$. Let us introduce $(\varphi_l)_{1\leq l\leq D}$ an orthonormal basis of $S_m+S_{m'}=S_{\max(m,m')}$ satisfying Assumption A\ref{hypo_sev}. Then we can write 
$s=\sum_l b_l\varphi_l.$ As the function $s\rightarrow I_n(s)$ is linear: 
\[
\sup_{\ft\in\rond{B}}I_n^2(\ft)=\sup_{\sum b_l^2=1} \left(\sum_{l=1}^D b_l I_n(\varphi_l)\right)^2
\leq \sup_{\sum b_l^2=1}\left( \sum_{l=1}^D b_l^2 \right)\left(\sum_{l=1}^{D} I_n^2(\varphi_l)\right)=\sum_{l=1}^{D} I_n^2(\varphi_l).
\]
We can remark that  $I_n(\varphi_{l})=\hat{a}_{l}-
\EE{\hat{a}_{l}}$ (see equation \eqref{def_I_J}) and by consequence, $\EE{I_n^2(\varphi_{l})}=\Var{\hat{a}_{l}}$ .  
By  \eqref{majoration_variance_somme_al}:
\begin{align*}\EE{\sup_{\ft\in\rond{B}}I^2_n(\ft)}&\leq \sum_{l=1}^D\EE{I_n^2(\varphi_l)}=\sum_{l=1}^D\Var[z_0]{\hat{a}_l}\\
&\leq \frac{(\psi_1+\Cl\psi_2) D}{n}+\frac{\cl\psi_1}{n}:=\frac{H^2}{n}.\end{align*}
We can now apply Lemma \ref{lem_talagrand_mixing} with \[M_2=\psi_1^{1/2}D^{1/2},\quad V=c_1D^{1/2}+c_2 D/q_n\quad\text{and}\quad H^2=(\psi_1+\Cl\psi_2)D+c_{\lambda}\psi_1.\]  
For $p(m,m')\geq 6(\psi_1+\Cl\psi_2)D/n+6\cl\psi_1/n$, we get
\stepcounter{numero}
\begin{align*}
E_{\arabic{numero}}:=&\EE{\sup_{\ft\in\rond{B}}I_n^2(\ft)-p(m,m')}_+\\
&\leq K_1\frac{V}{n}\exp\left(-k_1 \frac{H^2}{V}\right)
+K_2\frac{M_2^2}{p_n^2}
\exp\left(-k_2\frac{\sqrt{p_n} H}{\sqrt{q_n M_2}}\right)
+2\frac{M_2^2}{n^2}.
\end{align*}
As $2/(x+y)\geq \min(1/x,1/y)$,
\begin{align*}
    \exp\left(-k_1\frac{H^2}{V}\right)&
\leq \exp\left(
-\frac{k_1}{2}\min\left( \frac{H^2}{c_1D^{1/2}},\frac{H^2}{c_2D/q_n}\right)
\right)\\
&\leq \exp\left( - \frac{k_1H^2}{2c_1D^{1/2}}\right)+\exp\left(-\frac{k_1H^2}{2c_2Dq_n}\right)
\end{align*}
and therefore
\begin{align*}
E_{\arabic{numero}}&\leq K'_1  \left(\frac{D^{1/2}}{n}+\frac{D}{nq_n}\right)
\left(\exp\left(-\frac{k'_1 D}{D^{1/2}}\right)
+\exp\left(-\frac{k''_1 D q_n}{D}\right)\right)\\
&+ \frac{ K'_2D}{p_n^2}
\exp\left(-k'_2 p_n^{1/2}\frac{ D^{1/2}}{q_n^{1/2}}\frac{1}{D^{1/2}}\right)
+K'_3\frac{D}{n^2}\\
&\leq K_{1}^{\lambda}\left(\frac{D}{n}\exp(-k_{1}^{\lambda} D^{1/2}) 
+\frac{D}{n}\exp(-k^{\lambda}_1  q_n)\right)+ K_{2}^{\lambda} \frac{D\ln^2(n)}{n^2}
\exp\left(-k_2^{\lambda}  \frac{n^{1/2}}{\ln(n)}\right)+K_3^{\lambda}\frac{D}{n^2}
\end{align*}
where $(K_{i}^{\lambda})_{1\leq i\leq 3}$ and $(k_i^{\lambda})_{1\leq i\leq 2}$  depend on  $(\Vl,R,\gamma,(\psi_1,\psi_2),\norm{\nul}_{L^2(\mathcal{I})})$.  
The second term can be made smaller than $n^{-2}$ for $c_q$ large enough. The third is also smaller to $n^{-2}$ thanks to the exponential term. 
Then 
\[ \EE{\sup_{s\in\rond{B}_{m,m'}}I_n^2(s)-p(m,m')} \leq K_1^{\lambda} \frac{D_{m,m'}}{n}\exp(-k_1^{\lambda} D_{m,m'}^{1/2}) +K_4^{\lambda} \frac{D_{m,m'}}{n^2}.\]
All the dimensions $D_{m,m'}$ are different, so 
$\sum_{m'\in\rond{M_n}} D_{m,m'}e^{-cD_{m,m'}^{1/2}}\leq \sum_{l=1}^{\infty} le^{-cl^{1/2}}<\infty$. Moreover, as $\sup_{m'\in\rond{M}_n}D_{m,m'}\leq \sqrt{n}$, $\sum_{m'\in\rond{M}_n} D_{m,m'}\leq\max_{m'\in\rond{M}_n} D_{m,m'}^2\leq n$. Then   by \eqref{majoration_I_somme}, 
\begin{equation}
 \EE{\sup_{\ft\in\rond{B}_{m,\hat{m}}}I_n^2(\ft)-p(m,\hat{m})}
\leq \frac{c}{n}.
\label{majoration_risque_adaptatif_In}
\end{equation}
Collecting \eqref{risk_adaptif_decompo}, \eqref{majoration_J} and \eqref{majoration_risque_adaptatif_In}, for any $m\in\rond{M}_n$: 
\[ \EE{\norm{\hnuml[\hat{m}]-\nul}_{L^2}^ 2}\leq 3\norm{\numl-\nul}_{L^2}^2+4pen(m)+\frac{c}{n}.\]

All the constants involved in the bound of $J_n^2$ and $I_n^2$ ($M_2$, $H$, $V$) depends on $\Vl$, $\gamma$, $R$ and $\norm{\nul}_{L^2(\mathcal{I})}\leq \norm{\nul}_{\Besov_{2,\infty}^{\alpha}(\mathcal{I})}$.   Then there exists an continuous function $\psi$ such that $x\to \psi(x,v,c,r)$ is increasing and 
\[
 \EE{\norm{\hnuml[\hat{m}]-\nul}_{L^2(\mathcal{I})}^ 2}\leq 3\norm{\numl-\nul}_{L^2(\mathcal{I})}^2+4pen(m)+\frac{\psi\left(\norm{\nul}_{B_{2,\infty}^{\alpha}},\Vl,\gamma, R\right)}{n}.\]

\subsection{Proof of  Theorem \ref{lem_estimation_lambda}}

Let us first  control $\EE{(\Dge(y)-\Dg(y))^2}$. 
As $\phi$ is a diffeomorphism, the function 
$\left(\phi_x^{-1}\right)'(y)$ is bounded on $[0,i_2]\times\mathcal{I}$. The function  $s_{x,z}(y)=\left(\phi_{x}^{-1}\right)'(y)\nbOne_{\{x\leq y\leq z\}}$ is bounded by a constant on $\mathcal{I}$:
\begin{equation} \norm{s_{x,z}}_{L^{\infty}(\mathcal{I})}\leq \sup_{x\in[0,i_2],y\in\mathcal{I}}\left(\phi_x^{-1}\right)'(y):=\Phi_1.\label{eq_def_phi1}
\end{equation}
We have that 
\[
\EE{\left(\Dge(y)-\Dg(y)\right)^2}=
\EE{\left(\frac{1}{n}\sum_{k=1}^n s_{Z_{k-1},Y_k}(y)-\EE[\xil]{s_{Z_0,Y_1}(y)}
	\right)^2}\label{eq_majoration_E_Dn}
\]
with $\xil$ the stationary density of $(Z_{k-1},Y_k)$ introduced in Assumption A2.
By Lemma \ref{cor_majoration_variance}, we have
\begin{align*}
\EE{\left(\Dge(y)-\Dg(y)\right)^2}&
\leq \frac{1}{n}\left(\Phi_1^2+\Phi_1^2\int \Gl(z)\mul(dz)\right)+\frac{\Phi_1^2c_{\lambda}}{n^2}
+\Phi_1^2\frac{R^2\Vl^2(z_0)}{n^2(1-\gamma)^2}
\\
&\leq \frac{\Phi_1^2(1+\Cl)}{n}+\frac{\Phi_1^2(c_{\lambda}+R^2V_{\lambda}^2(z_0)/(1-\gamma)^2)}{n^2}
\end{align*} 
and therefore 
\begin{equation}\EE[z_0]{(\Dge(y)-\Dg(y))^2}\leq \frac{\Phi_1^ 2}{n}\left(1+\Cl\right)+\frac{c}{n^2}.\label{eq:majoration_D_Dn}
\end{equation}
For $n$ large enough, $1/\ln(n)$ is smaller than $ D_0/2$ ($D_0$ is defined in Assumption A\ref{hypo_D_positif}) and then 
 by Markov inequality, 
\begin{equation}
\PP{\Dge(y)\leq 1/\ln(n)}\leq \PP{\vert \Dge(y)-\Dg(y)\vert \geq  D_0/2}\leq \frac{4}{D_0^2}\EE{(\Dge(y)-\Dg(y))^2}.\label{eq:majoration_proba_D}
\end{equation}

As $\nul$ is a positive function, $\vert \hat{\lambda}_n(y)-\lambda(y)\vert \nbOne_{\{\hnuml[\hat{m}](y)\geq 0\}}\leq \vert \hat{\lambda}_n(y)-\lambda(y)\vert$ and therefore, according to the definition of the estimator $\hat{\lambda}_n$ (see \eqref{def_hat_lambda}),
\[\vert \hat{\lambda}_n(y)-\lambda(y)\vert \leq \left\vert \frac{\hnuml[\hat{m}](y)}{\Dge(y)}-\frac{\nul(y)}{\Dg(y)}\right\vert 
\nbOne_{\{\Dge(y)\geq 1/\ln(n)\}}+\lambda(y)\nbOne_{\{\Dge(y)\leq 1/\ln(n)\}}.\]
We can write: 
\[
 \left\vert \frac{\hnuml[\hat{m}](y)}{\Dge(y)}-\frac{\nul(y)}{\Dg(y)}\right\vert 
\leq \left\vert \frac{\hnuml[\hat{m}](y)-\nul(y)}{\Dge(y)}+
\frac{\nul(y)}{\Dge(y)\Dg(y)}(\Dg(y)-\Dge(y))
\right\vert 
.
\]
As  $\Dg\geq D_0$ by Assumption A\ref{hypo_D_positif}: 
\[\vert \hat{\lambda}_n(y)-\lambda(y)\vert 
\leq \ln(n)\left(\vert \hnuml[\hat{m}](y)-\nul(y)\vert \right)+\ln(n)\frac{\vert \Dge(y)-\Dg(y)\vert }{ D_0}\nul(y)+\lambda(y)\nbOne_{\{\Dge(y)\leq 1/\ln(n)\}}.\]
By \eqref{eq:majoration_D_Dn} and \eqref{eq:majoration_proba_D},
\begin{align*}
\EE{\norm{\hat{\lambda}_n-\lambda}_{L^2(\mathcal{I})}^2}
&\leq 3\ln^2(n)\EE{\norm{\hnuml[\hat{m}]-\nul }_{L^2(\mathcal{I})}^2}\\
&+3D_0^{-2}\ln^2(n)
\int_{\mathcal{I}}\EE{\left(\Dge(y)-\Dg(y)\right)^2}\nul^2(y)dy\\
&+12D_0^{-2}\int_{\mathcal{I}}\EE{\left(\Dge(y)-\Dg(y)\right)^2} \lambda^2(y)dy\\
&\leq 3\ln^2(n)\EE{\norm{\hnuml[\hat{m}]-\nul }_{L^2(\mathcal{I})}^2}+\cl'\frac{\ln^2(n)}{n}
\end{align*}
with $\cl'=\frac{\Phi_1^ 2}{D_0^2}(4+2\Cl)(3\norm{\nul}_{L^2(\mathcal{I})}^2+12\norm{\lambda}_{L^2(\mathcal{I})}^2)$.

\subsection{Proof of Theorem \ref{lowerparametric}}

We use the reduction scheme described in \citet[chapter 2]{tsybakov}.
By Markov inequality,  
\[
C'^2n^{-2\alpha/(2\alpha+1)}\El{\lambda}{\norm{\hat{\lambda}_n-\lambda}_{L^2(\mathcal{I})}^2}
\geq \Pl{\lambda}{\norm{\hat{\lambda}_n-\lambda}_{L^2(\mathcal{I})}\geq C' n^{-\alpha/(2\alpha+1)}}.\]
Our aim is to show that  
\[\inf_{\hat{\lambda}_n}\sup_{\lambda\in \Fcb} \Pl{\lambda}{\norm{\hat{\lambda}_n-\lambda}_{L^2(\mathcal{I})}\geq C' n^{-\alpha/(2\alpha+1)}}>0.\]
Instead of searching an infimum on the whole class $\Fcb$, we can limit ourselves to the finite set $\{\lambda_0,\ldots,\lambda_{P_n}\} \in \Fcb$, such that 
\begin{equation} \norm{\lambda_i-\lambda_j}_{L^2(\mathcal{I})}\geq 2C'n^{-\alpha/(2\alpha +1)}\nbOne_{\{i\neq j\}}.\label{eq_lambda_j_distants}
\end{equation}
 Then 
 \stepcounter{numero}
\begin{align*}
    E_{\arabic{numero}}&:=\inf_{\hat{\lambda}_n}\sup_{\lambda\in \Fcb}  \Pl{\lambda}{\norm{\hat{\lambda}_n-\lambda}_{L^2(\mathcal{I})}\geq C' n^{-\frac{\alpha}{2\alpha+1}}}\\
    &\geq \inf_{\hat{\lambda}_n}\max_{j}  \Pl{\lambda_{j}}{\norm{\hat{\lambda}_n-\lambda_j}_{L^2(\mathcal{I})}\geq C' n^{-\frac{\alpha}{2\alpha+1}}}.
\end{align*}
We note $\psi^*$ the predictor  
\[\psi^*:=\arg\min_{0\leq j\leq P_n} \norm{\hat{\lambda}_n-\lambda_j}_{L^2(\mathcal{I})}.\]
By the triangular inequality, 
$\norm{\hat{\lambda}_n-\lambda_j}_{L^2(\mathcal{I})}\geq\norm{\lambda_{\psi^*}-\lambda_j}_{L^2(\mathcal{I})}- \norm{\lambda_{\psi^*}-\hat{\lambda}_n}_{L^2(\mathcal{I})}$.

Consequently, as $\norm{\hat{\lambda}_n-\lambda_j}_{L^2(\mathcal{I})}\geq \norm{\hat{\lambda}_n-\lambda_{\psi^*}}_{L^2(\mathcal{I})}$, 
\[\left\{ 
\norm{\hat{\lambda}_n-\lambda_j}_{L^2(\mathcal{I})}\geq A_n \right\} \supseteq  \left\{\left\{\norm{\lambda_{\psi^*}-\hat{\lambda}_n}_{L^2(\mathcal{I})}\geq A_n\right\} \cup \left\{\norm{\lambda_{\psi^*}-\lambda_j}_{L^2(\mathcal{I})} \geq 2A_n\right\}\right\}.\]
By \eqref{eq_lambda_j_distants},
	$\norm{\lambda_{\psi^*}-\lambda_j}_{L^2(\mathcal{I})}\geq 2C'n^{-\alpha/(2\alpha +1)}\nbOne_{\{\psi^*\neq j\}}$.
Then setting $A_n=C' n^{-\alpha/(2\alpha+1)}$, \\
$\left\{ 
\norm{\hat{\lambda}_n-\lambda_j}_{L^2(\mathcal{I})}\geq C'n^{-\alpha/(2\alpha+1)} \right\} \supseteq  \left\{ \psi^*\neq j\right\}$
 and therefore: 
\[\inf_{\hat{\lambda}_n}\sup_{\lambda\in\Fcb}\Pl{\lambda}
{
\norm{ \hat{\lambda}_n-\lambda}_{L^2(\mathcal{I})}^2\geq C'n^{-\alpha/(2\alpha+1)}}\geq \inf_{\hat{\lambda}_n}\max_{j} \Pl{\lambda_j}{\psi^*\neq j}.
\]
We denote by $\mathbf{P}^{\lambda_j}$ the law  of $(Z_0,Y_1,Z_1,\ldots,Y_n,Z_n)$ under $\lambda_j$. 
The following lemma is exactly  Theorem 2.5 of \citet{tsybakov}.
\begin{lemma}\label{lem_minoration_L2}
Let us consider a series of functions $\lambda_0,\ldots,\lambda_{P_n}$ such that:
\begin{enumerate}
\item The functions $\lambda_i$ are sufficiently apart: $\forall i\neq j$
\[ \norm{\lambda_i-\lambda_j}_{L^2(\mathcal{I})}\geq 2C'n^{-\alpha/(2\alpha +1)}.\]
\item For all $i$, the function $\lambda_i$ belongs to the subspace $\Fcb$. 
 \item Absolute continuity: $\forall 1\leq j\leq P_n$, $\mathbf{P}^{\lambda_j}<<\mathbf{P}^{\lambda_0}$.  
\item The distance between the measures of probabilities is not too large: 
\[\frac{1}{P_n}\sum_{j=1}^{P_n}
\chi^2(\mathbf{P}^{\lambda_j},\mathbf{P}^{\lambda_0})\leq c\ln(P_n)\] with $0<c<1/8$, and $\chi^2(.,.)$  the $\chi$-square divergence.
\end{enumerate}
Then 
\begin{align*}\inf_{\hat{\lambda}_n}\sup_{\lambda\in \Fcb}
C'^2n^{-2\alpha/(2\alpha+1)} \El{\lambda}{\norm{\hat{\lambda}_n-\lambda}_{L^2(\mathcal{I})}^2}&\geq  \inf_{\hat{\lambda}_n} \max_j\Pl{\lambda_{j}}{\psi^*\neq j}\\
&\geq 
\frac{\sqrt{P_n}}{1+\sqrt{P_n}}\left(1-2c-2\sqrt{\frac{c}{\ln(P_n)}}\right)&\\
&>0.
\end{align*}
\end{lemma}

\paragraph*{ Step 1: Construction of $(\lambda_0,\ldots\lambda_{P_n})$.} 
Let us set \[\lambda_0(x)= \ep \nbOne_{\{i_1\leq x\leq j_2\}}+ \mathbf{a} \frac{x^b}{(b+1)\mg(x)}\nbOne_{\{x>j_2\}}\quad\text{where}\quad \ep=\max_{x\in[i_1,j_2] }\mathbf{a} \frac{x^b}{(b+1)\mg(x)}\] with $\mathcal{J}=[j_1, j_2]$ defined in (\ref{matcalJ}).
As $\lambda_0$ is constant on $\mathcal{J}$, this function belongs to the H\"older space $H^{\alpha}(\mathcal{J})$ and $\norm{\lambda_0}_{H^{\alpha}(\mathcal{J})}=\ep$ (see Appendix \ref{beshol} for the definition of the Hölder space). It remains to ensure that it belongs to $\Fcb$. If $\ep>\Lg$, then $\Fcb=\emptyset$. 
If $\Lg=\ep$, then any function $\lambda\in\Fcb$ satisfies: $\forall x\in[i_1,j_2], \lambda(x)=\lambda_0(x)$. Let us assume that $\ep<\Lg$: in that case, there exists $\delta>0$ such that 
$\norm{\lambda_0}_{H^{\alpha}(\mathcal{J})}\leq \Lg-\delta$.

We consider a non-negative function $K\in H^{\alpha}(\R)$, bounded,  with support in $[0,j_2-i_1[$ and such that $\norm{K}_{L^1}\leq 1$. 
 We set $h_n=n^{-1/(2\alpha+1)}$,  $\nb=\lceil1/h_n\rceil$ and, for $0\leq k\leq p_{n}-1$, $x_k=i_1+h_nk(j_2-i_1)$.   We consider the functions 
$\varphi_k(x):= \cons h_n^{\alpha}K\left((x-x_k)/h_n\right)$ with $a<1$. The functions $\varphi_k$ have support in $[x_k,x_{k+1})\subset \mathcal{J}$. Moreover, by a change of variable $y=(x-x_k)/h_n$, 
$\norm{\varphi_k}_{L^1}=a h_n^{\alpha+1} \norm{K}_{L^1}\leq ah_n^{\alpha+1}$ and $\norm{\varphi_k}_{L^2}^2=a^2h_n^{2\alpha+1} \norm{K}_{L^2}^2$.  We consider the set of functions 
\[\rond{G}_n:=
\left\{\lambda_{\epsilon}:=\lambda_0+\sum_{k=0}^{\nb-1}\epsilon_k\varphi_k,\quad (\epsilon_k)\in\{0,1\}^{\nb} \right\}
.\]
 The cardinal of $\rond{G}_n$ is $2^{\nb}$. For two vectors $(\epsilon,\eta)$ with values in $\{0,1\}^{\nb}$, the distance  between two functions $\lambda_{\epsilon}$ and $\lambda_{\eta}$ is: 
\begin{equation}
\norm{\lambda_{\epsilon}-\lambda_{\eta}}_{L^2}^2=
a^2h_n^{2\alpha+1}\norm{K}_{L^2}^2\sum_{k=1}^{\nb} (\epsilon_k-\eta_k)^2.\label{eq:distance_deux_fonctions_ep}
\end{equation}
As the series $\epsilon_k$ and $\eta_k$ have values in $\{0,1\}$, the quantity 
\[\rho(\epsilon,\eta):=\sum_{k=1}^{\nb}\nbOne_{\{\epsilon_k\neq  \eta_k\}}=\sum_{k=1}^{\nb} (\epsilon_k-\eta_k)^2\]
 is the Hamming distance between $\eta$ and $\epsilon$. 
To apply Lemma \ref{lem_minoration_L2}, we need that, $\forall \eta\neq \epsilon$,  
\[ \norm{\lambda_{\epsilon}-\lambda_{\eta}}_{L^2}^2\geq 4C'^2 h_n^{2\alpha} \text{ and consequently }\rho(\epsilon,\eta)\geq Ch_n^{-1}.\] 
This is not the case if we take the whole $\rond{G}_n$ (the minimal Hamming distance between two vectors $\epsilon$ and $\eta$ is 1). We need to extract a sub-series of functions. 
According to \citet[Lemma 2.7]{tsybakov} (bound of Varshamov-Gilbert), it is possible to extract a family $(\epsilon_{(0)},\ldots,\epsilon_{(P_n)})$ of the set $\Omega=\{0,1\}^{\nb}$ such that $\epsilon_{(0)}=(0,\ldots,0)$ and 
 \[\forall\:\;  0\leq j<k\leq P_n,\quad\rho(\epsilon_{(j)},\epsilon_{(k)})\geq \nb/8, \quad \textrm{and}\quad P_n\geq 2^{\nb/8}.\]
As $p_n\geq n^{1/(2\alpha+1)}$, 
\begin{equation}
\ln(P_n)\geq \ln(2)n^{1/(2\alpha+1)}/8 .\label{ln_Pn}
\end{equation}
 We define 
 \[\lambda_{j}:=\lambda_{\epsilon_{(j)}}\quad \textrm{and}\quad \rond{H}_n=\{\lambda_0,\lambda_1,\ldots,\lambda_{P_n}\}.\]

 Then, for any $\lambda_j,\lambda_k\in\rond{H}_n$, if $j\neq k$ , as $p_n=\lceil1/h_n\rceil $, by \eqref{eq:distance_deux_fonctions_ep},
 \[\norm{\lambda_j-\lambda_k}_{L^2}^2\geq \cons^2 \norm{K}_{L^2}^2h_n^{2\alpha+1}\nb/8\geq\cons^2\norm{K}_{L^2}^2h_n^{2\alpha}/8.\] 
This is exactly the expected lower bound if we take $C'=\cons\norm{K}_{L^2}/(4\sqrt{2})$. 
 
\paragraph*{Step 2: Functions $\lambda_j$ belong to $\Fcb$.}

We already know that $\lambda_0$ belongs to $\Fcb$. Let us first compute the norm of $\lambda_j$ on $H^{\alpha}(\mathcal{J})$.
 Let us set $\rg=\lfloor\alpha\rfloor$.   We have that  $(K(./h_n))^{(\rg)}=h_n^{-\rg} K^{(\rg)}(./h_n)$. We compute the modulus of smoothness: 
\begin{align*}
\omega(\varphi_k^{(\rg)},t)_{\infty}&=a\omega\left( 
h_n^{\alpha}\left(
K\left(\frac{.-x_k}{h_n}\right)\right)^{(\rg)},t
\right)_{\infty}
=ah_n^{\alpha} \omega\left(h_n^{-\rg}K^{(\rg)}\left(\frac{.-x_k}{h_n}\right),t\right)_{\infty}\\
&=ah_n^{\alpha-\rg} \omega\left(K^{(\rg)},\frac{t}{h_n}\right)_{\infty}
\end{align*}
and 
\begin{align*}
|\varphi_{k}\vert_{H^{\alpha}}&=\sup_{t>0}t^{\rg-\alpha} \omega(\varphi_k^{(\rg)},t)_{\infty}=\cons\sup_{t>0} t^{\rg-\alpha}h_n^{\alpha-\rg} \omega\left(K^{(\rg)},\frac{t}{h_n}\right)_{\infty}\\
&=a\sup_{z>0} z^{\rg-\alpha}\omega(K^{(\rg)},z)=\cons\vert K\vert_{H^\alpha}
\end{align*} 
by the change of variable $z=t/h_n$. 
The functions $\varphi_k$ have disjoint supports. For any $(x,y)\in\mathcal{J}$, there exists $(i,j)$ such that $x\in[x_i,x_{i+1}($ and $y\in[x_j,x_{j+1}($. 
Then 
\[\lambda_k^{(\rg)}(x)-\lambda_k^{(\rg)}(y)= \ep_i\left(\varphi_i^{(\rg)}(x)-\varphi_i^{(\rg)}(y)\right)+\ep_j\left(\varphi_j^{(\rg)}(x)-\varphi_j^{(\rg)}(y)\right).\]
Therefore 
\[\omega(\lambda_k^{(\rg)},t)_{\infty}\leq \sup_{i,j} \left(\omega(\varphi_{j}^{(\rg)},t)_{\infty}+\omega(\varphi_i^{(\rg)},t)_{\infty}\right)\leq 2\omega (\varphi_1^{(\rg)},t)_{\infty}\]
and
$\vert \lambda_k\vert _{H^{\alpha}(\mathcal{J})}\leq 2a\vert K\vert_{H^{\alpha}}$.
Moreover, 
\[\norm{\lambda_k}_{L^\infty (\mathcal{J})}\leq \norm{\lambda_0}_{L^\infty (\mathcal{J})}+ah_n^{\alpha}\norm{K}_{L^\infty}\leq \norm{\lambda_0}_{L^\infty (\mathcal{J})}+2a\norm{K}_{L^\infty }\] and consequently $\norm{\lambda_k}_{H^{\alpha}(\mathcal{J})}\leq \norm{\lambda_0}_{H^{\alpha}(\mathcal{J})}+2a\norm{K}_{H^{\alpha}}$.
Then $\lambda_k\in H^{\alpha}(\mathcal{J},\Lg)$ for $a$ sufficiently small. It remains to check that $\lambda_k\in\Fcb$. For any $0\leq k\leq P_n$:
\begin{enumerate}
\item As $K$ is non-negative, $\forall x \geq i_1$, $\lambda_k(x)\geq \mathbf{a}\frac{x^{b}}{(b+1)\mathbf{m}(x)}$. 
\item  $ \norm{\lambda_k}_{\holder(\mathcal{J})}\leq \Lg$ for $a$ small enough. 
\item $\int_0^{i_1}\lambda_k(u)\Mg(u)du=0\leq \llg$. 
\end{enumerate}
Therefore $\lambda_k\in\Fcb$ for $a$ small enough. 
\paragraph*{Step 3: Absolute continuity.}
We denote by $\pj$ the transition densities  $\pl[\lambda_j]$. As $(Z_0,Y_1,Z_1,\ldots,Y_n,Z_n)$ is a Markov process,
\[\mathbf{P}^{\lambda_j}(z_0,dy_1,dz_1\ldots,dy_n,dz_n)=\pj (z_0,y_1)Q(y _1,dz_1)\ldots \pj(z_{n-1},y_n)Q(y_n,dz_n)dy_1 \ldots dy_n .\]
By  (\ref{densitet1}),  we can rewrite: 
$\po (x ,y)=A_{x,y} \exp (-\tilde{A}_{x ,y})$ where 
\begin{equation}A_{x,y}:=\lambda_0 (y)(\phi_x^{-1})'(y)\nbOne_{\{y\geq x\}},\quad  \Au:=\int_{x}^{y}\lambda_0 (u)(\phi_x^{-1})'(u)du \label{eq_definition_A}\end{equation}
and 
$\pj (x ,y)=(\A+\B ) \exp (-\Au-\Bu)$
where  $\B=\sum_{k=1}^m \epsilon_k B^k_{x,y}$, $\Bu=\sum_{k=1}^{\nb} \epsilon_k \tilde{B}^{k}_{x,y}$ and 
\begin{equation}
B_{x,y}^{k}:= \varphi_k(x)(\phi_x^{-1})'(y)\nbOne_{\{y\geq x\}},
 \quad  \tilde{B}_{x ,y}^{k}:=\int_{x}^{y} \varphi_k(u)(\phi_x^{-1})'(u)du
.\label{eq_def_B}\end{equation}
The probability density $\mathbf{P}^{\lambda_0}$ is null if one of the $Q(y_i,dz_i)$ is null, if one of the indicator function  $\nbOne_{\{y_{i+1}\geq z_i\}}=0$, or if one $y_i$ is smaller than $i_1$; then $\mathbf{P}^{\lambda_j}$ is absolutely continuous with
respect to $\mathbf{P}^{\lambda_0}$.

\paragraph*{Step 4: The $\chi^2$ divergence.}

As $\mathbf{P}^{\lambda_0}, \mathbf{P}^{\lambda_j}$ are equivalent measures,  we have: 
\[ 
\chi^2 (\mathbf{P}^{\lambda_j}, \mathbf{P}^{\lambda_0})=\int \left(\frac{d \mathbf{P}^{\lambda_j}}{d \mathbf{P}^{\lambda_0}}\right)^2 d \mathbf{P}^{\lambda_0} -1.
\]
\stepcounter{numero}
Let us set $E_{\arabic{numero}}:=\chi^2 (\mathbf{P}^{\lambda_j}, \mathbf{P}^{\lambda_0})+1$. We can write: 

\begin{multline*}
E_{\arabic{numero}}=\int_{(\R^+)^n}  \left(\frac{\pj (z_0 ,y_1)...\pj(z_{n-1} ,y_n)}{\po (z_0 ,y_1)...\po (z_{n-1} ,y_n)}\right)^2 \po (z_0 ,y_1)...\po (z_{n-1} ,y_n)\\
\times \int_{(\R^+)^{n}}Q(y_1,dz_1)\ldots Q(y_n,dz_n)dy_1\ldots dy_n.
\end{multline*}
As $Q$ is the transition density, for any $y_n$, $\int_{\R^+} Q(y_n,dz_n)=1$. Moreover,  as $\int_{\R^+}\po(x,y)dy=\int_{\R^+}\pj(x,y)dy=1$,  
\begin{align}
E_{\arabic{numero}}
&=\int_{(\R^+)^{2(n-1)}} \frac{(\pj (z_0 ,y_1)...\pj (z_{n-2} ,y_{n-1}))^2}{\po (z_0 ,y_1)...\po (z_{n-2} ,y_{n-1})}
Q(y_1,dz_1)\ldots Q(y_{n-1},dz_{n-1})dy_1\ldots dy_{n-1}  \nonumber\\
 &\times \int_{\R^+}\frac{(\pj (z_{n-1} ,y_{n}))^2}{\po (z_{n-1} ,y_{n})}dy_n.
\label{chi2_L2}
\end{align}

This expression of the $\chi^2$ divergence enables us to approximate it more closely. 
Let us set 
\begin{align}
DP&:= \int_{\R^+} \frac{\left(\pj(x,y)\right)^2}{\po(x,y)}dy-1=\int_{\R^+} \left(\frac{\pj (x ,y)}{\po (x ,y)}-1\right)^2 \po (x ,y)dy\nonumber\\
&= \int_{\R^+} \left(
\left(1+\frac{B_{x,y}}{A_{x,y}}\right)
\exp\left(-\tilde{B}_{x,y}\right)- 1
\right)^2
A_{x,y}\exp\left(-\tilde{A}_{x,y}\right)dy.
\label{def_DP}\end{align}
As the support of $\varphi_k$ is included in $\mathcal{J}$, we can remark that $B_{x,y}$ is null on $\mathcal{J}^c$ and 
\[
\tilde{B}_{x,y}^{k} =\int_{[x,y]\cap \mathcal{J}}  \varphi_k(u)(\phi_{x}^{-1})'(u)du.
\] 
We bound the $\chi^2$-divergence differently on $\mathcal{J}$ and $\mathcal{J}^c$: $DP=R_1+R_2$ where 
\begin{align*}
R_1
&:= \int_{\mathcal{J}} \left(\left(1+\frac{\B}{\A}\right) \exp \left(-\Bu\right) -1\right)^2 \A \exp (-\Au) dy,\\ 
R_2&
:= \int_{\mathcal{J}^c} \left( \exp \left(-\Bu\right) -1\right)^2 \A \exp (-\Au) dy.
\end{align*}
We have that 
$B_{x,y}^k\leq  \mathbf{M}(y) \norm{\varphi_k}_{\infty}  \nbOne_{\{y\geq x\}}\nbOne_{\{y\in [x_k,x_{k+1}(\}}$
and therefore, as $\norm{\varphi_k}_{\infty}=ah_n^{\alpha}\norm{K}_{\infty}$, 
\begin{equation}
B_{x,y}\leq  \sup_{y\in \mathcal{J}} \mathbf{M}(y) ah_n^{\alpha}\norm{K}_{\infty}\nbOne_{\{y\in \mathcal{J}\}}\leq C ah_n^{\alpha}\nbOne_{\{y\in \mathcal{J}\}}.\label{eq_bound_B_xy}
\end{equation}
By \eqref{eq_def_B}, we obtain, as the functions $\varphi_k$ are supported in $\mathcal{J}$: 
\[
\tilde{B}_{x,y}^{k} = \int_x^{y}
\varphi_k(z)(\phi_x^{-1})'(z)dz
\leq \sup_{z\in \mathcal{J}}(\Mg(z)) \sup_{k} \norm{\varphi_k}_{L^1}
\leq ah_n^{\alpha+1}\sup_{z\in \mathcal{J}}(\Mg(z))
\]
and, as $\nb=\lceil1/h_n\rceil$,
\begin{equation}
    \tilde{B}_{x,y}\leq \sum_{k=1}^{p_n} \tilde{B}_{x,y}^k\leq C'ap_nh_n^{\alpha+1}\leq C'a h_n^{\alpha}.\label{eq_bound_tilde_Bxy}
\end{equation}
Then  by \eqref{eq_bound_tilde_Bxy} and as $\int_{\R^+} \A \exp (-\Au) dy=1$ \[R_2\leq \int_{\R^+} O( a^2h_n^{2\alpha})A_{x,y}\exp\left(-\tilde{A}_{x,y}\right)dy=O (a^2h_n^{2\alpha}).\]
As $\lambda_0=\ep$ on $\mathcal{J}$,   we get by \eqref{eq_definition_A} that 
\[\ep \units{y\geq x}\inf_{y\in \mathcal{J}}\mg(y)\leq \sup_{y\in\mathcal{J}}\A\leq \sup_{y\in \mathcal{J}}\Mg(y)\ep\units{y\geq x}.\]
Moreover, on $\R^+$, $\exp(-\Au)\leq 1$. 
Then by \eqref{eq_bound_B_xy} and \eqref{eq_bound_tilde_Bxy}, we get that \[R_1=  \int_{\mathcal{J}} \left( \left( 1+O(ah_n^{\alpha}\right)\exp\left(-O(ah_n^{\alpha})\right)-1\right)^2 O(1)dy=\int_{\mathcal{J}} O(a^2h_n^{2\alpha}) dy = O(a^2 h_n^{2\alpha}).\]
Therefore $DP=O(a^2h_n^{2\alpha})$ and, by  \eqref{chi2_L2} and \eqref{def_DP}, we get by recurrence
\begin{align*}
E_{\arabic{numero}}&=\int_{(\R^+)^{2(n-1)}}  \frac{(\pj (z_0 ,y_1)...\pj (z_{n-2} ,y_{n-1}))^2}{\po (z_0 ,y_1)...\po (z_{n-2} ,y_{n-1})}Q(y_1,dz_1)\ldots Q(Y_{n-1},dz_{n-1})d y_{1} ...dy_{n-1}\\
&\times\left(\OO{  \cons^2  h_n^{2\alpha}}+1\right)  \\
&=\prod_{i=1}^n\left(\OO{ \cons^2  h_n^{2\alpha }}+1\right)  =1+\cons^2 n\OO{h_n^{2\alpha}}.
\end{align*}
As $h_n= n^{-\frac{1}{2\alpha+1}}$, 
\[\chi^2 (\mathbf{P}^{\lambda_0}, \mathbf{P}^{\lambda_j})=E_{\arabic{numero}}-1\leq \cons^2\OO{n^{1/(2\alpha+1)}}.\]
By \eqref{ln_Pn}, $\ln(P_n)\geq \ln(2) n^{1/(2\alpha+1)}/8$ and therefore,
 \[\frac{1}{P_n}\sum_{k=1}^{P_n}\chi^2 (\mathbf{P}^{\lambda_0}, \mathbf{P}^{\lambda_j})= \cons^2\OO{ n^{1/(2\alpha+1)}}=
a^2\OO{\ln(P_n)}\leq \ln(P_n)/8\] 
for $a$ small enough, which concludes the proof.
\fussy

\section*{Acknowledgements}
N. Krell  was partly supported by the Agence Nationale de la Recherche
PIECE 12-JS01-0006-01. 
The research of E. Schmisser was supported in part by the Labex CEMPI (ANR-11-LABX-0007-01)

\appendix

\section{Technical proofs and results}
\sloppy

\subsection{Proof of Lemma \ref{cor_majoration_variance}}
We consider a function $\ft$ such that $\norm{\ft}_{\infty}=1$; we obtain the expected result by dividing $\ft$ by its $L^{\infty}$-norm. 
According to Assumption A\ref{hypo_contraction},
\[\left\vert \EE[z_0]{\frac{1}{n}\sum_{k=1}^n \ft(Y_k,Z_k)-
\int \ft(y,z)\rhol(dy,dz)}\right\vert 
\leq \frac{1}{n} \sum_{k=1}^n R\Vl(z_0)\gamma^k\leq \frac{R\Vl(z_0)}{n(1-\gamma)} \]
which proves the first inequality. Let us set  $\tilde{\ft}(Y,Z)=\ft(Y,Z)-\EE[z_0]{\ft(Y,Z)} $. We have: 
\begin{equation}
\EE[z_0]{\left(\frac{1}{n}\sum_{k=1}^n\tilde{\ft}(Y_k,Z_k)\right)^2}
=\frac{1}{n^2}\sum_{k} \EE[z_0]{\tilde{\ft}^2(Y_k,Z_k)}+\frac{2}{n^2}\sum_{k<k'} \EE[z_0]{\tilde{\ft}(Y_k,Z_k)\tilde{\ft}(Y_{k'},Z_{k'})}.
\label{eq:cor_majoration_variance_decomposition_variance}
\end{equation}
 We notice that: 
\begin{align*}
\EE[z_0]{ \tilde{\ft}^2(Y_k,Z_k) }
&= \EE[z_0]{ \ft^2(Y_k,Z_k)}-\left(\EE[z_0]{ \ft(Y_k,Z_k)}\right)^{2}\leq   \EE[z_0]{ \ft^2(Y_k,Z_k)}\\ 
&\leq \int \ft^2(y,z)\rhol(dy,dz)+\EE[z_0]{ \ft^2(Y_k,Z_k)}-\int \ft^2(y,z)\rhol(dy,dz)\\
&\leq \int \ft^2(y,z)\rhol(dy,dz)+ R\Vl(z_0)\gamma^k 
\end{align*}
by Assumption A\ref{hypo_contraction}.
Therefore
\begin{equation}
\EE[z_0]{\frac{1}{n^2}\sum_{k=1}^n \tilde{\ft}^2(Y_k,Z_k)}
\leq \frac{1}{n} \int \ft^2(y,z)\rhol(dy,dz)
+ \frac{R\Vl(z_0)}{n^2(1-\gamma)}.
\label{eq:cor_majoration_variance_maj_t2}
\end{equation}
Let us bound the last term of \eqref{eq:cor_majoration_variance_decomposition_variance}. We can remark that $(Z_0,Y_1,Z_1,\ldots,Y_k,Z_k,\ldots)$ is an inhomogeneous Markov chain. Therefore, for any $(k<k')$, $\Ec{s(Y_{k'},Z_{k'})}{Y_k,Z_k}=\Ec{s(Y_{k'},Z_{k'})}{Z_k}$  and by Assumption A\ref{hypo_contraction},
\begin{align}\left\vert \Ec{\tilde{\ft}(Y_{k'},Z_{k'})}{ Y_k,Z_k} \right\vert 
&\leq  \left\vert \Ec{ \ft(Y_{k'},Z_{k'})}{Y_k,Z_k}-\int \ft(y,z)\rhol(dy,dz)\right\vert \nonumber\\
&+\left\vert -\EE[z_0]{\ft(Y_{k'},Z_{k'})}+\int \ft(y,z)\rhol(dy,dz) \right\vert \nonumber \\
&\leq 
 R\gamma^{k'-k}\Vl(Z_k)+ R\Vl(z_0)\gamma^{k'}\label{eq:bound_E_Zkl} .
\end{align}
Then 
\begin{align*}
E_k&:=\sum_{k'=k+1}^n \vert  \EE[z_0]{\tilde{\ft}(Y_k,Z_k)\tilde{\ft}(Y_{k'},Z_{k'})}\vert =
\sum_{k'=k+1}^n \vert \EE[z_0]{\tilde{\ft}(Y_k,Z_k)\Ec{\tilde{\ft}(Y_{k'},Z_{k'})}{Y_k,Z_k}}\vert \\
&\leq \sum_{k'=k+1}^{n} \EE[z_0]{\vert \tilde{\ft}(Y_k,Z_k)\vert (R\gamma^{k'-k}\Vl(Z_k)+R\Vl(z_0)\gamma^{k'})} \\
&\leq \frac{R}{1-\gamma}  \EE[z_0]{\vert \tilde{\ft}(Y_k,Z_k)\vert \left(\Vl(Z_k)+  \gamma^k\Vl(z_0)\right)}. 
\end{align*}
As $\vert \tilde{\ft}(Y_k,Z_k)\vert \leq \vert \ft(Y_k,Z_k)\vert + \EE[z_0]{\vert \ft(Y_k,Z_k)\vert }$, 
\begin{align*}
E_k&\leq \frac{R}{1-\gamma}   \left(\EE[z_0]{\vert \ft(Y_k,Z_k)\vert \Vl(Z_k)}
+ \EE[z_0]{ \vert \ft(Y_k,Z_k)\vert } \EE[z_0]{\Vl(Z_k)} \right)\\
&+  \frac{R\Vl(z_0)}{1-\gamma}\gamma^k2\EE[z_0]{\vert \ft(Y_k,Z_k)\vert }.
\end{align*}
By Assumption A\ref{hypo_contraction}, for any function $\vert\psi\vert\leq \Vl$, \[ \EE[z_0]{\psi(Y_k,Z_k)}\leq \ \int_0^{\infty} \psi(y,z) \rhol(dy,dz) + R\Vl(z_0)\gamma^k.\] Then 
\begin{align*}
\sum_{k} E_k&\leq
\sum_k \frac{R}{1-\gamma} 
\left(\int \vert \ft(y,z)\vert \Vl(z)\rhol(dy,dz) 
+ R\Vl(z_0)\gamma^k\right)\\
&+ \sum_k \frac{R}{1-\gamma}
\left(
 \int \vert \ft(y,z)\vert \rhol(dy,dz)
+R\Vl(z_0)\gamma^k
\right)\\
&\times 
\left(
\int \Vl(z)\rhol(dy,dz)+R\Vl(z_0)\gamma^k
\right)\\
&+ \sum_k 
\frac{R\Vl(z_0)}{1-\gamma} 2\gamma^k
\left( 
\int \vert \ft(y,z)\vert \rhol(dy,dz)+R\Vl(z_0)\gamma^k
\right)\\
&\leq n \frac{R}{1-\gamma}
 \left( 
 \int \vert \ft(y,z)\vert \Vl(z)\rhol(dy,dz) + 
 \int \vert  \ft(y,z)\vert \rhol(dy,dz)\int \Vl(z)\mu(dz) \right)\\
& + \frac{R\Vl(z_0)}{(1-\gamma)^2}\left( R+R\int \Vl(z)\mul(dz)+(R+2)\int \vert \ft(y,z)\vert \rhol(dy,dz) \right)\\
&+ \frac{R^2(R+2)\Vl^2(z_0)}{(1-\gamma)(1-\gamma^2)} .
\end{align*}
By \eqref{eq:cor_majoration_variance_decomposition_variance} and \eqref{eq:cor_majoration_variance_maj_t2}, we get: 
\begin{align*}
\EE[z_0]{
\left (
\frac{1}{n}\sum_{k=1}^n \tilde{\ft}(Y_k,Z_k)
\right)^2}
 &\leq 
 \frac{1}{n}\left(
 \int \ft^2(y,z)\rhol(dy,dz)+ 
 \int \vert \ft(y,z)\vert G_{\lambda}(y,z)\rhol(dy,dz)\right)\\
 &+\frac{C}{n^2}
\end{align*}
 where $C$ depends only on $R$, $\Vl$ and $\gamma$ and we recall that  $G_{\lambda}(z)=\frac{R}{1-\gamma} \left(\Vl(z)+\int \Vl(u)\mu(du)\right)$. 
 
 \subsection{Proof of Lemma \ref{cor_betageometricmixing}}

Let $G$ be an event of $\rond{O}_{0}^k\times \rond{O}_{t+k}^{\infty}$. 
Then $G$ is a disjoint reunion of events $E^i\cap F^{i,j}$ where 
\begin{align*}
E^i&=\{(Y_1,Z_1)\in J_1^i,\ldots,(Y_k,Z_k)\in J_k^i\},\\
F^{i,j}&=\{(Y_{t+k},Z_{t+k})\in I_0^{i,j},\ldots,(Y_{t+k+\n},Z_{t+k+\n})\in I_{\n}^{i,j}\}
\end{align*}
with $J_j^i$ and $I_l^{i,j}$ subsets of $(\R^+)^2$ and $1\leq \n<\infty$. 
Then 
\[D_G:=P_{ \rond{O}_{0}^k,\rond{O}_{t+k}^{\infty}}(G)-P_{ \rond{O}_{0}^k}\otimes P_{\rond{O}_{t+k}^{\infty}}(G)=\sum_{i,j} \PP{E^i\cap F^{i,j}}-\PP{E^i}\PP{F^{i,j}}.\] 
As $(Y_k,Z_k)_{k\in\N}$ is a Markov chain, 
\begin{align*}
A_{i,j}&:=\PP{E^i\cap F^{i,j}}-\PP{E^i}\PP{F^{i,j}}\\
&=\int_{J_1^i\times\ldots\times J_k^i} \pl(z_0,dy_1)Q(y_1,dz_1)\ldots \pl(z_{k-1},dy_k)Q(y_k,dz_k)\\
&\times \int_{I_0^{i,j}} \left(\Pc{Y_{t+k}\in dy'_0,Z_{t+k}\in dz'_0}{Z_k= z_k}-\Pc{Y_{t+k}\in dy'_0, Z_{t+k}\in dz'_0}{Z_0=z_0}\right)  \\
&\times \int_{I_1^{i,j}\times\ldots\times I_{\n}^{i,j}} \pl(z'_0,dy'_1)Q(y'_1,dz'_1)\ldots \pl(z'_{\n-1},dy'_{\n})Q(y'_{\n},dz'_{\n}).
\end{align*}
To simplify the notations, let us set 
\[
    \rl[k](x,dy_1,dz_1,\ldots,dy_k,dz_k):=\pl(x,dy_1)Q(y_1,dz)\ldots \pl(z_{k-1},y_k)Q(y_k,z_k).
\]
and $\rl^t(x,dy,dz)=\Pc{Y_{t}\in dy,Z_t\in dz}{Z_0=x}$. 
Then 
\begin{align*}
A_{i,j}
&=\int_{J_1^i\times\ldots\times J_k^i} \rl[k](z_0,dy_1,dz_1,\ldots,dy_k,dz_k)\\
&\times \int_{I_0^{i,j}} \left(\rl^t(z_k,dy'_0,dz'_0)-\rl^{t+k}(z_0,dy'_0,dz'_0)\right) \nbOne_{y'_0,z'_0\in I_0^{i,j}} \\
&\times \int_{I_1^{i,j}\times\ldots\times I_{\n}^{i,j}} \rl[\n](z'_0,dy'_1,dz'_1,\ldots,dy'_{\n},dz'_{\n}).
\end{align*}
We regroup the $F^{i,j}$: 
\begin{align*}
A_i&:=\sum_j A_{i,j}=\sum_j\PP{E^i\cap F^{i,j}}-\PP{E^i}\PP{F^{i,j}}\\&=
\int_{J_1^i\times\ldots\times J_k^i} \rl[k](z_0,dy_1,dz_1,\ldots,dy_k,dz_k)
\left(\EE[z_k]{\psi(Y_{t},Z_{t})}-\EE[z_0]{\psi(Y_{t+k},Z_{t+k})}\right)
\end{align*}
where $\psi(y,z):=\sum_j \nbOne_{(y,z)\in I_0^{i,j}} 
\int_{I_1^{i,j}\times\ldots\times I_{\n}^{i,j}} \rl[\n](z,dy'_1,dz'_1,\ldots,dy'_{\n},dz'_{\n})
$.
We can remark that 
$\psi(y,z)= \sum_j \nbOne_{y,z\in I_0^{i,j}} \PP[z]{(Y_1,Z_1)\in I_1^{i,j},\ldots,(Y_{\n},Z_{\n})\in I_{\n}^{i,j}}$
and by the law of total probability, $\psi(y,z)\leq 1$. We can apply Assumption A\ref{hypo_contraction} to the function $\psi$:
\[
\left\vert \EE[z_k]{\psi(Y_{t},Z_{t})}-\EE[z_0]{\psi(Y_{t+k},Z_{t+k})} \right\vert\leq  R\gamma^{t} \Vl(z_k)+R\gamma^{t+k} \Vl(z_0).
\]
 Then 
\begin{align*}
\vert D_G\vert &=\left\vert \sum_i A_i\right\vert \leq R \gamma^t \sum_i \int_{J_1^i\times\ldots\times J_k^i} \rl[k](z_0,dy_1,dz_1,\ldots,dy_k,dz_k) \left( \Vl(z_k)+\Vl(z_0)\right)\\
&\leq R\gamma^t \sum_i \EE[z_0]{\left(\Vl(Z_k)+\Vl(z_0)\right)\nbOne_{(Y_1,Z_1)\in J_1^i,\ldots,(Y_k,Z_k)\in J_k^i}}\\
&\leq R\gamma^t \left(\EE[z_0]{\Vl(Z_k)}+\Vl(z_0)\right).
\end{align*}
By Lemma \ref{cor_majoration_variance}, 
\[\EE[z_0]{\Vl(Z_k)}\leq \int \Vl(z)\mul(dz)+R\gamma^k \Vl(z_0).\]
Therefore
\[\beta_{Y,Z}(t)=\sup_k\sup_{G\in\rond{O}_0^k\times\rond{O}_{t+k}^{\infty}} \vert D_G\vert \leq R\gamma^t\left(\int \Vl(z)\mul(dz)+ (1+R)\Vl(z_0)\right).\]
As $\gamma<1$,
\[ \beta_{Y,Z}(t)\leq ce^{-\beta t}\]
with $\beta=-\ln(\gamma)$, $c=R\left( \int \Vl(z)\mul(dz)+ (1+R)\Vl(z_0)\right)$.

\subsection{Proof of Lemma \ref{hypo}}

\subsubsection{Assumption A\ref{hypo_contraction} is satisfied}

Assumption (S)\ref{hypo_lambda_borne} implies Assumption A\ref{hypo_contraction}\ref{hypo_contraction_pl_densite}. To prove Assumption A\ref{hypo_contraction}\ref{hypo_contraction_ergodicite} and \ref{hypo_contraction_V}, in analogy with \cite{kre2}, we apply the following result, 
which is Theorem 1.1 of \citet{baxendale} written for a Markov chain on $\R^2$ and a finite measure instead of a probability. 
\begin{result}[Sufficient conditions for ergodicity]
Let us consider $(Y_k,Z_k)_{k\geq 1}$  an homogeneous Markov chain on $(\R^2,\rond{B}(\R^2))$ with transition probability $\tilde{R}$. Under the following three conditions,
\begin{description}
\item[Minorization condition] There exist a set $\rond{C}\subset\R^2$ and a finite measure $\mathbf{s}$ such that $\forall (y_1,z_1)\in \rond{C},\forall A\in \rond{B}(\R^2)$, \[\int_A \tilde{R}(y_1,z_1,dy,dz)\geq \mathbf{s}(A). \]
\item[Strong aperiodicity condition] $\mathbf{s}(\rond{C})>0$. 
\item[Drift condition] There exists a function $\mathbf{V}:\R^2\to [1,\infty[$ and two constants $c<1$, $K>0$ such that 
\[ \EE[y_1,z_1]{\mathbf{V}(Y_2,Z_2)}\leq c \mathbf{V}(y_1,z_1)\units{(y_1,z_1)\in \rond{C}^c}+K\units{(y_1,z_1)\in\rond{C}}.\]
\end{description}
 Then the process $\{(Y_k,Z_k) : k \geq 0\}$ is recurrent positive and strongly ergodic, and has a unique stationary
probability measure $\rho$. 

Moreover, there exists $\gamma$ and $R$ depending only on $\mathbf{s}$, $c$ and $K$ such that, for any function $\psi\leq \mathbf{V}$,  
\[\left\vert \Ec{\psi(Y_k,Z_k)}{Z_0=z_{0}}-\EE[\rhol] {\psi(Y_1,Z_1)}\right\vert \leq R\mathbf{V}(z_{0})\gamma^k.\]
\end{result}

Then Assumptions A\ref{hypo_contraction}\ref{hypo_contraction_ergodicite} and \ref{hypo_contraction_V} are satisfied.

Let us check that its three conditions (minorization, strong aperiodicity and drift) are satisfied.
We need to control the transition density. As $(Z_0,Y_1,Z_1,\ldots,)$ is an (inhomogeneous) Markov chain, let us note \[ \rtl(y_1,z_1,dy,dz)=\Pc{Y_2\in dy,Z_2\in dz}{Z_1=z_1,Y_1=y_1}=\pl(z_1,dy)Q(y,dz).\]
Let us set 
\[i'_1=\max\left( i_1,\left( \frac{1}{\mathbf{a}(1-\kappa^{b+1})}\ln\left( \frac{2\kappa^{b+1}}{1-\kappa^{b+1}}\right)\right)^{1/(b+1)}\nbOne_{\{\kappa^{b+1}\geq 1/3\}}\right).\]
\begin{description}
\item[Minorization condition] Let us set $\rond{C}=\R^+\times [0,i'_1]$.    
For any $(y_1,z_1)\in\rond{C}$, any $A\subseteq(\R^+)^2$, by 
Assumption (S)\ref{hypo_phi_borne}, we have that 
\begin{align*}
\int_{A}\rtl(y_1,z_1,dy,dz)&=\int_A \lambda(y)(\phi_{z_1}^{-1})'(y)\exp\left(-\int_{z_1}^y \lambda(u)(\phi_{z_1}^{-1})'(u)du\right)\nbOne_{\{z_1\leq y\}} Q(y,dz) dy\\
&\geq \int_A \lambda(y)\mg(y)\exp\left(-\int_0^y\lambda(u)\Mg(u)du\right)\nbOne_{\{i'_1\leq y\}}Q(y,dz)dy
\end{align*}

By 
Assumption (S)\ref{hypo_lambda_borne} and \ref{hypo_lambda_minore},  for any $\lambda\in\Fcb$,
\begin{align*}
\int_{A}\rtl(y_1,z_1,dy,dz)
&\geq \int_{A\cap [i'_1,i'_2]\times[0,i'_1]} \lambda(y)\mg(y)\exp\left(-\int_0^y\lambda(u)\Mg(u)du\right)Q(y,dz)dy\\
    & \geq \int_{A\cap [i'_1,i'_2]\times[0,i'_1]} \mathbf{a} \frac{y^b}{b+1}\exp\left(-\llg -\Lg \int_{i_1}^{y} \Mg(u)du\right)Q(y,dz)dy\\
    &=:\mathbf{s}_{\cf}(A)
\end{align*}
and $\mathbf{s}_{\cf}(A)$ is a finite measure. 
\item[Strong aperiodicity condition] 
\begin{align*}
    \mathbf{s}_{\cf}(\rond{C})&=\int_{i'_1}^{i'_2} \mathbf{a} \frac{y^b}{b+1} \exp\left(-\llg-\Lg\int_{i_1}^y\Mg(u)du\right)\int_0^{i'_1}Q(y,dz)dy\\
    &\geq \int_{i'_1}^{(\kappa^{-1}i'_1)\wedge i'_2}\mathbf{a} \frac{y^b}{b+1} \exp\left(-\llg-\Lg \int_{i_1}^{i'_2}\Mg(u)du\right)\int_0^{i'_1}Q(y,dz)dy.
\end{align*}
For any $y\leq \kappa^{-1}i'_1$, by Assumption (S) \ref{hypo_Q}, $\int_0^{i'_1}Q(y,dz)=\Pc{Z_1\leq i'_1}{Y_1=y}=1$. Therefore 
\[\mathbf{s}_{\cf}(\rond{C})\geq \int_{i'_1}^{(\kappa^{-1}i'_1)\wedge i'_2} \mathbf{a} \frac{y^b}{b+1}\exp\left(-\llg -\Lg\int_{i_1}^{i'_2}\Mg(u)du\right).\]
Then 
$\mathbf{s}_{\cf}(\rond{C})>0$.

\item[Drift condition] 
For any $(y_1,z_1)$, as $(Y_1,Z_1,Y_2,Z_2)$ is an (inhomogeneous) Markov chain, \begin{align*}\Ec{\Vb(Z_2)}{Y_1=y_1,Z_1=z_1}&=\EE[z_1]{\Vb(Z_2)}\\
&=\int_0^{\infty} \Pc{\Vb(Z_2)> z}{Z_1=z_1}dz
\end{align*}
where $\Vb(z)=\exp\left(\mathbf{a}z^{b+1}\right)$. By Assumption (S)\ref{hypo_Q}, as $\Vb$ is an increasing function, 
$\Vb(Z_2)>z \iff Z_2\geq \Vb^{-1}(z) \implies Y_2\geq \kappa^{-1} \Vb^{-1}(z)$.
Then by \eqref{fonction_repartition_t1} and Assumption (S)\ref{hypo_phi_borne}, 
\begin{align}
    \EE[z_1]{\Vb(Z_2)}&\leq \int_{0}^{\infty} \PP[z_1]{Y_2\geq \kappa^{-1}\Vb^{-1}(z)}dz \nonumber\\
    &= \int_0^{\infty} \exp\left(-\int_{z_1}^{\kappa^{-1}\Vb^{-1}(z)} (\phi_{z_1}^{-1})'(u)\lambda(u)du\right) \nbOne_{\{\kappa^{-1}\Vb^{-1}(z)\geq z_1\}}dz\nonumber\\
    &\leq \int_0^{\infty} \exp\left(-\int_{z_1}^{\kappa^{-1}\Vb^{-1}(z)} \mg(u)\lambda(u)du\right) \nbOne_{\{\kappa^{-1}\Vb^{-1}(z)\geq z_1\}}dz=:\tilde{I}(z_1)\label{eq_bound_vb}
    \end{align}
    Let us make the change of variable $y=\kappa^{-1}\Vb^{-1}(z)$, then $dz=\kappa \Vb'(\kappa y)dy$ and
    \[\tilde{I}(z_1)= \kappa \int_0^{\infty} \Vb'(\kappa y) \exp\left(-\int_{z_1}^{y} \mg(u)\lambda(u)du\right) \nbOne_{\{y\geq z_1\}}dy.\]
    Let us first bound this quantity for $z_1\geq i_1$. By Assumption (S)\ref{hypo_lambda_minore}, for any $z_1\geq i_1$, 
    \[\int_{z_1}^{y}\mg(u)\lambda(u)du\geq \mathbf{a}(y^{b+1}-z_1^{b+1}).\] As $\Vb(y)=\exp\left(\mathbf{a}y^{b+1}\right)$, $\Vb'(\kappa y)=\mathbf{a} (b+1)\kappa^b y^b\exp\left(\mathbf{a} \kappa^{b+1} y^{b+1}\right)$ and, for any $z_1\geq i_1$, 
\begin{align}
\tilde{I}(z_1)&\leq \int_{z_1}^{\infty} \mathbf{a}(b+1)\kappa^{b+1} y^b\exp\left( -\mathbf{a} y^{b+1}(1-\kappa^{b+1})+\mathbf{a}z_1^{b+1})\right) dy\nonumber\\
&\leq \Vb(z_1)\frac{\kappa^{b+1}}{1-\kappa^{b+1}} \left[- \exp\left(-\mathbf{a} (1-\kappa^{b+1})y^{b+1}\right)\right]_{z_1}^{\infty}\nonumber\\
&\leq \frac{\kappa^{b+1}}{1-\kappa^{b+1}}\Vb(z_1)\Vb(z_1)^{\kappa^{b+1}-1} \label{eq_bound_vb_fin}. 
\end{align}
We have that  \[(i'_1)^{b+1}\geq \frac{1}{\mathbf{a}(1-\kappa^{b+1})}\ln\left(\frac{2\kappa^{b+1}}{1-\kappa^{b+1}}\right)\quad\text{then}\quad 
\left(\Vb(i'_1)\right)^{1-\kappa^{b+1}}\geq  \frac{2\kappa^{b+1}}{1-\kappa^{b+1}}.\]  Therefore, for any $z_1\geq i'_1$, as $\Vb$ is an increasing function, 
\[\Vb(z_1)^{\kappa^{b+1}-1}\leq \Vb(i'_1)^{\kappa^{b+1}-1}\leq \frac{1-\kappa^{b+1}}{2\kappa^{b+1}}.\]
Then 
\begin{equation}\tilde{I}(z_1)\leq \frac{\Vb(z_1)}{2}. 
\end{equation}
Moreover, by \eqref{eq_bound_vb},
\begin{align*}
    \sup_{z_1\leq i_1}\tilde{I}(z_1)
    &\leq \int_0^{\Vb(\kappa i_1)} \exp\left(-\int_0^{\kappa^{-1}\Vb^{-1}(z)} \mg(u)\lambda(u)du\right)dz
    \\
    &+ \int_{\Vb(\kappa i_1)}^{\infty}
    \exp\left(-\int_{0}^{i_1} \mg(u)\lambda(u)du-\int_{i_1}^{\kappa^{-1}\Vb^{-1}(z) }\mg(u)\lambda(u)du \right)
    dz\\
    &\leq  \Vb(\kappa i_1)+ \int_0^{\infty}\exp\left(-\int_{i_1}^{\kappa^{-1}\Vb^{-1}(z)}\mg(u)\lambda(u)du\right)dz\nbOne_{\{\kappa^{-1}\Vb^{-1}(z)\geq i_1\}}\\
    &\leq \Vb(\kappa i_1)+\tilde{I}(i_1)
    \end{align*}
    and by \eqref{eq_bound_vb_fin}, \[\sup_{z_1\leq i'_1} \tilde{I}(z_1)\leq  \Vb(i_1)+ \frac{\kappa^{b+1}}{1-\kappa^{b+1}}\Vb(i'_1)\leq C\Vb(i'_1)<\infty\]

Therefore the three conditions (minorization, strong aperiodicity and drift) are satisfied,  which gives Assumption A\ref{hypo_contraction}. 
\end{description}

\subsubsection{Assumption A\ref{hypo_D_positif} is satisfied}
It remains to prove that Assumption A\ref{hypo_D_positif} is satisfied. 
We recall that \[\Phi_0=\inf_{x\in[0,i_2],y\in\mathcal{I}}(\phi_x^{-1})'(y)\quad \text{and}\quad \Phi_1=\sup_{x\in[0,i_2],y\in\mathcal{I}} (\phi_x^{-1})'(y).\] 
By equation \eqref{def_D}, for any $y\in\mathcal{I}$,
\[\Dg(y)\geq \Phi_0\inf_{ y'\in\mathcal{I}} \int_0^{y'} \Pc{Y_1> y'}{Z_0=x}\mul(dx).\]
By equation \eqref{fonction_repartition_t1},  Assumption (S)\ref{hypo_phi_borne} and \ref{hypo_lambda_borne}, for any $y\in\mathcal{I}$,
\[\Pc{Z_1> y}{Z_0=x}\geq  \exp\left(-\llg-\int_{i_1}^{i_2}\Lg(\phi_x^{-1})'(u)du\right)\geq e^{-\llg-\Lg \Phi_1 i_2}.\]
Then 
\[\inf_{\lambda\in\Fcb} \inf_{y\in\mathcal{I}}\Dg(y)\geq \Phi_0 e^{-\llg-\Lg\Phi_1 i_2}\mul([0,i_1]).\]
It remains to bound $\mul([0,i_1])$ away from 0. 

As $\mul$ is the stationary density of $(Z_k)$, 
$\mul(]z,\infty])=\PP[\mul]{Z_1> z}$.
Therefore, by Markov inequality, as $\Vb$ is an increasing function, 
\[
\mul(]z,\infty[)=\PP[\mul]{\Vb(Z_1)> \Vb(z)}\leq \Vb^{-1}(z)\EE[\mul]{\Vb(Z_1)}.\]
By Lemma \ref{hypo}\ref{Vc}, \[\sup_{\lambda\in\Fcb} \EE[\mul]{\Vb(Z_1)}\leq \Vb(z_0)(1+\gamma R)<\infty.\] 
As   $\sup_{\lambda\in \Fcb}\EE[\mul]{\Vb(Z_1)}<\infty$, and  $\Vb$ is an increasing function,  there exists $y_0>0$, $\sup_{\lambda\in\Fcb}\mul(]y_0,\infty[)<1$ and consequently, $\inf_{\lambda\in\Fcb}\mul([0,y_0])>0$. 
Let us consider the sequence
\[(z_0:=i_1,z_1:=z_0/\sqrt{\kappa},\ldots,z_j:=z_{j-1}/\sqrt{\kappa}=\kappa^{-j/2}i_1,\ldots,z_{k_n}:=\kappa^{-k_n/2}i_1)\] where $z_{k_{n}-1}<y_0\leq z_{k_n}$. We can remark that
\begin{equation}
\inf_{\lambda\in\Fcb} \mul([0,z_{k_n}])>0. \label{eq_nul_kn}
\end{equation}
As $\mul$ is the stationary density, for any $z>0$,
\[
\mul([0,z])= \int_0^{\infty}\Pc{Z_1\leq z}{ Z_0= x} \mul(dx).
\]
As $\PP{Z_1\leq \kappa Y_1}=1$, $\Pc{Z_1\leq z}{Z_0=x}\geq \Pc{Y_1\leq \kappa^{-1}z}{Z_0=x}$ 
and by \eqref{fonction_repartition_t1},
\begin{align*}
    \mul([0,z_{j}])&\geq  \int_0^{\infty}\left(1- \exp\left(-\int_{x}^{\kappa^{-1}z_j} \lambda(u)(\phi_x^{-1})'(u)du\right)\units{x\leq \kappa^{-1}z_j}\right)\mul(dx)\\
    &\geq \int_{i_1}^{z_{j+1}}   \left(1- \exp\left(-\int_{x}^{\kappa^{-1}z_j} \lambda(u)(\phi_x^{-1})'(u)du\right)\right)\mul(dx)
\end{align*}
as $\kappa^{-1}z_j\geq z_{j+1}$. By Assumption (S)\ref{hypo_phi_borne} and \ref{hypo_lambda_minore}, $\lambda$  and $(\phi_x^{-1})'$ are bounded by below and there exists a constant $\eta$ such that 
\[\inf_{\lambda\in\Fcb}\inf_{u\in[i_1,z_{k_n}]} \lambda(u)(\phi_x^{-1})'(u)\geq \eta. \]
Therefore, as $\kappa^{-1}z_j=\kappa^{-1/2}z_{j+1}$,
\begin{align*}
\mul([0,z_{j}])
&\geq \int_{i_1}^{z_{j+1}} \left(1-\exp(-\eta(\kappa^{-1}z_{j}-z_{j+1})\right)\mul(dx)\\
&\geq \left(1-\exp(-\eta(\kappa^{-1}z_{j}(1-\sqrt{\kappa}))\right)\mul([i_1,z_{j+1}]).
\end{align*}
Let us set $c_j=\left(1-\exp(-\eta(\kappa^{-1}z_{j}(1-\sqrt{\kappa}))\right)$. We can note that 
\[\mul([0,z_{j}])\geq c_j\left(\mul([0,z_{j+1}])-\mul([0,i_1])\right)\]
and in particular,
$\mul([0,i_1])(1+c_0)\geq c_0\mul([0,z_1]].$
By recurrence, we obtain: 
\[ \left(1+ \sum_{j=0}^{k_n-1} \prod_{i=0}^j  c_{i} \right) \mul([0,i_1])\geq \left(\prod_{j=0}^{k_n-1} c_j\right) \mul([0,z_{k_n}]).\]
Then by \eqref{eq_nul_kn} \[\inf_{\lambda\in\Fcb} \mul([0,i_1])>0\]
which concludes the proof.

\subsection{Besov and H\"older spaces}\label{beshol}
\begin{definition}[Modulus of continuity]
The modulus of continuity is defined by 
\[\omega(f,t)=\sup_{\vert x-y\vert \leq t} \vert f(x)-f(y)\vert. \]
\end{definition} 
If $f$ is Lipschitz, the modulus of continuity is proportional to $t$. If $\omega(f,t)=o(t)$, then $f$ is constant: the modulus of continuity can not measure higher smoothness. 
\begin{definition}[Modulus of smoothness]
If $f$ is a function on $\mathcal{A}$, we define its modulus of smoothness by 
\[\omega_{\rg}(f,t)_p=\sup_{0<h\leq t} \norm{\Delta_h^{\rg}(f,.)}_{L^p(\mathcal{A})}\quad\textrm{where}\quad \Delta_h^{\rg}(f,x)=\sum_{k=0}^{\rg} (-1)^k C_{\rg}^k f(x+kh).\]
\end{definition}
We can remark that if $f$ is $C^{\rg}$, then \[\lim_{t\rightarrow 0}t^{-{\rg}}\omega_{\rg}(f,t)_{p}=\norm{f^{(\rg)}}_{L^p(\mathcal{A})}\quad \text{and}\quad 
\lim_{t\rightarrow 0} t^{-\alpha}\omega_{\rg+1}(f,t)_p= \lim_{t\rightarrow 0} t^{-\alpha+\rg} \omega(f^{(\rg)},t)_p.\] 
In particular, if  $f\in C^r(\mathcal{A})$ with $\mathcal{A}$ compact and if $f^{({\rg}+1)}$ is Lipschitz, then $\omega_{\rg+1}(f,t)_{p}=O(t^{\rg+1})$. 
If  $f^{(\rg)}$ is $(\alpha-\rg)$-Hölder-continuous, that is if $\forall x,y\in\mathcal{A}$, $|f^{(\rg)}(x)-f^{(\rg)}(y)|\leq C|x-y|^{\alpha-\rg}$, then \[\omega_{\rg+1}(f,t)_p=O(t^{\alpha}).\] 
If $f^{(\rg)}$ is piecewise-continuous and $(\alpha-\rg)$-Hölder on the points of continuity, then \[\omega_{\rg+1}(f,t)_p=O(t^{1/p}+t^{\alpha}).\]

The modulus of continuity and the modulus of smoothness are sub-linears: 
\[\omega_{\rg}(f+g,t)_p\leq \omega_{\rg}(f,t)_p+\omega_{\rg}(g,t)_p\quad \textrm{and}\quad \omega_{\rg}(af,t)_p=a\omega_{\rg}(f,t)_p.\]

\begin{definition}[Besov space]
The Besov space $\Besov_{2,\infty}^{\alpha}(\mathcal{A})$ is the set of functions:
\[\Besov_{2,\infty}^{\alpha}(\mathcal{A})=\{f\in L^2(\mathcal{A}),\; \sup_{t>0} t^{-\alpha}\omega_{\rg+1}(f,t)_2<\infty\}\]
where $\rg=\lfloor\alpha\rfloor$. 
The norm  is defined by: 
$\norm{f}_{B_{2,\infty}^{\alpha}}:=\sup_{t>0}t^{-\alpha}\omega_{\rg+1}(f,t)_2 + \norm{f}_{L^2(\mathcal{A}}$.
We denote
	$\rond{B}_{2,\infty}^{\alpha}(\mathcal{A},M_1)=\{f\in\rond{B}_{2,\infty}^{\alpha}(\mathcal{A}),\norm{f}_{B_{2,\infty}^{\alpha}(\mathcal{A})}\leq M_1\}	$.
\end{definition}

See \citet{devorelorentz} and \citet{meyer} for more details. 
We use the Besov space to control the risk of the estimator of the stationary density $\nul$. 

\begin{definition}[H\"older space]
The H\"older space is the set of functions: 
\[ H^{\alpha}(\mathcal{A})=\{f\in \rond{C}^{\rg}(\mathcal{A}), t^{\rg-\alpha}\omega(f^{(\rg)},t)_{\infty}<\infty\:\; \forall \;\: t>0\}\]
where $\rg=\lfloor\alpha\rfloor$. We note  
$\vert f\vert _{H^{\alpha}(\mathcal{A})}:=\sup_{t>0}t^{\rg-\alpha}\omega(f^{(\rg)},t)_{\infty}
$
and define the norm of the H\"older space $\norm{f}_{H^{\alpha}(\mathcal{A})}=\vert f\vert_{H^{\alpha}(\mathcal{A})}+\norm{f}_{L^{\infty}(\mathcal{A})}$
and 
$H^{\alpha}(\mathcal{A},M_1)=\{f\in H^{\alpha}(\mathcal{A}),\norm{f}_{H^{\alpha}(\mathcal{A})}\leq M_1\}$.
\end{definition}
As noted before, $t^{\rg-\alpha}\omega(f^{(\rg)},t)_{\infty}=t^{-\alpha}\omega_{\rg}(f,t)_{\infty}$: the H\"older space $H^{\alpha}(\mathcal{A})$ is included in $\Besov_{\infty,\infty}^{\alpha}(\mathcal{A})$ which itself is included in $\Besov_{2,\infty}^{\alpha}(\mathcal{A})$.  
We can remark that if a function is $C^{\rg}$ and piecewise $C^{\rg+1}$, it belongs to $\Besov_{2,\infty}^{\rg+1/2}$ but only to $H^{\rg}$.

\subsection{Proof of Lemma \ref{lem_regularite_nu}}

As $\numes$ is the stationary distribution of $(Y_k)$, by \eqref{eq_defi_mes_stat} and \eqref{densitet1}, we have, for all $y\in\mathcal{J}$: 
\begin{align*}
    \nul(y)&=\int_{\R^+\times \R^+} \nul(x)Q(x,dz)\pl(z,y)dx\\
    &=   \int_{\R^+} \nul(x)\int_0^y Q(x,dz)\lambda(y)(\phi_{z}^{-1})'(y)e^{-\int_{z}^{y} \lambda(u)(\phi_{z}^{-1})'(u)du} dx\\
    &=\int_{\R^+} \nul(x)\int_0^y Q(x,dz)\Lambda(z,y)dx
\end{align*}
with \[
     \Lambda(z,y)= \lambda(y)(\phi_{z}^{-1})'(y)e^{-\int_{z}^{y} \lambda(u)(\phi_{z}^{-1})'(u)du}.\]
 As  the Hölder spaces are stables by multiplication, composition and integration, $\Lambda$ has the same regularity than $\lambda$ and  $(\phi_x^{-1})'$. We have that 
\begin{equation*}
    \begin{split}
    \nul(y+h)-\nul(y)=\int_{\R^+} \int_0^{y+h} \nul(x)Q(x,dz)(\Lambda(z,y+h)-\Lambda(z,y))dx \\+\int_{\R^+}\int_y^{y+h}\nul(x)Q(x,dz)\Lambda(z,y)dx.
    \end{split}
    \end{equation*}
 Let us set 
$\tq[g](y)=\int_{\R^+}\int_0^y g(x)Q(x,dz)dx.
$
If $\tq[\nul]$ is differentiable, we get:
\[\nul'(y)= \int_{\R^+} \nul(x)\int_0^y Q(x,dz)\frac{\partial \Lambda}{\partial y}(z,y)dx + \Lambda(y,y)\tq[\nul]'(y)\]
and if $\tq[\nul]$ belongs to $C^{\rg}$, there exist $(c_{k_1,k_2})_{k_1+k_2\leq \rg-1}\in \mathbb{R}$ such that : 
\[\nul^{(\rg)}(y)=\int_{\R^+}\nul(x)\int_0^y Q(x,dz)\frac{\partial^k\Lambda}{\partial y^k}(z,y)dx+\sum_{k_1+k_2\leq \rg-1}c_{k_1,k_2}\frac{\partial^{k_1+k_2}\Lambda}{\partial y^{k_1}z^{k_2}}(y,y)\tq[\nul]^{(\rg-k_1-k_2)}(y).\]
It remains to study the regularity of the function $\tq[\nul]$. 

We  consider some particular transition measures $Q$ in order to understand how the regularity of $\tq[\nul]$ (and $\nul$) depends on the form and the regularity of $Q$. 

\paragraph*{Continuous transition measure}
There exists a function $Q_1$ such that $Q(x,dy)=Q_1(x,y)dy$, and we can write
\[\tq[\nul](y)=\int_{\R^+}\int_{0}^y \nul(x)Q_1(x,z)dzdx\quad \text{and}\quad  \tq[\nul]'(y)=\int_{\R^+} \nul(x)Q_1(x,y)dx.\]
Moreover, as $Q_1(x,y)=0$ if $x<y$, with $\mathcal{I}=[i_1,i_2]$, we get \[\norm{\tq[\nul]}_{L^{\infty}(\mathcal{I})} \leq i_2  \norm{Q_1}_{L^{\infty}([i_1,\infty[\times \mathcal{I})}\int_{\R^+}\nul(x)dx= i_2  \norm{Q_1}_{L^{\infty}([i_1,\infty[\times \mathcal{I})}.\]
Furthermore, by definition of the 
  Hölder semi-norm, for $\rg=\lfloor \alpha\rfloor$ 
\begin{align*}
    |\tq[\nul]|_{\holder(\mathcal{I})}&= \sup_{t>0} t^{\rg-\alpha-1} \sup_{y,y'\in \mathcal{I}, |y-y'|\leq t}
    \left\vert \tq[\nul]^{(\rg)}(y)-\tq[\nul]^{(\rg)}(y')\right\vert \\
&\leq \sup_{t>0} t^{\rg-\alpha-1} \int_{\R^+} \nul(x)\sup_{x\geq y} \sup_{y,y'\in\mathcal{I},|y-y'|\leq t|} \left| \frac{\partial^{\rg-1} Q_1}{\partial y^{\rg-1}} (x,y')-\frac{\partial^{\rg-1} Q_1}{\partial y^{\rg-1}}(x,y)\right|dx \\
&\leq \sup_{t>0} t^{\rg-\alpha-1} \sup_{z\geq  i_1}   \sup_{y,y'\in\mathcal{I},|y-y'|\leq t} \left| \frac{\partial^{\rg-1} Q_1}{\partial y^{\rg-1}}(z,y)-\frac{\partial^{\rg-1} Q_1}{\partial y^{\rg-1}}(z,y')\right|\int_{\R^+} \nul(x)dx\\
&= \sup_{z\geq i_1}|Q_1(z,.)|_{\holder[\alpha-1](\mathcal{I})}.
\end{align*}
Then $\norm{\tq[\nul]}_{\holder(\mathcal{I})}\leq i_2 \norm{ Q_1}_{\holder[\alpha-1]([i_1,\infty[\times \mathcal{I})}$. 

\paragraph{Deterministic transition measure}
Let us assume that $Q$ can be written $Q(x,dy)=\delta_{f(x)}(dy)$ with $f$ a bijection.  As $\PP{Z\leq \kappa Y}=1$,  $f(0)=0$.   Then we have that
\[\nul(y)=\int_0^{f^{-1}(y)}\nul(x) \Lambda(f(x),y)dx\quad\text{and}\quad 
\tq[\nul](y)= \int_0^{f^{-1}(y)} \nul(x) dx.\] If $f^{-1}$ is differentiable:  \[\tq[\nul]'(y)=(f^{-1})'(y)\nul(f^{-1}(y)).\] 
So we get: 
\[\nul'(y)= \int_{0}^{f^{-1}(y)}\nul(x) \frac{\partial \Lambda}{\partial y}(f(x),y)dx + \Lambda(y,y)(f^{-1})'(y)\nul(f^{-1}(y)). \]
The regularity of $\nul'$ on $\mathcal{I}$ depends on the regularity of $\nul$ on $f^{-1}(\mathcal{I})$ and of $\Lambda$ and $f^{-1}$ on $\mathcal{I}$. By recurrence, there exists a function $\psi_2$ such that 
\[\norm{\tq[\nul]}_{\holder(\mathcal{I})}\leq \psi_2\left(\norm{\lambda}_{\holder(\mathcal{J})},\norm{(\phi_.^{-1})'}_{\holder([0,j_2]\times \mathcal{J})}, \norm{f^{-1}}_{\holder(\mathcal{J})} \right)\] where  \[ \mathcal{J}_0=\mathcal{I} \quad ,\quad  \mathcal{J}_{k+1}=\operatorname{Conv}\left(\mathcal{I}\cup \bigcup_{i=1}^j   f_i^{-1}(\mathcal{J}_{k})\right)\quad\text{and}\quad \mathcal{J}= \mathcal{J}_{\lfloor \alpha\rfloor}\cup [i_1,i'_2].\]
If $f$ is not a bijection (and $f(x)\neq 0$), then $\nul$ can be less regular than $\lambda$. Let us consider $f(x)=\lfloor x/2\rfloor$. Then 
\[\nul(y)=
\int_{\R^+} \nul(x)
\sum_{k=0}^{\lfloor y\rfloor} \nbOne_{k=\lfloor x/2\rfloor }\Lambda(k,y)dx=\sum_{k=0}^{\infty}\numes([2k,2k+2])\Lambda(k,y)\nbOne_{k\leq y}\]
Then $\nul$ is a piecewise constant function and is not differentiable. We can remark that \[\tq[\nul](y)=\sum_{k=0}^{\infty} \numes([2k,2k+2])\nbOne_{k\leq y}\]
is not differentiable. 

If $Q(x,dy)=\delta_0(y)$ (which implies that the vectors $(Z_k,Y_k)$ are independent), then
$\nul(y)=\int_{\R^+}\nul(x)\Lambda(0,y)dx=\Lambda(0,y)$ has the same regularity as $\Lambda$. We can remark that $\tq[\nul](y)=\int_{\R^+}\nul(x)=1$ is $C^{\infty}$. 

\paragraph{General case}
Under Assumption (S), 
\[Q(x,dy)=Q_1(x,y)dy+p_0(x)\delta_0(dy)+\sum_{i=1}^{j_Q} p_i(x)\delta_{f_i(x)}(dy)\]
with $(f_i)$ invertible, therefore
\[\tq[\nul](y)=\int_{\R^+}\int_0^y \nul(x)Q_1(x,z)dz+ \int_{\R^+}\nul(x)p_0(x)dx+\sum_{i=1}^{j_Q} \int_{0}^{f_i^{-1}(y)} p_i(x)\nul(x)dx\]
and
\[\tq[\nul]'(y)=\int_{\R^+}\nul(x)\frac{\partial Q_1}{\partial y} (x,y)dx+\sum_{i=1}^{j_Q} p_i(f_i^{-1}(y))\nul(f_i^{-1}(y))(f_i^{-1})(y). \]
Therefore, there exists a function $\psi_2$ such that 
\[\norm{\nul}_{\holder(\mathcal{I})} \leq \psi_2\left(\norm{\Lambda}_{\holder(\mathcal{J})},(\norm{f_i}_{\holder(\mathcal{J})})_{1\leq i\leq j_Q},(\norm{p_i}_{\holder[\alpha-1](\mathcal{J})})_{1\leq i\leq j_Q},\norm{Q_1}_{\holder[\alpha-1](\mathcal{J})}\right).\]
As $\lambda\in \holder(\mathcal{J})$ and $\forall x,  (\phi_x^{-1})'\in \holder(\mathcal{J})$, then 
$\forall x$, $\Lambda(z,.)\in\holder(\mathcal{J})$ and there exists a continuous function $\psi_1$ such that 
\[ \norm{\Lambda(.,.)}_{H^{\alpha}([0,j_2] \times \mathcal{J})}\leq \psi_1\left(\norm{(\phi_.^{-1})'}_{H^{\alpha}([0,j_2]\times  \mathcal{I})}, \norm{\lambda}_{H^{\alpha}(\mathcal{J})}
\right)\] which ends the proof.

\subsection{Proof of Talagrand's inequality for beta-mixing variables}

The following  lemma is very useful to replace weak dependent variables by  variables which are independent by blocks. It is proved by \citet[proof of Proposition 5.1]{viennet97}. 
\begin{lemma}[Berbee's coupling lemma]\label{berbee}

The random variables $\{Y_{k}\}_{k\in\mathbb{N}}$ are exponentially  $\beta$-mixing. 
Let us set $q_n=\lfloor (r+1)\ln(n)/\beta\rfloor$ where $\beta$ characterizes the $\beta$-mixing coefficient (see Definition \ref{def_mixing}). We have that  $\beta(q_n)\leq 1/n^{r+1}$. We set $p_n=n/(2q_n)$. 
There exist random vectors
$(Y_1^ *,\ldots,Y_n^*)$ such that: 
\begin{itemize}
\item $Y_i$ and $Y_i^*$ have same law. 
\item The random vectors $(Y_{2kq_n+1}^*,\ldots,Y_{(2k+1)q_n}^*)_{0\leq k< p_n}$ are independent, as the random vectors\\ $(Y_{(2k+1)q_n+1}^*,\ldots,Y_{(2k+2)q_n}^*)_{0\leq k<p_n}$. 
\item For any integer $k$, $0\leq k\leq 2p_n-1$, $\PP{Y_{kq_n+1},\ldots,Y_{(k+1)q_n})\neq(Y^*_{kq_n+1},\ldots,Y^*_{(k+1)q_n})}\leq \beta_Y (q_n)\leq  n^{-(r+1)}$. 
\end{itemize}
\end{lemma}
Let us set $\Omega^*=\{\omega, \forall k, Y_k=Y_k^*\}$. Then
\[
\PP{\Omega^{*c}}\leq n\beta_Y (q_n)\leq\frac{1}{n^r}.\]

This following inequality comes from Talagrand's inequalities (see \citet[Corollary 2 p354]{birgemassart98}).
\begin{lemma}[Talagrand's inequality]\label{talagrand}

Let  $X_1,\ldots,X_n$ be independent random variables and $S$ a vectorial subspace of finite dimension $D$ satisfying Assumption \ref{hypo_sev}. We denote by  $\rond{F}$  a countable family of $S$. 
Let us set 
\[
F_n(u)=\frac{1}{n}\sum_{k=1}^ n u(X_k)-\EE{u(X_k)}\]with $u\in L^2$. 
If  
\[ \sup_{u\in\rond{F}}\norm{u}_{\infty}\leq M_2,\quad \EE{\sup_{u\in\rond{F}}\vert F_n(u)\vert }\leq H,\quad 
\sup_{u\in\rond{F}}\Var{u(X_k))}\leq V,\]
then
\[
\EE{\sup_{u\in\rond{F}}F_n^2(u)-6H^ 2}_+\leq C\left(\frac{V}{n}\exp\left(-\frac{nH^ 2}{6V}\right)+\frac{M_2^ 2}{n^ 2}\exp\left(-k_2\frac{nH}{M_2}\right)\right)\]
where $C$ is a universal constant and $k_2=(\sqrt{2}-1)/(21\sqrt{2})$. 
\end{lemma}

\begin{proof}[Proof of lemma \ref{talagrand}]
We apply Theorem 1.1 of \citet{kleinrio2005} to the functions $s^i(u)=\frac{u(Y_i)-\EE{u(Y_i)}}{2M_2}$ (notation used in Theorem 1.1 of \citet{kleinrio2005}). We obtain that 
\[\PP{ \sup_{u\in\rond{F}} \vert F_n(u)\vert 
\geq H+x}
\leq \exp\left(-\frac{nx^2}{2(V+4HM_2)+6M_2x}\right).\]
We modify this inequality following Corollary 2 of \citet{birgemassart98}. It gives: 
\[ \PP{ \sup_{u\in\rond{F}} \vert F_n(u)\vert 
\geq (1+\eta) H+x}
\leq \exp\left(-\frac{n}{3}\min\left(\frac{x^2}{2V},\frac{\min(\eta,1)x}
{7M_2}\right)\right).\]
The end of the proof is done in \citet[p222-223]{comtemerlevede2002}. 
\end{proof}

\begin{proof}[Proof of lemma \ref{lem_talagrand_mixing}]
To deduce lemma \ref{lem_talagrand_mixing}, we simply apply the Berbee's coupling lemma to exponential $\beta$-mixing variables, and then the Talagrand's inequality. 
Indeed, by Berbee's coupling lemma, as $Y_k^*$ and $Y_k$ have same law: 
\[I_n(\ft)=\frac{1}{n}\sum_{k=1}^n \ft(Y_k^*)-\EE{\ft(Y_k^*)} + \ft(Y_k)-\ft(Y_k^*).\]
We first bound the second part of the sum 
$I_2(\ft):= \frac{1}{n}\sum_{k=1}^n \ft(Y_k)-\ft(Y_k^*)$.
We have:
\[
I_2^2(\ft)=\frac{1}{n^2}\left(\sum_{k=1}^n (\ft(Y_k)-\ft(Y_k^*))\nbOne_{Y_k\neq Y_{k}^*}\right)^2\leq \frac{4M_2^2}{n^2} \left(\sum_{k=1}^n \nbOne_{\{Y_k\neq Y_{k}^*\}}\right)^2
\]
By Cauchy-Schwartz,
$
I_2^2(\ft) \leq \frac{4M_2^2}{n}\sum_{k=1}^n \nbOne_{\{Y_k\neq Y_k^*\}} $
and by  Berbee's coupling lemma, \\
$
\EE{\sup_{\ft\in\rond{B}} I_2(\ft)}\leq \frac{4M_2^2}{n^2}.
$

Let us now bound the first term 
$I_1(\ft):=\frac{1}{n} \sum_{k=1}^n \ft(Y_k^*)-\EE{\ft(Y_k^*)}$.
We have
\[
I_1=
\frac{1}{p_n}\sum_{j=0}^{p_n-1} u_{\ft}(X_{j,0})-\EE{u_{\ft}(X_{j,0}})+
  \frac{1}{p_n}\sum_{j=0}^{p_n-1} u_{\ft}(X_{j,1})-\EE{u_{\ft}(X_{j,1}})
\]
where  
$X_{j,i}:=\left(Y^*_{2(j+i)q_n+1},\ldots,Y^*_{(2(j+i)+1)q_n}\right)$
and
$ u_{\ft}(x_1,\ldots,x_{q_n}):= \frac{1}{q_n}\sum_{k=1}^{q_n} \ft(x_k)$.
The random variables $X_{j,0}$ are independent, the same can be said for $X_{j,1}$.  Moreover, $\vert X_{j,i}\vert \leq M_2$ and $\Var{X_{j,i}}\leq V$.
Let us set 
\[I_{n,i}^*(s):=\frac{1}{p_n} \sum_{j=0}^{p_n-1} u_{s}(X_{j,i})-\EE{u_s(X_{j,i})}.\]
We have: $I_1(s):=(I_{n,0}^*(s)+I_{n,1}^*(s))/2$.   
Then,
\begin{align*}
\EE{\sup_{s\in\rond{B}} I_1^2(s)-6H^2}_+
&\leq \EE{\sup_{s\in\rond{B} }\frac{1}{4}\left(2(I_{n,0}^*(s))^2+2(I_{n,1}^*(s))^2\right)-6H^2}\\
&\leq \sum_{i=0}^1 \EE{\sup_{s\in\rond{B}} (I^*_{n,i}(s))^2-6H^2}_+.
\end{align*}

As the dimension of $S$ is finite, we can find a countable family $\rond{F}$ dense in $\rond{B}$ and we can then apply the Talagrand's inequality to $I_{n,0}^*$ and $I_{n,1}^*$ which concludes the proof.   

\fussy 
\end{proof}

\section*{Appendix B: Simulations}\label{section_simulation}

For the simulations,  two very classical PDMP processes are considered: the TCP  and the size of a marked bacteria.

\paragraph*{TCP protocol.} 
The transmission control protocol (TCP)   is one of the main data transmission protocol in the Internet. The maximum number of packets that can be sent at time $t_k$ in a round is a random variable $X_{t_k}$. If the transmission is successful, then the maximum number of packets is increased by one:  $X_{t_{k+1}}=X_{t_k}+1$. If the transmission fails, then  $X_{t_{k+1}}=\kappa X_{t_k}$ with  $\kappa\in(0,1)$. A correct scaling of this process leads to a piecewise deterministic Markov process  $(X_t)$ with the characteristics:
\[\phi(x,t)=x+ct,\quad Q(x,\{y\})=\nbOne_{\{y=\kappa x\}},\quad \lambda.\]
Then the function $(\phi_x^{-1})'$ is constant:  $(\phi_x^{-1})'=1/c$. Let us denote by $\Lambda$ a primitive of $\lambda$. By \eqref{fonction_repartition_t1}, we have: 
\[
\mathbb{P}(Y_1> y\vert Z_0=x)=\exp\left(-\frac{1}{c}\left(\Lambda(y)-\Lambda(x)\right)\right)\nbOne_{\{y\geq  x\}}.\]
As $\lambda$ is positive, its primitive is invertible and by a change of variable: 
\[ \mathbb{P}(\Lambda(Y_1)> v\vert \Lambda(Z_0)=u)=\exp\left(-\frac{1}{c}\left(v-u\right)\right)\nbOne_{\{v\geq u\}}.\]
Then $\Lambda(Y_j)\vert \Lambda(Z_{j-1})$ follows an exponential law translated by $\Lambda(Z_{j-1})$ and of parameter $1/c$. 
Therefore, if we can find the inverse of the  function $\Lambda$, we can construct the sequence $(Y_j,Z_j)$ by recurrence: 
\begin{equation}\Lambda(Y_j)=\Lambda(Z_{j-1})+cE_j,\quad Z_j=\kappa Y_j \label{eq_construction_Zk_recurrence}
\end{equation}
where $E_j$ are i.i.d. of law $\rond{E}(1)$.

If $\lambda(x)=\lambda x^{\delta}$ with $\delta>-1$, then 
$Y_{j}^{\delta+1}=Z_{j-1}^{\delta+1}+c(\delta+1)/\lambda E_j$ and we obtain
\[Y_j=\sqrt[\delta+1]{Z_{j-1}^{\delta+1}+\frac{(\delta+1)c}{\lambda}E_j}.\]
This model satisfies  Assumption (S).  
In order to have a model with a non-increasing function $\lambda$, we also consider the function
$\lambda(x)=(x-a)^2+b$ with $a>0$, $b\geq 0$.
In that case, by \eqref{eq_construction_Zk_recurrence},
\[(Y_j-a)^3+3b(Y_j-a)=(Z_{j-1}-a)^3+3b(Z_{j-1}-a)+3cE_j\]
and, by Cardan's formula, this equation has a unique real solution, which is
\[Y_j=a+\frac{\sqrt[3]{Q+\sqrt{4b^3+Q^2}}+\sqrt[3]{Q-\sqrt{4b^3+Q^2}}}{2}
\]
where
$Q=3cE_j+(Z_{j-1}-a)^3+3b(Z_{j-1}-a)$.
This model also satisfies  Assumption (S).

\paragraph*{Bacterial growth.}

We choose randomly a bacteria, and follow its growth, until it divides in two parts more or less equal. Then we choose randomly one of its daughter, and so on.  Between the jumps, the  bacteria grows exponentially. During a jump,
the size of the bacteria is more or less divided by two. We model this by setting $Z_k=Y_k\times U_k$, where $U_k$ is a random variable  independent of $Y_k$, in $(0,1)$, and centered in $1/2$. The Beta distribution $\beta(\alpha,\alpha)$ satisfies these conditions. For $\alpha=1$, it is the uniform distribution, and when $\alpha$ increases, the distribution is more concentrated around $1/2$. We choose $\alpha=20$. 
Then 
\[\phi(x,t)=xe^{ct}, \quad Q(x,y)=\frac{1}{\beta(20,20)}\frac{y^{19}(x-y)^{19}}{x^{38}}\nbOne_{\{y\leq x\}}.\]
 Then $(\phi_x^{-1})'(y)=\frac{1}{y}$ and by \eqref{fonction_repartition_t1}, 
\[\mathbb{P}(Y_1> y\vert Z_0=x)=\exp\left(-\frac{1}{c}\int_{x}^{y} \frac{\lambda(s)}{s}ds\right)\nbOne_{\{y\geq  x\}}.
\]
We need to find a primitive of $\lambda(x)/x$. If $\lambda(x)=\lambda x^{\delta}$, $\delta>0$, then: 
\[\mathbb{P}(Y_1> y\vert Z_0=x)=\exp\left(-\frac{\lambda}{\delta  c}\left(y^{\delta}-x^{\delta}\right)\right)\nbOne_{\{y\geq  x\}}.\]
Therefore
\[\PP{Y_1^{\delta}> y\vert Z_0^{\delta}\geq x}= \exp\left(- \frac{\lambda}{\delta c} (y-x)\right)\nbOne_{\{y\geq x\}}\]
and the law of the random variable $Y_k^{\delta}$ is an exponential translated by $Z_{k-1}^{\delta}$ and of parameter $\lambda/\delta c$. 
Then
\[Y_{k}= \sqrt[\delta]{\frac{\delta c}{\lambda}E_k+Z_{k-1}^{\delta}},\quad Z_k=Y_kU_k\] with $E_k\sim \rond{E}(1)$ i.i.d. and $U_k\sim\beta(20,20)$ i.i.d. 
All the conditions of Assumption (S) are satisfied, except point  \ref{hypo_Q}. Indeed, $\PP{Z<Y}=1$, but there do not exists any $\kappa<1$ such that $\PP{Z\leq \kappa Y}=1$. However,  in the simulations, it seems that the process is ergodic and that A\ref{hypo_D_positif} is satisfied.  

\paragraph*{Computations}
For the two models, $\nul$ has a density with respect to the Lebesgue measure on $\mathbb{R}$, so it can be estimated on any compact interval $\mathcal{A}$, here $\mathcal{A}=[-1,5]$ to avoid edge effects. The estimator is computed thanks to a projection on a trigonometric basis. The constant involved in $pen(m)$, $\textit{cpen}$, should be greater than $\frac{3}{2}\left(\psi_1+\psi_2 C_{\lambda}\right)$, with $\psi_1=\psi_2=\frac{1}{3}$. The problem is that $C_{\lambda}$, a correlation term, is not easily tractable. 
We set \textit{cpen}=$\psi_1+\psi_2=2/3$ for all models.    This choice seems confirmed by the simulations results: the oracle $or$ remains close to 1. 

The constant \textit{cpen} could be determined via the slope heuristic. Indeed, if the constant in the penalty is too small, the algorithm selects the maximal dimension. If the penalty is large enough, it selects  models of reasonable size. We then let the constant $c$ in the penalty vary and note the dimension selected. For $c$ smaller than a value $c_{min}$,  the largest models are selected, and for $c$ greater than $c_{min}$,  smaller models are chosen. The "best" constant is $c=2c_{min}$. See \citet{arlotmassart2009} for instance. 
 
 Figure \ref{fig_choix_dim} shows the
 selected dimension with respect to \textit{cpen}, the constant in the penalty. 
When the constant in the penalty increases, the chosen dimension first decreases very rapidly, until  \textit{cpen}=$0.24$, then it decreases very slowly towards 1. Then  $2c_{min}=0.48$. Our chosen penalty constant, $2/3$, is a little greater than $2c_{min}$, and selects the same dimension (here 17).

\begin{figure}[h]
\caption{Choice of the dimension}\label{fig_choix_dim}
\includegraphics[width=0.5\linewidth]{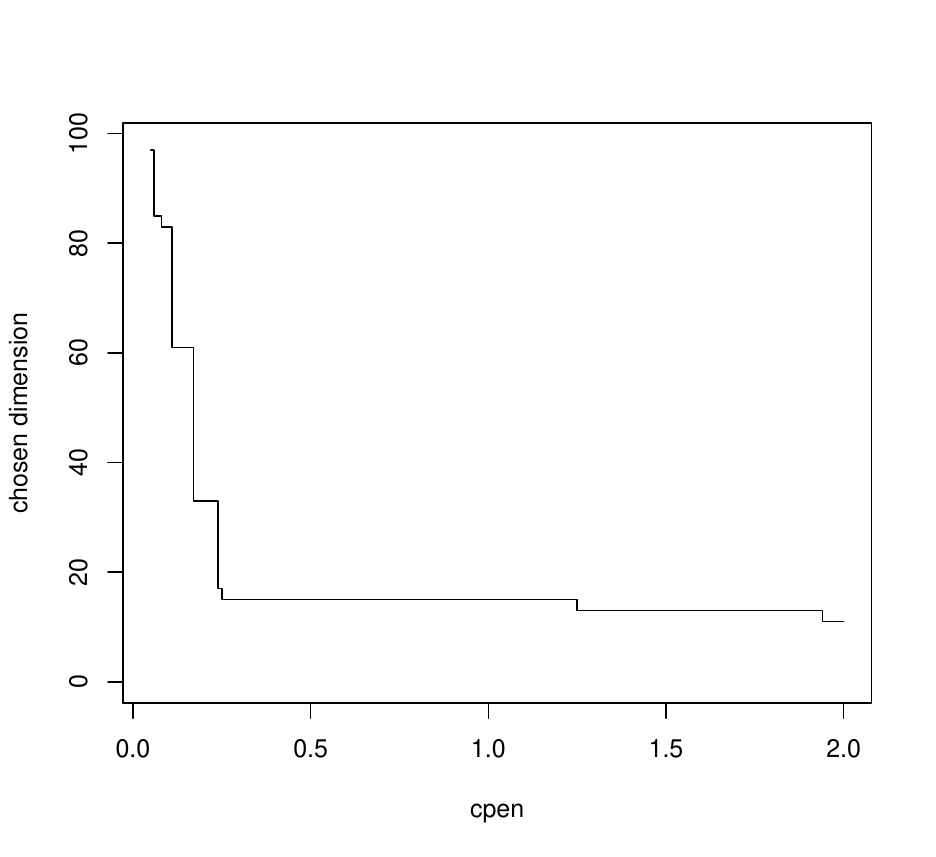}

\end{figure}

 However, the slope heuristic involves quite a lot of computations, so it can not be used for every simulation, only to check that the penalty constant is coherent.

In Figures \ref{figure_TCP} and \ref{figure_bacteria}, for each graph,  five simulations of the PDMP with $n=10^5$ are realized. For each simulation,  the estimator $\hat{\lambda}$, the density $\hnuml[\hat{m}]$ and $\hat{\Dg}_n$ are drawn.

In the tables, 200 simulations for each 4-tuple $(n,c,\kappa,\lambda)$ are computed. The estimation interval $\mathcal{I}=[0.5,2]$  is such that $\Dg$ is greater than the threshold $(\ln(n))^{-1}$ on $\mathcal{I}$ for $n=10^{-5}$ for all our models. 
For each set of parameters,  the mean of the selected dimension $\hat{D}_m$,    the mean and the standard variation of the $L^2$ error on  $\mathcal{I}$,  denoted by "risk" and "sd" are calculated.
We also want to prove that our estimator is truly adaptive. As $\nul$ is unknown, we can not check that $\hat{m}$ is the better choice for estimating $\nul$.  Instead, let us consider the estimator 
\[\hat{\lambda}_m=\frac{\hnuml}{\Dge} \nbOne_{\{\hnuml\geq 0\}}\nbOne_{\{\Dge \geq (\ln(n))^{-1}\}}.\]
Then  $\hat{\lambda}_n=\hat{\lambda}_{\hat{m}}.$
The optimal dimension is 
\[m_{opt}=\min_{m\in\rond{M}_n} \norm{\hat{\lambda}_m-\lambda}_{L^2(\mathcal{I})}^2\]
and the minimal risk $\norm{\hat{\lambda}_{m_{opt}}-\lambda}_{L^2(\mathcal{I})}^2$. 
In the tables, we give the empirical means of $D_{\hat{m}}$, $D_{m_{opt}}$, the empirical mean and standard deviation of the risk and the empirical mean of the oracle
\[ or:= \textrm{mean}\left( \norm{\hat{\lambda}_{\hat{m}}-\lambda}_{L^2(\mathcal{I})}^2/\norm{\hat{\lambda}_{m_{opt}}-\lambda}_{L^2(\mathcal{I})}^2\right).\]

In Figure \ref{fig_cv}, four simulations are realized, each for a different value of $n$ ($n=10^2$, $10^3$, $10^4$ and $10^5$) in order to show the convergence of our estimator.

\paragraph*{Results}

 In Figures \ref{figure_TCP}-\ref{figure_bacteria}, the estimator $\hat{\lambda}$ is very close to $\lambda$, at least when $x$ is neither too small nor too large, that is when there are enough values to compute the estimator.  The estimators  $\hnuml[\hat{m}]$ and $\hat{\Dg}_n$ are quite smooth, whereas $\hat{\lambda}$ tends to oscillate. This is due to the division of two estimators.
In Tables \ref{figure_TCP}-\ref{figure_bacteria}, the risk decreases when  $n$ increases and seems to tend toward 0. The oracle remains close to 1, our estimator is really adaptive. 
When the number of observations is small, the risk may seem quite important (for instance, for figure \ref{figure_bacteria} when $\lambda(x)=x^2$). This is simply because $\Dg$ is smaller than the threshold ($1/\ln(10^2)=0.2$), and the estimator $\hat{\lambda}$ is set to 0 on some part of $\mathcal{I}$, or even on the whole interval.  The estimation near 0 can be good for some models, for instance when $\kappa=1/5$ and $\lambda(x)=x$, because the random variables $Z_k$ take smaller values (at a jump, we divide the process by 5 instead of by 2). The function $\mathbf{D}$ then take higher values near 0, and the estimator $\hat{\lambda}$ is positive even for small values of $x$. 
This problem is illustrated in Figure \ref{fig_cv}: when $n$ increases, the estimator is better both because the support interval of $\hat{\lambda}$ increases and because on the support interval, the estimator is closer to the true function.

\begin{figure}[p]
\begin{small}
\caption{TCP protocol: $\phi(x,t)=x+t$, $Q(x,y)=\delta_{y=\kappa x}$} \label{figure_TCP}

\begin{minipage}[c]{0.45\linewidth}
\includegraphics[height=3.5cm,width=\linewidth]{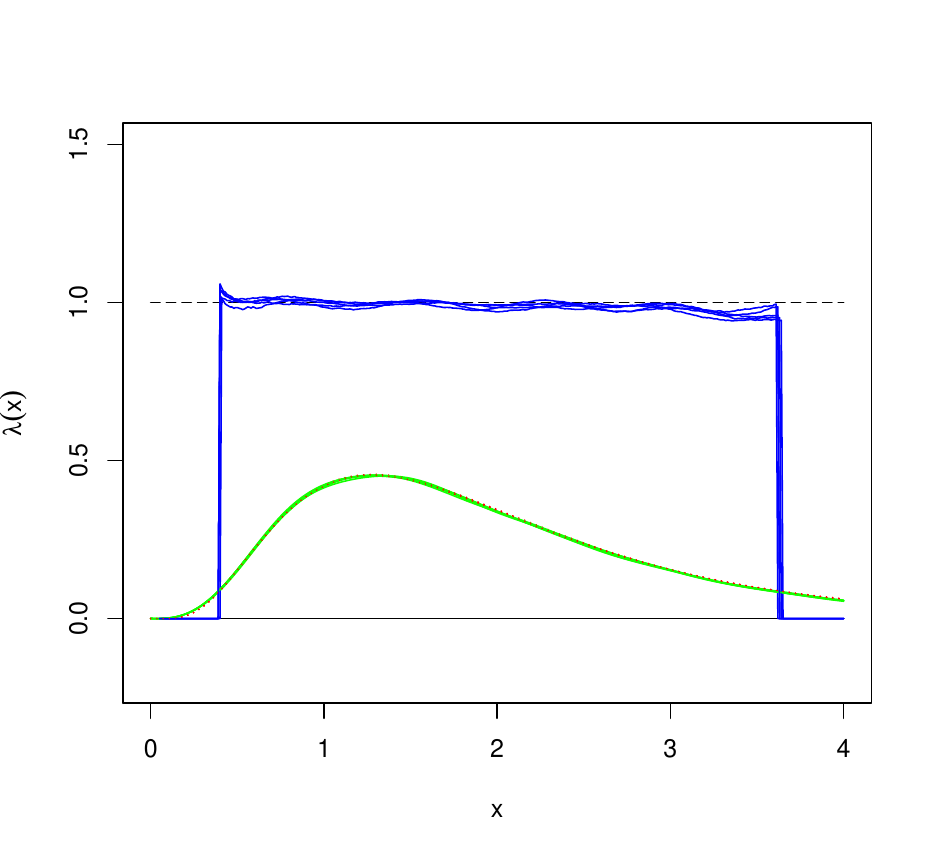}
\end{minipage}
\hfill
\begin{minipage}[c]{0.54\linewidth}
\[\kappa=1/2,\lambda(x)=1,\mathcal{I}=[0.5,2]\]
\begin{tabular}{|c|n{2}{1}n{2}{1}|c c n{1}{2}|}
\hline
$n$& {$D_{\hat{m}}$}&{$D_{m_{opt}}$} & risk & sd& {or}\\
\hline
$10^2$ & 5.1 & 6.7 & $0.12$ & $0.06$ & 1.27\\ 
$10^3$& 8.3 & 10.1 & $9.6$\texttt{e-}$3$ & $9.6$\texttt{e-}$3$ & 1.72\\ 
$10^4$& 12.3 & 14.4 & $8.5$\texttt{e-}$4$ & $5.5$\texttt{e-}$4$& 1.64\\ 
$10^5$& 17.1 & 18.6 & $1.2$\texttt{e-}$4$ & 6.6\texttt{e-}$5$ & 1.47\\
\hline
\end{tabular}

\end{minipage}

\begin{minipage}[c]{0.45\linewidth}
\includegraphics[height=3.5cm,width=\linewidth]{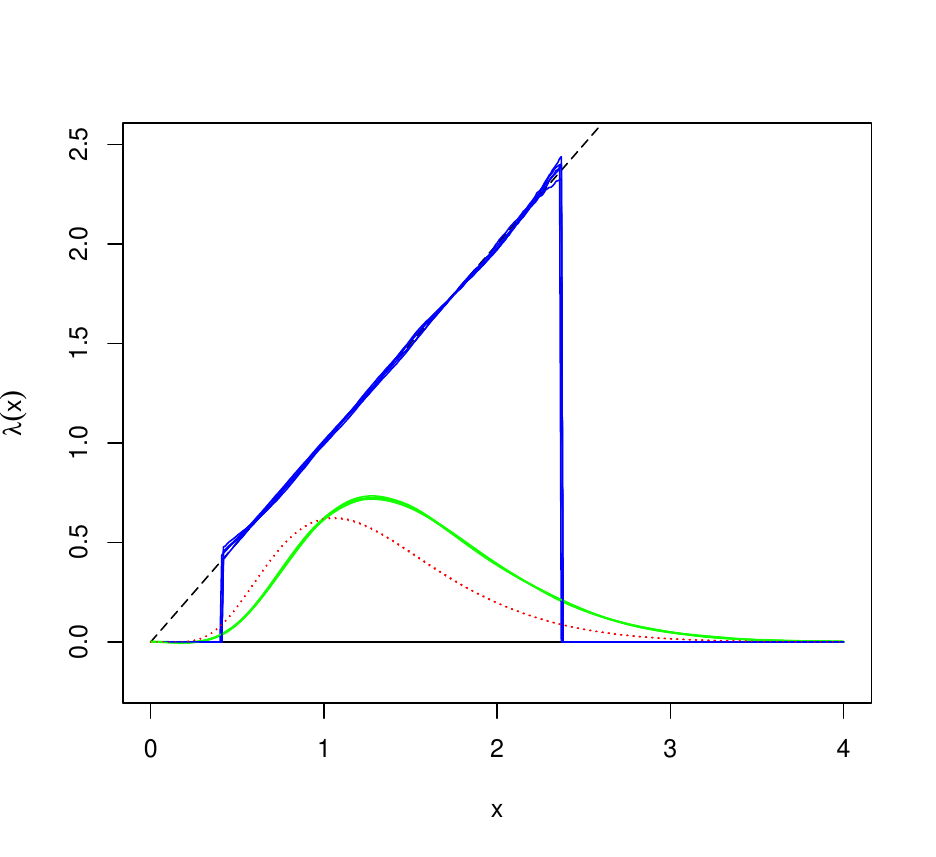}
\end{minipage}
\hfill
\begin{minipage}[c]{0.54\linewidth}
\[\kappa=1/2,\lambda(x)=x,\mathcal{I}=[0.5,2]\] 
\begin{tabular}{|c|n{2}{1}n{2}{1}|ccn{1}{2}|}
\hline
$n$& {$D_{\hat{m}}$}& {$D_{m_{opt}}$}& risk & sd&{or}\\
\hline
$10^2$& 7.1 & 8.5 & $0.24$ &$0.20$ & 1.16\\ 
$10^3$& 10.4&12.8& $9.5$\texttt{e-}$3$&$6.5$\texttt{e-}$3$& 1.64\\ 
$10^4$& 14.1 & 16.4 & $1.1$\texttt{e-}$3$ & $7.3$\texttt{e-}$4$ & 1.49\\ 
$10^5$& 18.1& 20.5& $1.2$\texttt{e-}$4$ & $6.9$\texttt{e-}$05$ & 1.31\\
\hline
\end{tabular}\\
\end{minipage}

\begin{minipage}[c]{0.45\linewidth}
\includegraphics[height=3.5cm,width=\linewidth]{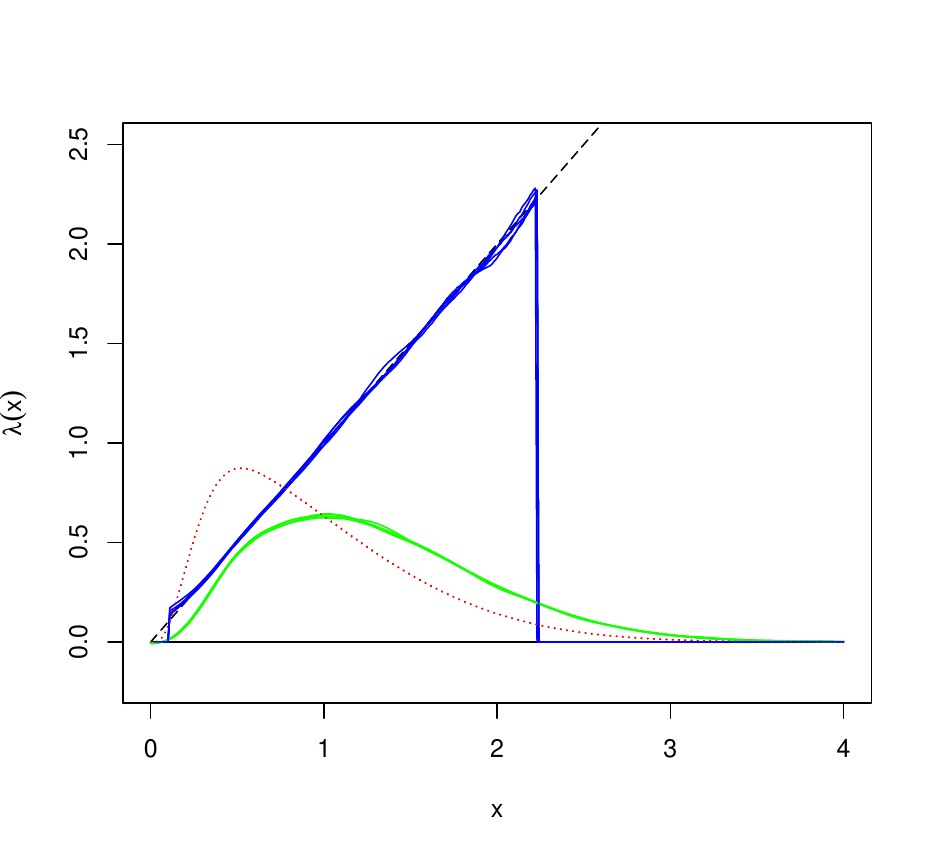}
\end{minipage}
\hfill
\begin{minipage}[c]{0.54\linewidth}
\[\kappa=1/5,\lambda(x)=x,\mathcal{I}=[0.5,2]\] 
\begin{tabular}{|c|n{2}{1}n{2}{1}|ccn{1}{2}|}
\hline
$n$& {$D_{\hat{m}}$}&{$D_{m_{opt}}$}& risk &sd&{or}\\
\hline
$10^2$& 6.6& 6.7 & $0.54$& $0.24$&1.08\\ 
$10^3$& 10.9& 10.7& $0.077$& $0.083$ & 1.40\\ 
$10^4$& 18.0& 15.9& $1.5\texttt{e-}3$& $1.0\texttt{e-}3$ & 2.17\\ 
$10^5$& 25.9 & 22.4 & $2.0\texttt{e-}4$& $1.2\texttt{e-}4$ & 1.81\\
\hline
\end{tabular}
\end{minipage}

\begin{minipage}[c]{0.45\linewidth}
\includegraphics[height=3.5cm,width=\linewidth]{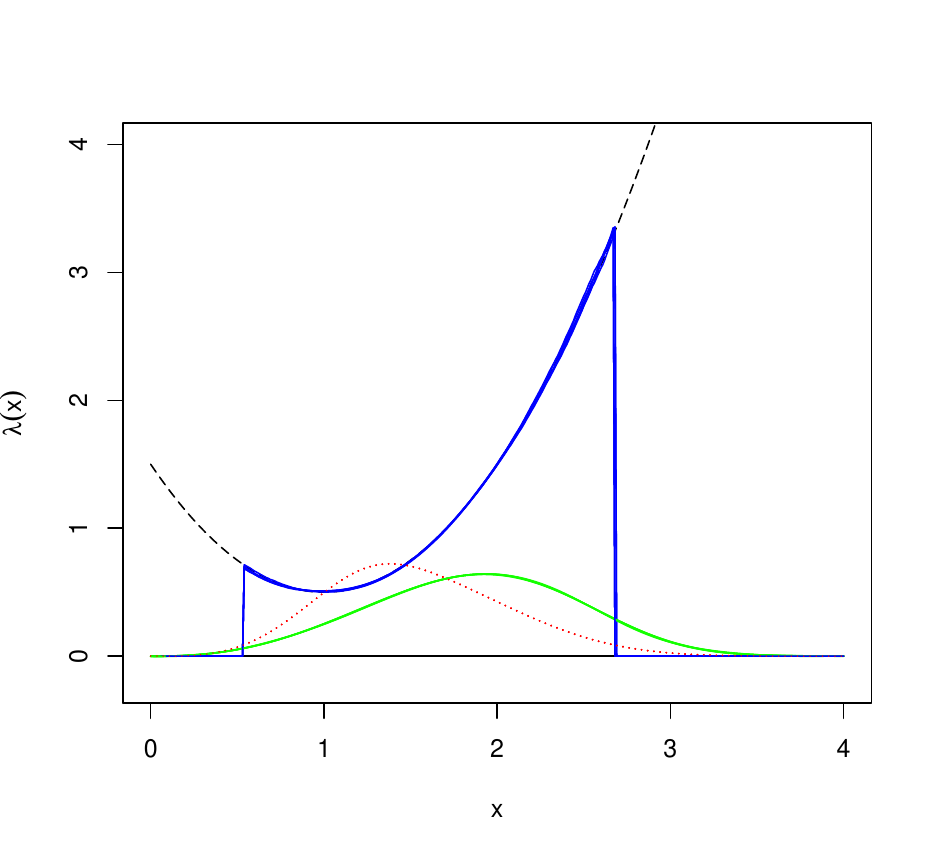}
\end{minipage}
\hfill
\begin{minipage}[c]{0.54\linewidth}
\[\kappa=1/2,\lambda(x)=(x-1)^2+1/2,\mathcal{I}=[0.5,2]\] 
\begin{tabular}{|c|n{2}{1}n{2}{1}|ccn{1}{3}|}
\hline
$n$& {$D_{\hat{m}}$}&{$D_{m_{opt}}$}& risk & sd& {or}\\
\hline
$10^2$& 5.4& 6.6& $0.080$ & $0.014$ & 1.06\\
$10^3$& 7.3& 8.9& $0.045$& $5.4\texttt{e-}3$& 1.01\\ 
$10^4$& 9.3 &11.3& $0.027$ & $2.3\texttt{e-}3$& 1.004\\ 
$10^5$&$11.7$& 16.3 & $0.014$ & $5.9\texttt{e-}4$& 1.003 \\
\hline
\end{tabular}
\end{minipage}
\textcolor{black}{- - : true $\lambda$} $\quad$
\textcolor{blue}{-- : estimated $\hat{\lambda}$} $\quad$
\textcolor{red}{. . : estimated $D_n$} $\quad$
\textcolor{green}{-- : estimated $\hnuml[\hat{m}]$}
\end{small}
\end{figure}

\begin{figure}[p]
\caption{Bacterial growth}\label{figure_bacteria}
\[\phi(x,t)=x \exp(ct), \;\lambda(x)=x^{\delta},\; Q(x,y)=\beta(20,20)\frac{y^{19}(x-y)^{19}}{y^{39}}\units{y\leq x}\]
\begin{minipage}[c]{0.45\linewidth}
\includegraphics[height=3.5cm,width=\linewidth]{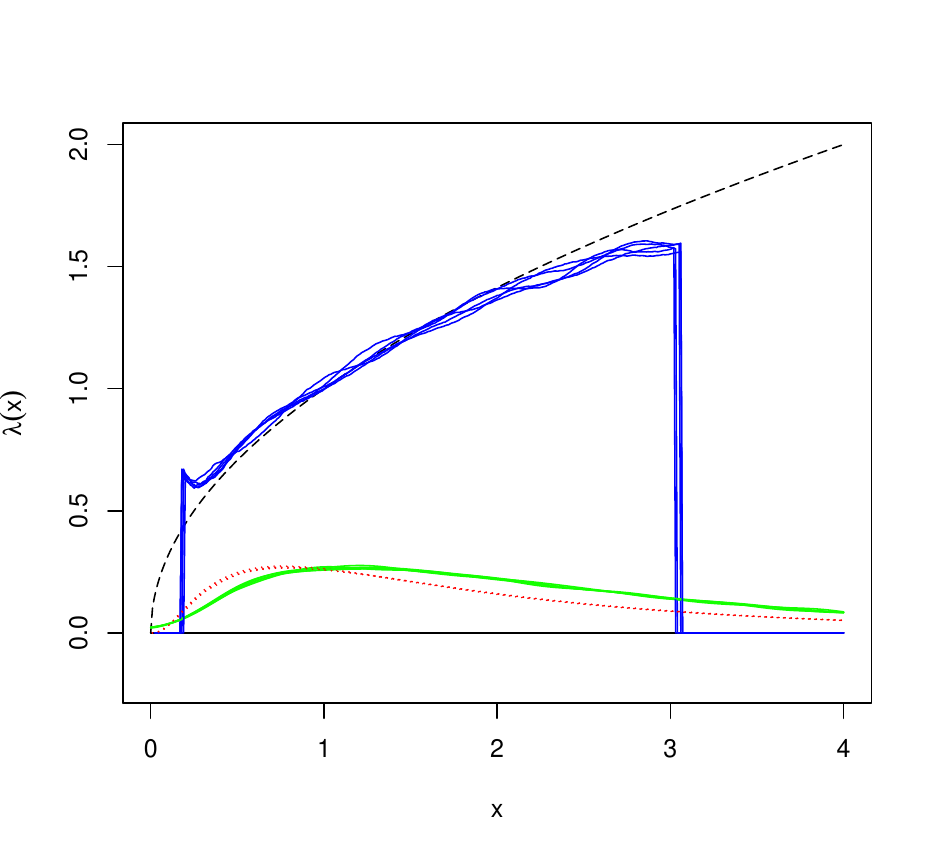}
\end{minipage}
\hfill
\begin{minipage}[c]{0.54\linewidth}
\[\lambda(x)=\sqrt{x}, c=1, \mathcal{I}=[0.5,2]\]
\begin{tabular}{|c|n{2}{1}n{2}{1}|ccn{1}{2}|}
\hline
$n$& {$D_{\hat{m}}$}&{$D_{m_{opt}}$}& risk & sd & {or}\\
\hline
$10^2$&3.8&6.1&$0.72$&$0.29$&1.04\\ 
$10^3$&6.6&7.8&$0.014$&$0.021$&2.32\\ 
$10^4$ & 11.2&10.7&$2.4\texttt{e-}3$&$1.5\texttt{e-}3$&1.97\\ 
    $10^5$&17.7&16.9&$9.0\texttt{e-}4$&$3.1\texttt{e-}4$&1.29\\
\hline
\end{tabular}
\end{minipage}

\begin{minipage}[c]{0.45\linewidth}
\includegraphics[height=3.5cm,width=\linewidth]{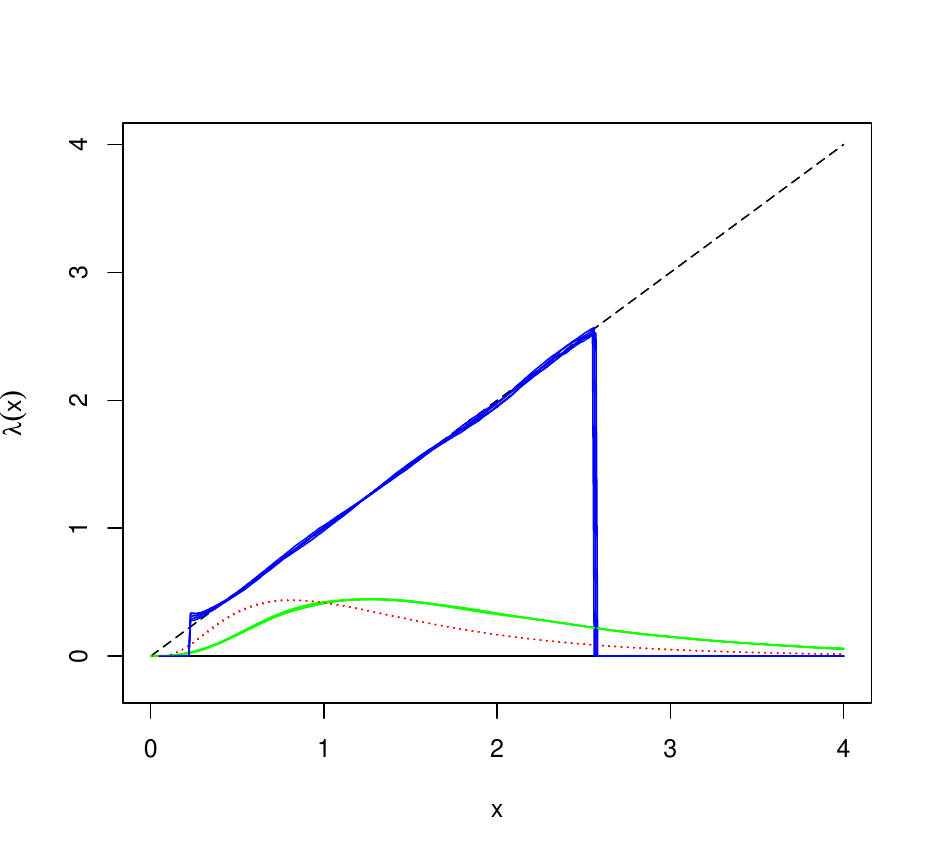}
\end{minipage}
\hfill
\begin{minipage}[c]{0.54\linewidth}
\[\lambda(x)=x,c=1,\mathcal{I}=[0.5,2]\]
\begin{tabular}{|c|n{2}{1}n{2}{1}|ccn{1}{2}|}
\hline
$n$&{$D_{\hat{m}}$}&{$D_{m_{opt}}$}& risk& sd & {or}\\
\hline
$10^2$&5.2&6.3&$0.59$&$0.26$&1.07\\ 
$10^3$&8.2&9.1&$9.2\texttt{e-}3$&$9.0\texttt{e-}3$&2.05\\ $10^4$&12.0&13.2&$1.2\texttt{e-}3$&$8.1\texttt{e-}4$&1.60\\ $10^5$&17.3&17.1&$2.2\texttt{e-}4$&$1.4\texttt{e-}4$&1.54 \\
\hline
\end{tabular}
\end{minipage}

\begin{minipage}[c]{0.45\linewidth}
\includegraphics[height=3.5cm,width=\linewidth]{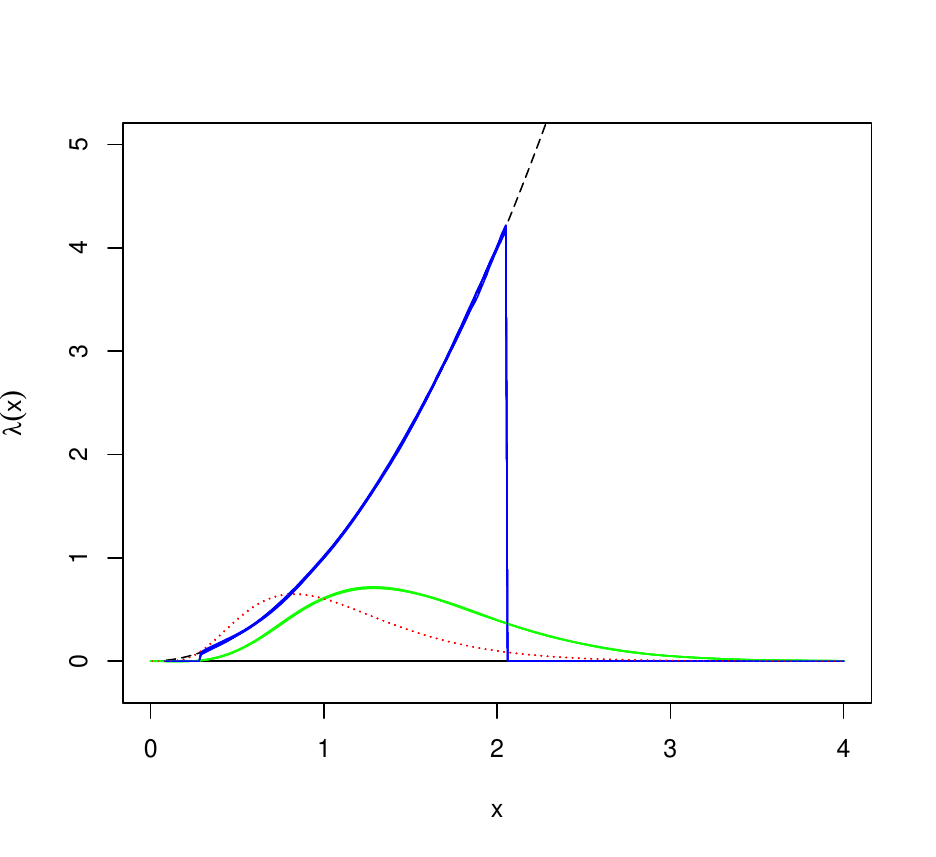}
\end{minipage}
\hfill
\begin{minipage}[c]{0.54\linewidth}
\[\lambda(x)=x^2, c=1, \mathcal{I}=[0.5,2]\]
\begin{tabular}{|c|n{2}{1}n{2}{1}|ccn{1}{3}|}
\hline
$n$& {$D_{\hat{m}}$}&{$D_{m_{opt}}$}& risk & sd&{or}\\
\hline
$10^2$&7.1&7.9&$2.64$&$0.35$&1.01\\ $10^3$&10.1&11.7&$1.48$&$0.21$&1.003\\ $10^4$&13.5&14.2&$0.41$&$0.094$&1.002\\
$10^5$&17.4&17.5&$2.8\texttt{e-}4$&$2.5\texttt{e-}4$&1.62 \\
\hline
\end{tabular}
\end{minipage}

\begin{minipage}[c]{0.45\linewidth}
\includegraphics[height=3.5cm,width=\linewidth]{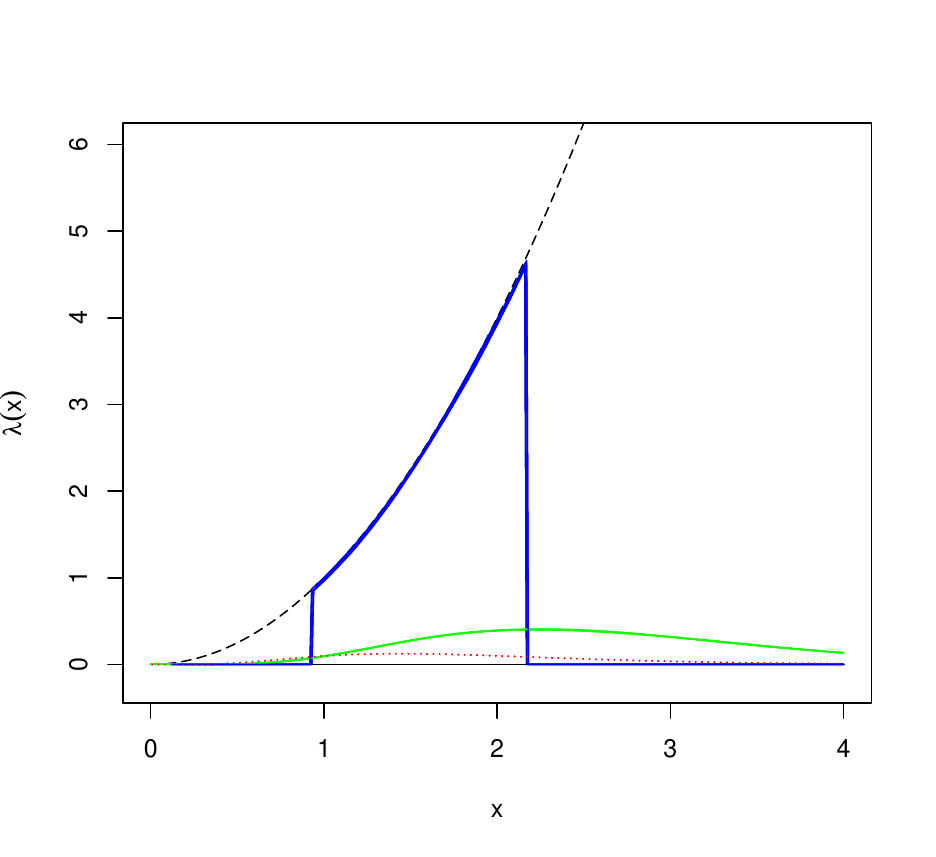}
\end{minipage}
\hfill
\begin{minipage}[c]{0.54\linewidth}
\[\lambda(x)=x^2, c=3, \mathcal{I}=[0.5,2]\]
\begin{tabular}{|c|n{2}{1}n{2}{1}|ccn{1}{3}|}
\hline
$n$& {$D_{\hat{m}}$}&{$D_{m_{opt}}$}& risk & sd&{or}\\
\hline
$10^2$&3.8&3.0&$4.29$&0&1\\
$10^3$&5.7&3.0&$4.29$&0&1\\ 
$10^4$&7.8&11.1&$1.56$&$0.11$&1.001\\
$10^5$&10.0&13.9&$0.087$&$0.0026$&1.001\\
\hline
\end{tabular}
\end{minipage}

\textcolor{black}{- - : true $\lambda$} $\quad$
\textcolor{blue}{-- : estimated $\hat{\lambda}$} $\quad$
\textcolor{red}{. . : estimated $D_n$}$\quad$
\textcolor{green}{--: estimated $\hnuml[\hat{m}]$}
\end{figure}

\begin{figure}
\caption{Convergence of the estimator (size of a marked bacteria)}\label{fig_cv}

\begin{tabular}{cc}
$\phi(x,t)=xe^{t},\quad \lambda(x)=x^2$
&$\phi(x,t)=xe^{3t},\quad,\lambda(x)=x^2$\\
\includegraphics[width=0.45\linewidth]{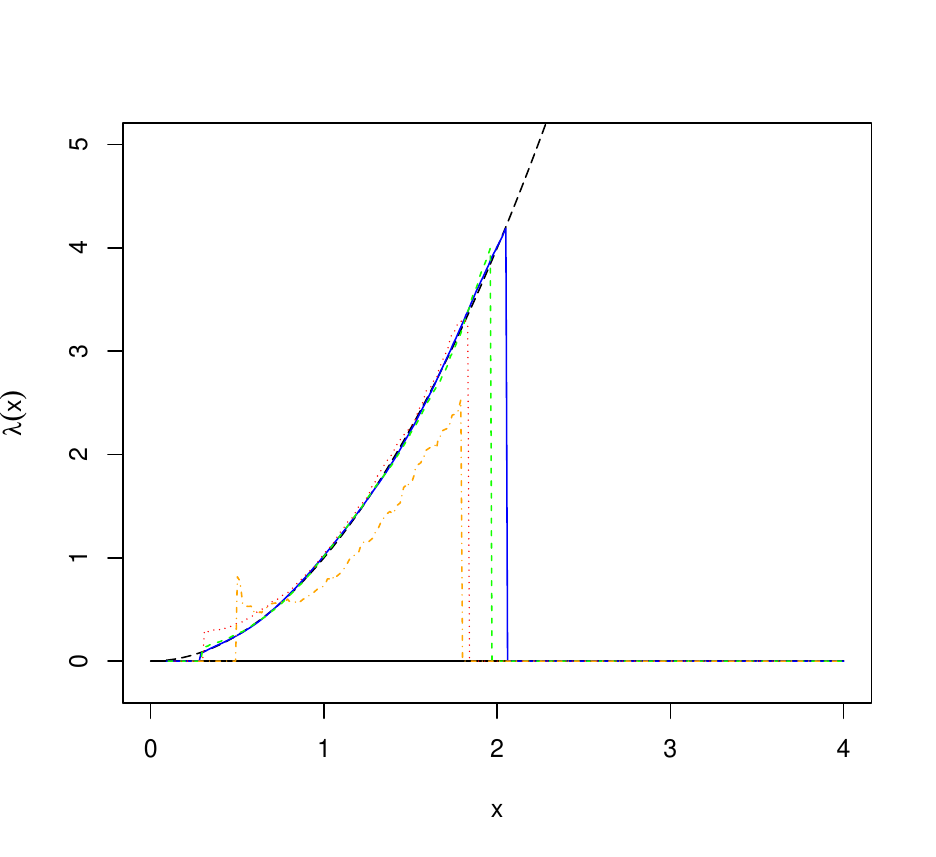}&

\includegraphics[width=0.45\linewidth]{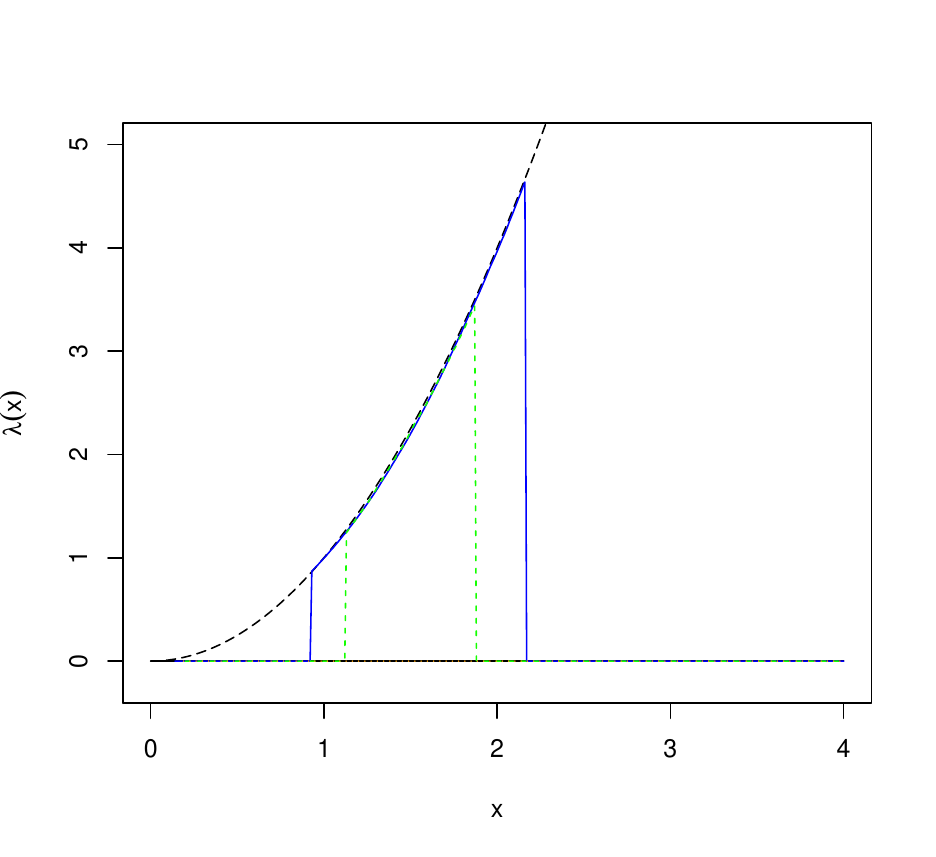}\\
Estimator's support increases with $n$& The estimator is null if $n\leq 10^3$
\end{tabular}
\textcolor{blue}{$\--$} $n=10^5$, \textcolor{green}{- -} $n=10^4$, \textcolor{red}{$\ldots$} $n=10^3$, \textcolor{orange}{$- .$} $n=10^2$. 
\end{figure}

\end{document}